\documentclass[reqno]{amsart}
\usepackage{amssymb, amsfonts}
\usepackage{enumerate}
\usepackage[usenames, dvipsnames]{color}
\usepackage{hyperref}

\usepackage{verbatim}

\numberwithin{equation}{section}

\newtheorem{theorem}{Theorem}[section]
\newtheorem{corollary}[theorem]{Corollary}
\newtheorem{lemma}[theorem]{Lemma}
\newtheorem{proposition}[theorem]{Proposition}

\newtheorem*{AssumpA'}{Assumption A$'$ $(\rho,\varepsilon)$}
\newtheorem*{AssumpA''}{Assumption A$''$ $(\rho,\varepsilon)$}

\theoremstyle{definition}
\newtheorem{remark}[theorem]{Remark}

\theoremstyle{definition}

\theoremstyle{definition}
\newtheorem{assumption}[theorem]{Assumption}

\makeatletter
\def\dashint{\operatorname%
{\,\,\text{\bf--}\kern-.98em\DOTSI\intop\ilimits@\!\!}}
\makeatother

\def\bC{\mathbb{C}}

\def\bH{\mathbb{H}}

\def\bR{\mathbb{R}}
\def\bZ{\mathbb{Z}}

\def\fL{\mathfrak{L}}

\def\cB{\mathcal{B}}
\def\cC{\mathcal{C}}
\def\cD{\mathcal{D}}

\def\cH{\mathcal{H}}

\def\cL{\mathcal{L}}
\def\cM{\mathcal{M}}

\def\cQ{\mathcal{Q}}
\def\cR{\mathcal{R}}

\def\cX{\mathcal{X}}

\begin{document}

\title[$L_p$-estimates for Elliptic and parabolic equations]{On $L_p$-estimates for elliptic and parabolic equations with $A_p$ weights}

\author[H. Dong]{Hongjie Dong}
\address[H. Dong]{Division of Applied Mathematics, Brown University, 182 George Street, Providence, RI 02912, USA}

\email{Hongjie\_Dong@brown.edu}

\thanks{H. Dong was partially supported by the NSF under agreement DMS-1056737.}

\author[D. Kim]{Doyoon Kim}
\address[D. Kim]{Department of Mathematics, Korea University, 145 Anam-ro, Seongbuk-gu, Seoul, 02841, Republic of Korea}

\email{doyoon\_kim@korea.ac.kr}

\thanks{D. Kim was supported by Basic Science Research Program through the National Research Foundation of Korea (NRF) funded by the Ministry of Education (2014R1A1A2054865).}



\begin{abstract}
We prove generalized Fefferman-Stein type theorems on sharp functions with $A_p$ weights in spaces of homogeneous type with either finite or infinite underlying measure. We then apply these results to establish mixed-norm weighted $L_p$-estimates for elliptic and parabolic equations/systems with (partially) BMO coefficients in regular or irregular domains.
\end{abstract}

\maketitle

\section{Introduction}

The objective of this paper is two-fold.
The first is to present a few generalized versions of the Fefferman-Stein theorem on sharp functions.
One of our main theorems in this direction proves
\begin{equation}
							\label{eq0807_01}
\|f\|_{L_{p,q}(\cX, w \, d\mu)} \le N \|f^{\#}_{\text{dy}}\|_{L_{p,q}(\cX, w \, d\mu)},
\quad p, q \in (1, \infty),
\end{equation}
where $(\cX,\mu)$ is a space of homogeneous type, $w$ is a Muckenhoupt weight, $f^{\#}_{\text{dy}}$ is the sharp function of $f$ using a dyadic filtration of partitions of $\cX$,  and $L_{p,q}(\cX, w \, d\mu)$ is a mixed norm.
A space $\cX$ of homogeneous type is equipped with a quasi-distance and a doubling measure $\mu$.
A brief description of spaces of homogeneous type is given in Section \ref{sec1}.
For more discussions, see \cite{MR546295, MR546296, MR1096400}.
If $\cX$ is the product of two spaces $(\cX_1,\mu_1)$ and $(\cX_2,\mu_2)$ of homogeneous type, $w = w_1(x')w_2(x'')$,
and $\mu (x) = \mu_1(x') \mu_2(x'')$, where $x' \in \cX_1$ and $x'' \in \cX_2$,
then the $L_{p,q}(\cX, w \, d\mu)$ norm is defined as
$$
\|f\|_{L_{p,q}(\cX, w \, d\mu)} = \left(\int_{\cX_2} \left( \int_{\cX_1} |f|^p w_1(x') \, \mu_1(dx') \right)^{q/p} w_2(x'') \, \mu_2(dx'') \right)^{1/q}.
$$
See \eqref{eq0806_02} for a precise definition of the mixed norm.
If $\cX$ is the Euclidean space $\bR^d$, $\mu$ is the Lebesgue measure on $\bR^d$, $w \equiv 1$, and $p=q$, the inequality \eqref{eq0807_01} is the celebrated Fefferman-Stein theorem on sharp functions.
See, for instance, \cite{MR0447953, MR1232192},
where the measure of the underlying space is clearly infinite.
Here we deal with both the cases of a finite measure $\mu(\cX) < \infty$ and an infinite measure $\mu(\cX) = \infty$.
We also present a different form of the Fefferman-Stein theorem.
See Theorems \ref{thm2} and \ref{thm3} and Corollaries \ref{cor1} and \ref{cor2}.

The second objective, which is in fact a motivation of writing this paper, is to apply the generalized versions of the Fefferman-Stein theorem to establishing weighted $L_{p,q}$-estimates for elliptic and parabolic equations/systems.
For instance, for the second order non-divergence type parabolic equation
$$
-u_t + a^{ij} D_{ij} u + b^i D_i u + c u = f
$$
in $\bR^{d+1}$,
we prove the a priori weighted $L_{p,q}$-estimate
$$
\| u \|_{W_{p,q,w}^{1,2}(\bR^{d+1})} \le N \|f\|_{L_{p,q,w}(\bR^{d+1})},
\quad p, q \in (1, \infty),
$$
where
$$
\|u\|_{W_{p,q,w}^{1,2}} = \|u_t\|_{L_{p,q,w}} + \|u\|_{L_{p,q,w}} + \|Du\|_{L_{p,q,w}} + \|D^2 u\|_{L_{p,q,w}}.
$$
In this case the coefficients $a^{ij}(t,x)$ are allowed to be very rough so that they have no regularity assumptions in one spatial variable and have small mean oscillations in the remaining variables.
We also treat higher order non-divergence type elliptic and parabolic systems, and higher order divergence type elliptic and parabolic systems in Reifenberg domains.
See Theorems \ref{thm0607_1}, \ref{thm4.4}, \ref{thm0807_1}, \ref{thm0522_1}, \ref{thm5.4}, \ref{thm0625_1}, and the assumptions of the theorems on coefficients. While we focus on the a priori estimates, in Section \ref{sec8} we illustrate how to derive from them the corresponding existence results. We do not employ the usual method of continuity because we were not able to find in the literature the solvability even for simple equations in mixed-norm weighted spaces.

Regarding the Fefferman-Stein theorem on sharp functions, there has been considerable study on its generalizations and applications.
For instance, see \cite{MR807149, MR1232192, MR2797562, MR946637, MR2033231} and references therein.
In \cite{MR946637}, Fujii proved a Fefferman-Stein type inequality as in \eqref{eq0807_01} when the underlying space is $\bR^d$ with the Lebesgue measure and the weighted norms on both sides of the inequality have two different weights.
If two weights are the same, they belong to the class of Muchenhoupt weights $A_\infty$.
See the definition of $A_p$, $p \in (1,\infty]$ in Section \ref{sec1}.
In Martell's paper \cite{MR2033231}, one can find an analog of the classical Fefferman-Stein inequality when the underlying space is a space of homogeneous type having either a finite or an infinite measure.
The sharp function in \cite{MR2033231} is a generalization of the classical  sharp function and is associated with approximations of the identity.
Theorem 4.2 in \cite{MR2033231} is quite close to Theorem \ref{thm2} in this paper in the sense that both theorems deal with spaces of homogeneous type with a finite or infinite underlying measure.
On the other hand, in \cite{MR2033231} a distance, instead of a quasi-distance, is assumed.

Sharp functions in this paper are based on a filtration of partitions, instead of balls, of the underlying space
and are in a convenient form to utilize Fefferman-Stein type inequalities in the proofs of $L_p$ or $L_{p,q}$-estimates for elliptic and parabolic equations/systems with irregular coefficients.
In the case of a finite underlying measure, our theorems are more quantitative than Theorem 4.2 in \cite{MR2033231} and are stated in such a way that one can control the weighted $L_p$ or $L_{p,q}$-norm of a function $f$ by that of $f^{\#}_{\text{dy}}$ if the support of $f$ is sufficiently small.
Moreover, we have a different form of the Fefferman-Stein theorem (Theorem \ref{thm3}) and the mixed-norm versions (Corollaries \ref{cor2} and \ref{cor3}), which are very useful for the mixed-norm case as well as when equations/systems have very irregular coefficients.
Especially, Theorem \ref{thm3} and Corollary \ref{cor3} generalize \cite[Theorem 2.7]{MR2540989}, where the whole Euclidean space $\bR^d$ is considered with the Lebesgue measure and no weights.

By following the arguments from \cite{MR2435520}, we give detailed proofs of the Fefferman-Stein type inequalities presented in this paper.
In particular, to derive the mixed-norm case from the unmixed version,
we refine the extrapolation theorem of J. L. Rubio de Francia \cite{MR745140} to incorporate spaces of homogeneous type.
In fact, such an extension
is mentioned in \cite{MR2078632} without a proof.
In Appendix \ref{appendix01},
we present a proof of our version of the theorem, the statement of which is modified from the original extrapolation theorem in that the assumption of the theorem is required to hold only for weights in a certain subset of the Muckenhoupt weights $A_p$, not all weights in $A_p$.
See Theorem \ref{thm0623}.
It turns out that such a modification, when applied to $L_p$ or $L_{p,q}$-estimates, gives more precise information on the parameters involved in the estimates.

After we treat the Fefferman-Stein type inequalities in Sections \ref{sec1} and \ref{sec03}, as examples we prove a priori weighted $L_{p,q}$-estimates for three classes of systems/equations.
In Section \ref{sec4} we consider non-divergence type higher order elliptic and parabolic systems defined in the whole Euclidean space or on a half space.
The coefficients of the systems have small mean oscillations with respect to the spatial variables in small balls or cylinders and, in the parabolic case, have no regularity assumptions (only measurable) in the time variable.
Coefficients in this class (called BMO$_x$ coefficients) are less regular than those having vanishing mean oscillations (VMO).
See \cite{MR2771670} and references therein for a discussion about BMO and VMO coefficients.
Here we generalize the results in \cite{MR2771670}, where no weights and unmixed norms are considered.

The novelty of the results in Section \ref{sec4} as well as those in the later sections is that we prove mixed-norm estimates for arbitrary $p, q \in (1,\infty)$.
In \cite{MR2352490} Krylov proved mixed-norm estimates for second order parabolic equations in $\bR^{d+1}$
with the same class of coefficients in Section \ref{sec4}.
However, due to Lemma 3.3 there and the fact that the Hardy-Littlewood maximal function theorem holds for $p \in (1,\infty]$, the estimates are proved only for $q > p$.
The mixed norm in \cite{MR2352490} is defined as
$$
\|f\|_{q,p} = \left( \int_\bR \left( \int_{\bR^d} |f(t,x)|^p \, dx \right)^{q/p} \, dt \right)^{1/q},
\quad
1 < p < q < \infty.
$$
When the coefficients are VMO in $x$ and independent with respect to $t$, mixed-norm estimates for parabolic systems in non-divergence form were established in \cite{MR2286441}.
When the coefficients are measurable functions of only $t$, weighted mixed norm estimates for non-divergence type parabolic equations on a half space or on a wedge were proved in \cite{MR2561181, MR3231530} with power type weights.
Recently, in \cite{GV14} Gallarati and Veraar proved $L_{p,q}$-estimates for any $p, q \in (1,\infty)$ when the coefficients are uniformly continuous in the spatial variables and measurable in time.
To the best of the authors' knowledge, our results regarding mixed norms are the first to deal with not only the case for arbitrary $p, q \in (1,\infty)$ but also higher order (including second order) elliptic and parabolic systems/equations with BMO coefficients.
It is unclear to us whether it is possible to prove such mixed-norm estimates without using $A_p$ weights.

In Section \ref{sec06} we further relax the regularity assumptions on coefficients for second order elliptic and parabolic equations in non-divergence form.
The main feature of the coefficients is that they have no regularity assumptions in one spatial variable.
Such coefficients are considered in \cite{MR2338417, MR2300337, MR2332574, MR2644213, MR2584743, MR2896169, MR2833589} for equations defined in the whole Euclidean space or on a half space with the Lebesgue measure and no weights.
In a recent paper \cite{MR3318165}, the authors investigated second order elliptic and parabolic equations in both divergence and non-divergence form with the same class of coefficients as in Section \ref{sec06} in weighted Sobolev spaces with weights being certain powers of the distance function to the boundary.
These are the same weights used by Krylov, for instance, in \cite{MR1708104}.
Note that the powers of the distance function in \cite{MR1708104} vary with the order of derivatives and, depending on the power, such weights may not be in the class of $A_p$ weights.
Thus the results in \cite{MR1708104} cannot be directly deduced from those in this paper.
On the other hand, Theorem \ref{thm5.4} generalizes the main result in \cite{DKZ14} on weighted $L_p$ estimates for the Neumann boundary value problem. See Remark \ref{rem0811}.

Finally, in Section \ref{sec07} we prove a priori weighted $L_{p,q}$-estimates for higher order (including second order) elliptic and parabolic systems in divergence form.
The coefficients are, roughly speaking, locally measurable in one direction and have small mean oscillations with respect to the other directions in small balls or cylinders.
If no weights and unmixed norms are assumed, the results in Section \ref{sec06} have been developed in \cite{MR2835999}.
Concerning the second order equations/systems case, see, for instance, \cite{MR2601069, MR2764911, MR2800569, MR2680179, MR3073000} and references therein.
Recently, in \cite{MR3461425, MR3467697} the authors treated divergence type second order parabolic systems in Sobolev and Orlicz spaces with $A_p$ weights and unmixed norms.
A noteworthy difference
is that in this paper the weights are in $A_p$ for $L_{p,q}$-estimates when $p=q$, whereas in \cite{MR3461425, MR3467697} the weights are in $A_{p/2}$ for $L_p$-estimates.
Due to the property of Muckenhoupt weights, the set $A_p$ is strictly larger than $A_{p/2}$.

Throughout Sections \ref{sec4}, \ref{sec06}, and \ref{sec07}, the main approach is the mean oscillation estimates combined with Fefferman-Stein type inequalities.
For this approach, see, for instance, \cite{MR2304157, MR2435520, MR2771670}.
When deriving desired mean oscillation estimates, we take full advantage of the existence and uniqueness results as well as unmixed $L_p$-estimates for second and higher order elliptic and parabolic equations/systems in Sobolev spaces without weights proved in \cite{MR2771670,MR2300337, MR2833589, MR2835999}.
As to divergence type systems in Reifenberg flat domains, here we prove a priori $L_{p,q}$-estimates using sharp functions, whereas previously a level set type argument in the spirit of \cite{MR1486629} is used.
To the best of the authors' knowledge, sharp functions have not been previously used when treating equations/systems on Reifenberg flat domains.
Indeed, this was mainly due to the lack of the corresponding Fefferman-Stein theorem.
For equations/systems in Reifenberg flat domain and the level set type argument, see, for instance, \cite{MR2069724, MR2835999}.
Utilizing sharp functions makes it possible to derive mixed-norm estimates from mean oscillation estimates via the Fefferman-Stein theorem with mixed norms.

The paper is organized as follows.
In Sections \ref{sec1} and \ref{sec03}, we present generalized versions of the Fefferman-Stein theorem on sharp functions and their proofs.
In Section \ref{sec3}, we introduce some notation and function spaces to be used in the later sections.
Finally, as hinted above, in Sections \ref{sec4}, \ref{sec06}, and \ref{sec07}, we establish a priori weighted $L_{p,q}$-estimates for elliptic and parabolic equations/systems in either non-divergence or divergence form.
In Section \ref{sec8}, we explain how to derive the corresponding existence results from the a priori estimates in the previous sections.
In Appendix \ref{appendix01}, we give a detailed proof of a refined Rubio de Francia extrapolation theorem and an auxiliary lemma, the latter of which is used in Section \ref{sec07}.

\section{Generalized Fefferman-Stein theorem}
                            \label{sec1}

Let $\cX$ be a set. Recall that a nonnegative symmetric function on $\cX \times \cX$  is called a quasi-metric on $\cX$ if there exists a constant $K_1$ such that
\begin{equation}
							\label{eq0709_1}
\rho(x,  y)  \le  K_1(\rho(x,  z)  + \rho(z,  y))
\end{equation}
for any $x, y, z \in \cX$, and $\rho(x, y)=0$ if and only if $x = y$. We denote balls in $\cX$ by
$$
B_r(x)=\{y\in \cX\,:\,\rho(x,y)<r\}.
$$
Let $\mu$ be a complete Borel measure defined on a $\sigma$-algebra on $\cX$ which contains all the balls in $\cX$, and satisfy the doubling property: there  exists  a constant $K_2$ such that for any $x\in \cX$ and  $r>0$,
\begin{equation}
							\label{eq0709_2}
0 < \mu(B_{2r}(x)) \le K_2\mu(B_r(x)) < \infty.
\end{equation}
Throughout the paper, we assume that the Lebesgue differentiation theorem holds in $(\cX,\mu)$.
For this theorem one may assume that $\mu$ is Borel regular or the set of continuous functions is dense in $L_1(\cX, \mu)$. See \cite{MR2867756, MR0442579} and references therein.
We say that $(\cX,\rho,\mu)$ is a space of homogeneous type if $\cX$ is a set
endowed with a quasi-metric $\rho$ and $(\cX,\rho,\mu)$ satisfies the above assumptions.
We also assume that balls $B_r(x)$ are open in $\cX$.

Due to a result by Christ \cite[Theorem 11]{MR1096400}, there exists a filtration of partitions (also called dyadic decompositions) of $\cX$ in the following sense.

\begin{theorem}
							\label{thm0326}
Let $(\cX,\rho,\mu)$ be a space of homogeneous type as described above.
For each $n\in\bZ$, there is a collection of disjoint open subsets $\bC_n:=\{Q_\alpha^n\,:\,\alpha\in I_n\}$ for some index set $I_n$, which satisfy the following properties
\begin{enumerate}
\item For any $n\in \bZ$, $\mu(\cX\setminus \bigcup_{\alpha}Q_\alpha^n)=0$;
\item For each $n$ and $\alpha\in I_n$, there is a unique $\beta\in I_{n-1}$ such that $Q_\alpha^n\subset Q_\beta^{n-1}$;
\item For each $n$ and $\alpha\in I_n$, $\operatorname{diam}(Q_\alpha^n)\le N_0\delta^n$;
\item Each $Q_\alpha^n$ contains some ball $B_{\varepsilon_0\delta^n}(z_\alpha^n)$
\end{enumerate}
for some constants $\delta\in (0,1)$, $\varepsilon_0>0$, and $N_0$ depending only on $K_1$ and $K_2$.
\end{theorem}

Instead of the above partitions one may use a significantly refined version of dyadic decompositions in \cite{MR2901199}.

Denote $\tilde \cX=\bigcap_{n\in \bZ}\bigcup_{\alpha}Q_\alpha^n$. By properties (1) and (2) in Theorem \ref{thm0326}, we have
$$
\mu(\cX\setminus \tilde \cX)=0,\quad \tilde \cX=\lim_{n\nearrow \infty}\bigcup_{\alpha}Q_\alpha^n.
$$
By properties (2), (3), and (4) in Theorem \ref{thm0326}, we have
\begin{equation}
                                    \label{eq4.57}
\mu(Q_\beta^{n-1})\le N_1\mu(Q_\alpha^{n})
\end{equation}
for any $n$, $\alpha\in I_n$, $\beta\in I_{n-1}$ such that $Q_\alpha^n\subset Q_\beta^{n-1}$.

We denote $f \in L_p(\cX,\mu)$ or $f \in L_p(\mu)$, $p \in [1,\infty)$, if
$$
\int_\cX |f|^p\, \mu(dx) = \int_\cX |f|^p \, d\mu < \infty.
$$
For a function $f\in L_{1,\text{loc}}(\cX,\mu)$ and $n\in \bZ$, we set
$$
f_{|n}(x)=\dashint_{Q_\alpha^n}f(y)\,\mu(dy),
$$
where $x\in Q_\alpha^n \in \bC_n$.
For $x\in \tilde \cX$, we define the (dyadic) maximal function and sharp function of $f$ by
\begin{align*}
\cM_{\text{dy}} f(x)&=\sup_{n<\infty}|f|_{|n}(x),\\
f_{\text{dy}}^{\#}(x)&=\sup_{n<\infty}\dashint_{Q_\alpha^n\ni x}|f(y)-f_{|n}(x)|\,\mu(dy).
\end{align*}
Since the quantities of concern are the $L_p$ norms and $\mu(\cX\setminus \tilde \cX)=0$, we take the zero extension of them to $\cX\setminus \tilde \cX$.
We also define the maximal function and sharp function of $f$ over balls by
\begin{align*}
\cM f(x)&=\sup_{B_r(x^*) \ni x}\dashint_{B_r(x^*)}|f(y)|\,\mu(dy),\\
f^{\#}(x)&=\sup_{B_r(x^*) \ni x}\dashint_{B_r(x^*)}|f(y)-(f)_{B_r(x^*)}|\,\mu (dy),
\end{align*}
where the supremums are taken with respect to all balls $B_r(x^*)$ containing $x$ and
$$
(f)_{B_r(x^*)}=\dashint_{B_r(x^*)} f(y)\,\mu (dy).
$$
The advantage to work on the maximal and sharp functions defined over dyadic decompositions is that one can neglect the geometry of the space.
By Properties (2) and (3) of the filtration in Theorem \ref{thm0326} and the doubling property of $\mu$, it is easily seen that
\begin{equation}
                    \label{eq8.52}
\cM_{\text{dy}} f(x)\le N \cM f(x)\quad \text{and}\quad
f_{\text{dy}}^{\#}(x)\le Nf^{\#}(x)
\quad \mu\text{-a.e.},
\end{equation}
where $N$ is a constant depending only on $K_1$ and $K_2$.

For any $p\in (1,\infty)$, let $A_p=A_p(\mu)= A_p(\cX, d \mu)$ be the set of all nonnegative functions $w$ on $(\cX,\rho,\mu)$ such that
$$
[w]_{A_p}:=\sup_{x_0\in \cX,r>0}\left(\dashint_{B_r(x_0)}w(x)\,d\mu\right)
\left(\dashint_{B_r(x_0)}\big(w(x)\big)^{-1/(p-1)}\,d\mu\right)^{p-1}<\infty.
$$
By H\"older's inequality, we have
\begin{equation}
                                \label{eq7.53}
A_p\subset A_q,\quad
1 \le
[w]_{A_q}\le [w]_{A_p},
\quad 1<p<q<\infty.
\end{equation}
Denote $A_\infty=\bigcup_{p=2}^\infty A_p$.
We write $f \in L_p(\cX, w \, d\mu)$ or $f \in L_p(w \, d \mu)$ if
$$
\int_\cX |f|^p w \, \mu(dx) = \int_\cX |f|^p w \, d\mu < \infty.
$$
We use $\omega(\cdot)$ to denote the measure $\omega(dx)=w\mu(dx)$, i.e., for $A \subset \cX$,
\begin{equation*}
\omega(A) = \int_A w(x) \, \mu(dx).
\end{equation*}

Throughout this section, we assume that $(\cX,\rho,\mu)$ is a space of homogeneous type with a filtration of partitions from Theorem \ref{thm0326}.
The following Hardy-Littlewood maximal function theorem with $A_p$ weights was obtained in \cite{MR0740173}.
\begin{theorem}
                                    \label{thm1}
Let $p\in (1,\infty)$, $w\in A_p$. Then for any $f\in L_p(w\,d\mu)$, we have
$$
\|\cM f\|_{L_p(w\,d\mu)}\le N\|f\|_{L_p(w\,d\mu)},
$$
where $N=N(K_1,K_2,p,[w]_{A_p})>0$.
If $K_0 \ge 1$ and $[w]_{A_p}\le K_0$, then one can choose $N$ depending only on $K_1$, $K_2$, $p$, and $K_0$.
\end{theorem}
Note that in Theorem \ref{thm1}, $\mu(\cX)$ can be either finite or infinite, and $\cX$ is allowed to be a bounded space with respect to $\rho$.

Our first result of this paper is the following generalization of the Fefferman-Stein theorem on sharp functions with $A_\infty$ weights.

\begin{theorem}
                                    \label{thm2}
Let $p,q \in (1, \infty)$, $\varepsilon>0$, $K_0 \ge 1$, $w\in A_q$, $[w]_{A_q}\le K_0$, and $f\in L_p(w\,d\mu)$.

\begin{enumerate}
\item[(i)] When $\mu(\cX)<\infty$, we have
\begin{equation}
                            \label{eq1.51}
\|f\|_{L_p(w\,d\mu)}\le N\|f_{\operatorname{dy}}^{\#}\|_{L_p(w\,d\mu)}+
N(\mu(\cX))^{-1}\big(\omega(\operatorname{supp} f)\big)^{\frac 1 p}\|f\|_{L_1(\mu)}.
\end{equation}
If in addition we assume that $\operatorname{supp} f \subset B_{r_0}(x_0)$ and $\mu(B_{r_0}(x_0))\le \varepsilon\mu(\cX)$ for some $r_0\in (0,\infty)$ and $x_0\in \cX$, then
\begin{equation}
                                        \label{eq3.41}
\|f\|_{L_p(w\,d\mu)}\le N\|f_{\operatorname{dy}}^{\#}\|_{L_p(w\,d\mu)}+
N\varepsilon \|f\|_{L_p(w\,d\mu)}.
\end{equation}
Here $N=N(K_1,K_2,p,q,K_0)>0$. In particular,  when $\varepsilon$ is sufficiently small depending on $K_1$, $K_2$, $p$, $q$, $K_0$, it holds that
\begin{equation}
                            \label{eq1.54}
\|f\|_{L_p(w\,d\mu)}\le N\|f_{\operatorname{dy}}^{\#}\|_{L_p(w\,d\mu)}.
\end{equation}

\item[(ii)] When $\mu(\cX)=\infty$, \eqref{eq1.54} holds.
\end{enumerate}
\end{theorem}
In the special case when $\cX=\bR^d$ with the Lebesgue measure, Theorem \ref{thm2} was established in \cite{MR946637}.

Our second result is a further generalization of Theorem \ref{thm2} in the spirit of \cite[Theorem 2.7]{MR2540989}.

\begin{theorem}							\label{thm3}
Let $p,q \in (1, \infty)$, $\varepsilon>0$, $K_0 \ge 1$, $w\in A_q$, and $[w]_{A_q}\le K_0$.
Suppose that
$$
f\in L_p(w \, d \mu),\quad g\in L_p(w\,d\mu), \quad v\in L_p(w\,d\mu),\quad |f| \le v,
$$
and for each $n \in \bZ$ and $Q \in \bC_n$,
there exists a measurable function $f^Q$ on $Q$
such that $|f| \le f^Q \le v$ on $Q$ and
\begin{equation}							 \label{eq10.53}
\dashint_Q |f^Q(x) - \left(f^Q\right)_Q| \,\mu(dx)
\le g(y)\quad \forall \,y\in Q.
\end{equation}

\begin{enumerate}
\item[(i)] When $\mu(\cX)<\infty$, we have
\begin{equation*}
\|f\|^p_{L_p(w\,d\mu)}\le N\|g\|^\beta _{L_p(w\,d\mu)}\|v \|_{L_p(w\,d\mu)}^{p-\beta}+
N(\mu(\cX))^{-p} \omega (\operatorname{supp} v)\|v\|^p_{L_1(\mu)}.
\end{equation*}
If in addition we assume that $\operatorname{supp} v \subset B_{r_0}(x_0)$ and $\mu(B_{r_0}(x_0))\le \varepsilon\mu(\cX)$ for some $r_0\in (0,\infty]$ and $x_0\in \cX$, then
\begin{equation}
                                \label{eq12.09}
\| f \|_{L_p(w\,d\mu)}^p
\le N\|g\|^\beta _{L_p(w\,d\mu)}\|v \|_{L_p(w\,d\mu)}^{p-\beta}
+N\varepsilon^p\|v \|_{L_p(w\,d\mu)}^{p}.
\end{equation}
Here $(\beta, N) =  (\beta,N)(K_1,K_2,p,q,K_0)>0$.

\item[(ii)] When $\mu(\cX)=\infty$, \eqref{eq12.09} holds with $\varepsilon=0$.
\end{enumerate}
\end{theorem}

In Theorem \ref{thm3}, we are mostly interested in the case when $v\le K_3|f|$ for some constant $K_3>0$. Then from \eqref{eq12.09}, we obtain
$$
\| f \|_{L_p(w\,d\mu)}
\le N\|g\|_{L_p(w\,d\mu)}
$$
provided that $\varepsilon$ is sufficiently small.
Note that when $f^Q=v=|f|$ and $g=2f_{\text{dy}}^{\#}(x)$, \eqref{eq10.53} is satisfied due to the triangle inequality and Theorem \ref{thm3} is reduced to Theorem \ref{thm2}.

By using the extrapolation theory of Rubio de Francia (see, for instance, \cite{MR2797562}),
we deduce Corollaries \ref{cor1}, \ref{cor2}, and \ref{cor3} below from Theorems \ref{thm1}, \ref{thm2}, and \ref{thm3}.
In particular, we use the following version of the extrapolation theorem,
the main feature of which is that, to prove the desired estimate \eqref{eq0623_02} for a given $p \in (1,\infty)$,
the estimate \eqref{eq0623_01} as an assumption needs to hold only for a subset of $A_{p_0}$, not for all weights in $A_{p_0}$.
Certainly, to obtain the estimate \eqref{eq0623_02} for all $p \in (1,\infty)$, we need the estimate \eqref{eq0623_01} for all weights in $A_{p_0}$.
For a proof, one can just refer to the proof of Theorem 1.4 in \cite{MR2797562} with a slight clarification of the constants involved.
For the reader's convenience, we present a proof in Appendix \ref{appendix01}.

\begin{theorem}
							\label{thm0623}
Let $f, g:\cX \to \bR$ be a pair of measurable functions, $p_0, p \in (1,\infty)$, and $w \in A_p$.
Then there exists a constant $\Lambda_0=\Lambda_0(K_1, K_2, p_0, p, [w]_{A_p}) \ge 1$ such that if
\begin{equation}
							\label{eq0623_01}
\|f\|_{L_{p_0}(\tilde{w} \, d\mu)} \le N_0 \|g\|_{L_{p_0}(\tilde{w} \, d\mu)}
\end{equation}
for every $\tilde{w} \in A_{p_0}$ satisfying $[\tilde{w}]_{A_{p_0}} \le \Lambda_0$,
then we have
\begin{equation}
							\label{eq0623_02}
\|f\|_{L_p(w \, d\mu)} \le 4 N_0 \|g\|_{L_p(w \, d\mu)}.
\end{equation}
Moreover, if $[w]_{A_p} \le K_0$, then $\Lambda_0$ can be chosen so that it depends only on $K_1$, $K_2$,
$p_0$, $p$, and $K_0$.
\end{theorem}

In the sequel, we write $x=(x',x'')$. Let $(\cX',\rho_1,\mu_1)$ and $(\cX'',\rho_2,\mu_2)$ be two spaces of homogeneous type. Define $\mu$ to be the completion of the product measure on $\cX'\times \cX''$ and
$$
\rho((x',x''), (y',y''))=\rho_1(x',y')+\rho_2(x'',y'')
$$
be a quasi-metric on $\cX'\times \cX''$. Let $\cX$ be a subset of $\cX'\times \cX''$, which satisfies the following conditions:

\begin{enumerate}[(a)]
\item $(\cX,\rho|_{\cX\times \cX},\mu|_\cX)$ is of homogeneous type;
\item for any $p\in (1,\infty)$, $w_1=w_1(x')\in A_p(\mu_1)$ and $w_2=w_2(x'')\in A_p(\mu_2)$,
$$
w=w(x):=w_1(x')w_2(x'')
$$
is an $A_p$ weight on $(\cX,\rho|_{\cX\times \cX},\mu|_\cX)$.
\end{enumerate}
The doubling property of $\mu$ along with the condition (b) is are satisfied when, for instance, there exists a constant $\delta>0$ such that $\mu(B_r(x)  \cap \cX)\ge \delta \mu(B_r(x))$ for any $x\in \cX$ and $r>0$ because they are satisfied by $\cX'\times \cX''$.

For any $p,q\in (1,\infty)$ and weights $w_1=w_1(x')$ and $w_2=w_2(x'')$, we define the weighted mixed norm on $\cX$ by
\begin{equation}
							\label{eq0806_02}
\begin{aligned}
&\|f\|_{L_{p,q}(w\,d\mu)}\\
&:=\|f\|_{L_{p,q}(\cX,w\,d\mu)}=\left(\int_{\cX''}\big(\int_{\cX'} |f|^p I_{\cX} w_1(x')\,\mu_1(dx')\big)^{q/p}w_2(x'')\,\mu_2(dx'')\right)^{1/q}.
\end{aligned}
\end{equation}

By applying the extrapolation theorem to Theorem \ref{thm1}, we obtain the following corollary.

\begin{corollary}
                                    \label{cor1}
Let $p,q\in (1,\infty)$, $K_0 \ge 1$, $w_1=w_1(x')\in A_p(\mu_1)$, $w_2=w_2(x'')\in A_q(\mu_2)$,
$[w_1]_{A_p}\le K_0$, $[w_2]_{A_q}\le K_0$, and $w=w(x)=w_1(x')w_2(x'')$. Then for any $f\in L_{p,q}(w\,d\mu)$, we have
$$
\|\cM f\|_{L_{p,q}(w\,d\mu)}\le N\|f\|_{L_{p,q}(w\,d\mu)}.
$$
where $N=N(K_1,K_2,p,q,K_0)>0$.
\end{corollary}

\begin{corollary}
                                        \label{cor2}
Let $p,q,p',q'\in (1,\infty)$, $K_0 \ge 1$, $w_1=w_1(x')\in A_{p'}(\mu_1)$, $w_2=w_2(x'')\in A_{q'}(\mu_2)$,
$[w_1]_{A_{p'}}\le K_0$, $[w_2]_{A_{q'}}\le K_0$,
$w=w(x):=w_1(x')w_2(x'')$, and $f\in L_{p,q}(w\,d\mu)$.
Suppose that either $\mu(\cX)=\infty$ or $\operatorname{supp}\,f\subset B_{r_0}(x_0)$ and $\mu(B_{r_0}(x_0))\le \varepsilon\mu(\cX)$ for some $r_0\in (0,\infty)$ and $x_0\in \cX$, where $\varepsilon>0$ is a constant depending on $K_1$, $K_2$, $p$, $q$, $p'$, $q'$, and $K_0$.
Then we have
\begin{equation}
                                        \label{eq8.31}
\|f\|_{L_{p,q}(w\,d\mu)}\le N\|f_{\operatorname{dy}}^{\#}\|_{L_{p,q}(w\,d\mu)}.
\end{equation}
where $N=N(K_1,K_2,p,q,p',q',K_0)>0$.
\end{corollary}

For the proof, due to \eqref{eq7.53} we may assume that $p'q/p\ge q'$.
Indeed, otherwise, we find a large enough $p''$ such that $p'' q/p \ge q'$. Then $w_1 \in A_{p'}(\mu_1) \subset A_{p''}(\mu_1)$. For $w_1\in A_{p'}(\mu_1)$, we define
$$
\psi(x'')=\|I_{\cX}f(\cdot,x'')\|^{p/p'}_{L_p(w_1\,d\mu_1)},
\quad \phi(x'')=\|I_{\cX}f_{\text{dy}}^{\#}(\cdot,x'')\|^{p/p'}_{L_p(w_1\,d\mu_1)}.
$$
It follows from Theorem \ref{thm2} that, for any $\widetilde{K}_0 \ge 1$ and $\tilde{w}_2\in A_{p'}(\mu_2)$ satisfying $[\tilde{w}_2]_{A_{p'}} \le \widetilde{K}_0$ ,
we have
\begin{align}
							\label{eq0624_02}
&\|\psi\|_{L_{p'}(\tilde{w}_2\,d\mu_2)}
=\|f\|_{L_{p'}(w_1\tilde w_2\,d\mu|_{\cX})}\\
&\le N\|f_{\operatorname{dy}}^{\#}\|_{L_{p'}(w_1\tilde w_2\,d\mu|_{\cX})}+
N\varepsilon \|f\|_{L_{p'}(w_1\tilde w_2\,d\mu|_{\cX})}
\le N\|\phi+\varepsilon\psi\|_{L_{p'}(\tilde{w}_2\,d\mu_2)},
\end{align}
where $N$ depends on $K_1$, $K_2$, $p$, $q$, $p'$, $q'$, $K_0$, and $\widetilde{K}_0$.
Note that
$A_{q'}(\mu_2)\subset A_{p'q/p}(\mu_2)$, that is,
$$
[w_2]_{A_{p'q/p}} \le [w_2]_{A_{q'}} \le K_0.
$$
Setting $\widetilde{K}_0=\Lambda_0$ with $\Lambda_0 = \Lambda_0(K_1, K_2, p', p'q/p, K_0)$ from Theorem \ref{thm0623},
by Theorem \ref{thm0623} and the fact that the constant $N$ in \eqref{eq0624_02} is determined only by $K_1$, $K_2$, $p$, $q$, $p'$, $q'$,  $K_0$, and $\Lambda_0$, we get
$$
\|\psi\|_{L_{p'q/p}(w_2\,d\mu_2)}\le N\|\phi+\varepsilon\psi\|_{L_{p'q/p}(w_2\,d\mu_2)},
$$
where $N = N(K_1, K_2, p,q,p',q',K_0)$.
Upon taking $\varepsilon$ sufficiently small, we reach \eqref{eq8.31}.

\begin{corollary}
                                    \label{cor3}
Let $p,q,p',q'\in (1,\infty)$, $K_0 \ge 1$, $w_1=w_1(x')\in A_{p'}(\mu_1)$, $w_2=w_2(x'')\in A_{q'}(\mu_2)$,
$[w_1]_{A_{p'}}\le K_0$, $[w_2]_{A_{q'}}\le K_0$,
$w=w(x):=w_1(x')w_2(x'')$, and $f, g \in L_{p,q}(w\,d\mu)$.
Suppose that either $\mu(\cX)=\infty$ or $\operatorname{supp} f\subset B_{r_0}(x_0)$ and $\mu(B_{r_0}(x_0))\le \varepsilon\mu(\cX)$ for some $r_0\in (0,\infty]$ and $x_0\in \cX$, where $\varepsilon>0$ is a constant depending on $K_1$, $K_2$, $p$, $p'$, $q$, $q'$, and $K_0$.
Moreover, for each $n \in \bZ$ and $Q \in \bC_n$,
there exists a measurable function $f^Q$ on $Q$
such that $|f| \le f^Q \le K_3|f|$ on $Q$ for some constant $K_3>0$
and
\begin{equation*}	
\dashint_Q |f^Q(x) - \left(f^Q\right)_Q| \,\mu(dx)
\le g(y)\quad \forall \,y\in Q.
\end{equation*}
Then we have
\begin{equation*}
\| f \|_{L_{p,q}(w\,d\mu)}
\le N\|g\|_{L_{p,q}(w\,d\mu)},
\end{equation*}
where $N=N(K_1,K_2,K_3,p,q,p',q',K_0)>0$.
\end{corollary}

\section{Proof of Theorems \ref{thm2} and \ref{thm3}}
							\label{sec03}

Let $(\cX,\rho,\mu)$ be a space of homogeneous type with a filtration of partitions $\{\bC_n\,:\,n\in \bZ\}$ introduced as in Section \ref{sec1}.
The notation below are chosen to be compatible with those in \cite{MR2435520}. Let $\tau=\tau(x)$ be a function on $\tilde \cX$ with values in $\{\infty,0,\pm 1,\pm 2,\ldots\}$. We call $\tau$ a stopping time relative to the filtration if for each $n=0,\pm 1,\pm 2,\ldots$, the set $\{x \in \tilde{\cX}\,:\,\tau(x)=n\}$ is either empty or the union of some sets in $\bC_n$ intersected with $\tilde{\cX}$.
For any stopping time $\tau$, we define $f_{|\tau}(x)=f_{|\tau(x)}(x)$ for any $x\in \tilde \cX$ such that $\tau(x)<\infty$ and $f_{|\tau}(x)=f(x)$ otherwise.

To prove Theorems \ref{thm2} and \ref{thm3}, we estimate the measure of level sets of $f$ by its maximal and sharp functions, which generalizes Lemma 3.2.9 of \cite{MR2435520}, where $w=1$, $\cX=\bR^d$, and $\mu(\cX)=\infty$. We follow the argument there with some modifications.

Using H\"{o}lder's inequality and the definition of $A_p$, one can get

\begin{lemma}
							\label{lem0320_1}
Let $p \in (1,\infty)$. If $w \in A_p$, then
$$
\left(\dashint_B f \, \mu(d x)\right)^p \le \frac{[w]_{A_p}}{\omega(B)} \int_B f^p w(x) \, \mu(dx)
$$
for all nonnegative $f$ and all balls $B$ in $\cX$.
\end{lemma}

\begin{lemma}
							\label{lem0331_1}
Let $p \in (1,\infty)$, $K_0 \ge 1$, and $w \in A_p(\mu)$ with $[w]_{A_p}\le K_0$.
Then we have
$$
\frac{\omega(E)}{\omega(Q_{\alpha}^n)}
\le N \left( \frac{\mu(E)}{\mu\left(Q_\alpha^n\right)}\right)^{\beta}
$$
for any $E \subset Q_\alpha^n$,
where $(\beta, N) = (\beta, N)(K_1, K_2, p,K_0)$.
\end{lemma}

\begin{proof}
Under the assumptions, a reverse H\"{o}lder's inequality for $A_p$ weights in spaces of homogeneous type was established in  \cite[Theorem 3.2]{MS1981}, from which and H\"older's inequality one can derive the inequality in the lemma.
\end{proof}

\begin{lemma}
							\label{lem0319_1}
Let $Q_{\alpha_n}^n$ be open sets from Theorem \ref{thm0326} such that
$$
Q_{\alpha_n}^n \subset Q_{\alpha_{n-1}}^{n-1},
\quad n \in \bZ.
$$
Then either $\mu\left( \bigcup_n Q_{\alpha_n}^n \right) = \infty$ or there exists $n_0 \in \bZ$ such that
$$
\mu(\cX) = \mu \left(Q_{\alpha_{n_0}}^{n_0} \right).
$$
Furthermore, for any $w \in A_p(\mu)$, $1 < p < \infty$, we have
$$
\omega\left( \bigcup_n Q_{\alpha_n}^n \right) = \infty
\quad
\text{or}
\quad
\omega(\cX) = \omega \left(Q_{\alpha_{n_0}}^{n_0} \right)
$$
depending on either $\mu\left( \bigcup_n Q_{\alpha_n}^n \right)$ is infinite or not.
\end{lemma}

\begin{proof}
First assume that there are infinitely many $n \in \bZ \cap \{n \le 0\}$
such that
$$
\mu \left(Q_{\alpha_{n-1}}^{n-1} \setminus Q_{\alpha_n}^n \right) > 0.
$$
Thanks to property (1) in Theorem \ref{thm0326}, for all such $n$'s,
$$
Q_{\alpha_{n-1}}^{n-1} \setminus Q_{\alpha_n}^n \not\subset \left(\cX \setminus \bigcup_{\alpha} Q_\alpha^n\right).
$$
Thus, there exist $x \in \cX$ and $Q_{\alpha_n'}^n \in \bC_n$ such that
$$
x \in Q_{\alpha_{n-1}}^{n-1} \setminus Q_{\alpha_n}^n,
\quad
x \in \bigcup_\alpha Q_\alpha^n,
\quad
\text{and}
\quad
x \in Q_{\alpha_n'}^n \subset Q_{\alpha_{n-1}}^{n-1}.
$$
Since
$\mu\left(Q_{\alpha_{n-1}}^{n-1}\right) \le N_1 \mu\left(Q_{\alpha_n'}^n\right)$,
we have $\mu\left(Q_{\alpha_n}^n\right) \le N_1 \mu\left(Q_{\alpha_n'}^n\right)$.
This shows
\begin{equation*}
\mu\left(Q_{\alpha_{n-1}}^{n-1}\right) \ge \mu\left(Q_{\alpha_n}^n\right) + \mu \left(Q_{\alpha_n'}^n\right) \ge \left(1+\frac{1}{N_1}\right) \mu\left(Q_{\alpha_n}^n\right)
\end{equation*}
and
$$
\mu\left(Q_{\alpha_{n-1}}^{n-1} \setminus Q_{\alpha_n}^n\right) \ge \frac{1}{N_1} \mu\left(Q_{\alpha_n}^n\right).
$$
This implies that
$$
\mu\left( \bigcup_n Q_{\alpha_n}^n \right) = \lim_{n \to -\infty} \mu\left(Q_{\alpha_n}^n\right) = \infty.
$$

Now we assume that there exists $n_0 \in \bZ$ such that
$$
\mu \left(Q_{\alpha_{n-1}}^{n-1} \setminus Q_{\alpha_n}^n \right) = 0
$$
for all $n \le n_0$.
Then
$$
\mu\left( \left(\bigcup_{n \in \bZ} Q_{\alpha_n}^n\right) \setminus Q_{\alpha_{n_0}}^{n_0}\right) = 0.
$$
In particular,
$$
\mu \left(Q_{\alpha_n}^n \setminus Q_{\alpha_{n_0}}^{n_0}\right) = 0
$$
for any $n \in \bZ$. On the other hand, for each $Q_{\alpha_n}^n$, $\alpha_n \in I_n$, there exists $z_{\alpha_n}^n \in Q_{\alpha_n}^n$ such that
$$
B_{\varepsilon_0 \delta^n} (z_{\alpha_n}^n) \subset Q_{\alpha_n}^n.
$$
We claim that
\begin{equation}
							\label{eq0317_1}
z_{\alpha_n}^n \in \overline{Q_{\alpha_{n_0}}^{n_0}}
\quad
\text{for all}
\,\, n \in \bZ.
\end{equation}
To prove this, suppose that $z_{\alpha_n}^n \in Q_{\alpha_n}^n \setminus \overline{Q_{\alpha_{n_0}}^{n_0}}$ for some $n \in \bZ$.
Then since the set $Q_{\alpha_n}^n \setminus \overline{Q_{\alpha_{n_0}}^{n_0}}$ is open, there exists an open ball $B_r(z_{\alpha_n}^n)$ such that
$$
B_r(z_{\alpha_n}^n) \subset Q_{\alpha_n}^n \setminus \overline{Q_{\alpha_{n_0}}^{n_0}}.
$$
However,
$$
0 < \mu \left( B_r(z_{\alpha_n}^n) \right)
\le \mu \left( Q_{\alpha_n}^n \setminus \overline{Q_{\alpha_{n_0}}^{n_0}} \right)
\le \mu\left( Q_{\alpha_n}^n \setminus Q_{\alpha_{n_0}}^{n_0}\right) = 0,
$$
which is a contradiction. Hence \eqref{eq0317_1} is proved.
Now for any fixed $x \in \cX$,
$$
\rho(x, z_{\alpha_n}^n) \le K_1 \left( \rho(x,z_{\alpha_{n_0}}^{n_0}) + \rho(z_{\alpha_{n_0}}^{n_0}, z_{\alpha_n}^n)\right) \le K_1 \left(\rho(x, z_{\alpha_{n_0}}^{n_0}) +  N_0 \delta^{n_0}\right),
$$
where the last inequality is due to the property (3) in Theorem \ref{thm0326} and the fact that $z_{\alpha_n}^n \in \overline{Q_{\alpha_{n_0}}^{n_0}}$.
This along with the property (4) in Theorem \ref{thm0326} implies that, for each $x\in \cX$,
$$
x \in B_{\varepsilon_0 \delta^n}(z_{\alpha_n}^n) \subset Q_{\alpha_n}^n,
$$
provided that $-n$ is sufficiently large.
Therefore, $\cX = \bigcup_n Q_{\alpha_n}^n$ and
$$
\mu(\cX) = \mu \left(\bigcup_n Q_{\alpha_n}^n \right) = \mu\left( \left(\bigcup_{n \in \bZ} Q_{\alpha_n}^n\right) \setminus Q_{\alpha_{n_0}}^{n_0}\right) + \mu \left(Q_{\alpha_{n_0}}^{n_0} \right) = \mu \left(Q_{\alpha_{n_0}}^{n_0} \right).
$$

To prove the second statement of the lemma, from the definition of the measure $\omega(\cdot)$, we easily see that $\omega(\cX) = \omega \left(Q_{\alpha_{n_0}}^{n_0} \right)$ if $\mu(\cX) = \mu\left(Q_{\alpha_{n_0}}^{n_0} \right)$.
Otherwise,
for each $n \in \bZ$, find balls $B_n$ satisfying
$$
Q_{\alpha_n}^n \subset B_n
\quad
\text{and}
\quad
\mu(B_n) \le N(K_1, K_2) \mu(Q_{\alpha_n}^n).
$$
By taking $B = B_m$ and $f = I_{Q_{\alpha_m}^m \setminus Q_{\alpha_n}^n}$, $m \le n$, in Lemma \ref{lem0320_1}, we obtain
$$
\left(1 - \frac{\mu(Q_{\alpha_n}^n)}{\mu(Q_{\alpha_m}^m)}\right)^p
\le N \left(\frac{\mu(Q_{\alpha_m}^m) -\mu(Q_{\alpha_n}^n)}{\mu(B_m)} \right)^p
\le N \left(1 - \frac{\omega(Q_{\alpha_n}^n)}{\omega(Q_{\alpha_m}^m)} \right),
$$
where $N=N(K_1,K_2, p, [w]_{A_p})$.
This shows $\omega(Q_{\alpha_n}^n) \to \infty$ if $\mu(Q_{\alpha_n}^n) \to \infty$ as $n \to -\infty$.
\end{proof}

\begin{lemma}
							\label{lem0331_2}
Let $p \in (1,\infty)$ and $w \in A_p(\mu)$.
For $x \in \tilde{\cX}$ and $f \in L_p(w \, d\mu)$,
\begin{equation}
							\label{eq0402_1}
\lim_{n \to -\infty} |f|_{|n}(x) =
\left\{\begin{aligned}
\|f\|_{L_1(\mu)} \left(\mu(\cX)\right)^{-1},
\quad &\text{if} \quad \mu(\cX) < \infty,
\\
0, \quad &\text{if} \quad \mu(\cX) = \infty.
\end{aligned}
\right.
\end{equation}
This holds true as well if $f \in L_1(\cX,\mu)$.
\end{lemma}

\begin{proof}
If $\mu(\cX) < \infty$, from Lemmas \ref{lem0320_1} and \ref{lem0319_1}, we see that $f \in L_1(\cX,\mu)$ and the claim in the lemma holds true.
When $\mu(\cX) = \infty$,
by Lemma \ref{lem0320_1}, for each $n \in \bZ$, we get
$$
\left(|f|_{|n}(x)\right)^p \le N \frac{1}{\omega(Q_\alpha^n)} \|f\|_{L_p(w \, d\mu)}^p.
$$
By Lemma \ref{lem0319_1}, $\omega(Q_\alpha^n) \to \infty$ as $n \to -\infty$ when $\mu(\cX) = \infty$. Thus \eqref{eq0402_1} follows. The second statement is clear from Lemma \ref{lem0319_1}.
\end{proof}

\begin{lemma}
							\label{lem0401_1}
Let $p \in (1,\infty)$, $w \in A_p$, and $\alpha = 1/(2N_1)$, where $N_1$ is given in \eqref{eq4.57}.
For a nonnegative $f \in L_p(w \, d \mu)$, set
\begin{equation}
							\label{eq0401_1}
\lambda_0 =
\left\{\begin{aligned}
2 N_1 \|f\|_{L_1(\mu)} \left(\mu(\cX)\right)^{-1},
\quad &\text{if} \quad \mu(\cX) < \infty,
\\
0, \quad &\text{if} \quad \mu(\cX) = \infty,
\end{aligned}
\right.
\end{equation}
and, for $\lambda > \lambda_0$,
define
$$
\tau(x) = \inf_{n \in \bZ} \{ n : f_{|n}(x) > \alpha \lambda \},
\quad
x \in \tilde{\cX}.
$$
Then $\tau$ is a stopping time.
The same statement holds true if $f \in L_1(\cX,\mu)$.
\end{lemma}

\begin{proof}
By Lemma \ref{lem0331_2}, if either $f \in L_p(w \, d\mu)$ with $w \in A_p$ or $f \in L_1(\cX,\mu)$, for each $x \in \tilde{\cX}$,
$$
\lim_{n \to -\infty} f_{|n}(x) < \alpha \lambda.
$$
This shows $\tau(x) > - \infty$.
Observe that
\begin{equation}
							\label{eq0319_1}
f_{|m}(x) = f_{|m}(y)
\quad
x , y \in Q_{\alpha}^n,
\quad
m \le n.
\end{equation}
For each $n \in \bZ$, since $\{x \in \tilde{\cX} : \tau(x) = n\} \subset \bigcup_\alpha Q_\alpha^n$, one can see that
$$
\{x \in \tilde{\cX} : \tau(x) = n\} = \bigcup_{\alpha\in I_n}\left(Q_{\alpha}^n \cap \{x \in \tilde{\cX} : \tau(x) = n\}\right)
\subset \bigcup_{\alpha'} \left( Q_{\alpha'}^n \cap \tilde{\cX} \right),
$$
where $\alpha'$ is such that
$Q_{\alpha'}^n \cap \{x \in \tilde{\cX} : \tau(x) = n\} \ne \emptyset$.
In this case by \eqref{eq0319_1}, for any $y \in Q_{\alpha'}^n \cap \tilde{\cX}$, we have $\tau(y) = \tau(x)$.
That is,
$$
Q_{\alpha'}^n \cap \tilde{\cX} \subset \{x \in \tilde{\cX} : \tau(x) = n\}.
$$
Hence,
$$
\{x \in \tilde{\cX} : \tau(x) = n\} = \bigcup_{\alpha'} \left( Q_{\alpha'}^n \cap \tilde{\cX} \right).
$$
Therefore, $\tau$ is a stopping time.
\end{proof}

The following level set estimate is a key lemma in the proof of Theorem \ref{thm2}.

\begin{lemma}
                                \label{lem2.2}
Let $p \in (1,\infty)$, $K_0 \ge 1$, $\alpha=1/(2N_1)$, where $N_1$ is defined in \eqref{eq4.57}, and $w\in A_p(\mu)$ be a weight on $\cX$ satisfying $[w]_{A_p}\le K_0$.
Recall $\omega(dx)=w\mu(dx)$.
Then for any $f\in L_p(w \, d \mu)$ and $\lambda> \lambda_0$, where $\lambda_0$ is defined in \eqref{eq0401_1}, we have
\begin{equation}
                                        \label{eq4.43}
\omega\{x\,:\,|f(x)|\ge \lambda\}\le N \lambda^{-\beta}\int_\cX I_{\cM_{\operatorname{dy}} f(x)>\alpha \lambda}(f_{\operatorname{dy}}^{\#}(x))^\beta\,\omega(dx),
\end{equation}
where $N>0$ and $\beta\in (0,1)$ are constants depending only on $K_1$, $K_2$, $p$, and $K_0$.
The estimate \eqref{eq4.43} also holds if $f \in L_1(\cX,\mu)$ and $\omega \in A_\infty(\mu)$.
In this case, the constants $\beta$ and $N$ depend on $K_1$, $K_2$, $p'$, and $[w]_{A_{p'}}$ for some $p' \in (1,\infty)$.
\end{lemma}

\begin{proof}
Throughout the proof, since $\mu(\cX \setminus \tilde{\cX}) = 0$, whenever needed, $\cX$ is to be understood as $\tilde{\cX}$.

First we assume that $f\ge 0$ in $\cX$.
Define  $\tau(x)=\inf\{n\,:\,f_{|n}(x)>\alpha \lambda\}$.
By Lemma \ref{lem0401_1}, $\tau=\tau(x)$ is a well-defined stopping time.
When $\tau(x)<\infty$, we have $f_{|\tau-1}(x)\le \alpha \lambda$, which together with \eqref{eq4.57} implies that $f_{|\tau}(x)\le \lambda/2$. By the Lebesgue differentiation theorem, $f_{|n}\to f$ $\mu$-a.e., which gives that for almost every $x$ satisfying $f(x)\ge \lambda$, we have $\tau<\infty$ and $f_{|\tau}(x)\le \lambda/2$. Since $\omega$ is absolutely continuous with respect to $\mu$, we get
\begin{align}
                        \label{eq11.15}
\omega\{x\,:\,f(x)\ge \lambda\}&=\omega\{x\,:\,f(x)\ge \lambda,f_{|\tau}(x)\le \lambda/2\}\nonumber\\
&\le \omega\{x\,:\,f(x)-f_{|\tau}(x)\ge \lambda/2\}.
\end{align}
Note that for $n\in \bZ$, $\{x\,:\,\tau(x)=n\}$ is a union of sets in $\bC_n$. Let $Q_\alpha^n$ be one of these sets so that $\tau(x)=n$ for any $x\in Q_\alpha^n$. Then by the Chebyshev inequality,
$$
\mu\{x\in Q_\alpha^n\,:\,f(x)-f_{|\tau}(x) \ge \lambda/2\}
\le 2\lambda^{-1}\int_{Q_\alpha^n}(f(x)-f_{|n}(x))_+\,\mu(dx),
$$
which implies that for any $y\in Q_\alpha^n$, we have
$$
\mu\{x\in Q_\alpha^n\,:\,f(x)-f_{|\tau}(x) \ge \lambda/2\}/\mu(Q_\alpha^n)
\le 2\lambda^{-1}f_{\text{dy}}^{\#}(y).
$$
By Lemma \ref{lem0331_1}, there exist $N$ and $\beta$ depending on $K_1$, $K_2$, $p$, and $K_0$  such that
\begin{align*}
&\frac{\omega\{x\in Q_\alpha^n\,:\,f(x)-f_{|\tau}(x) \ge \lambda/2\}}{\omega(Q_\alpha^n)}\\
&\le N\left(\frac{\mu\{x\in Q_\alpha^n\,:\,f(x)-f_{|\tau}(x) \ge \lambda/2\}}{\mu(Q_\alpha^n)}\right)^\beta
\le N\lambda^{-\beta}(f_{\text{dy}}^{\#}(y))^\beta,
\end{align*}
which gives
\begin{equation}
                                                    \label{eq11.27}
\omega\{x\in Q_\alpha^n\,:\,f(x)-f_{|\tau}(x)\ge \lambda/2\}
\le N\lambda^{-\beta}\int_{Q_\alpha^n}(f_{\text{dy}}^{\#}(x))^\beta\,\omega(dx).
\end{equation}
Observe that
$$
\{x\,:\,\tau(x)<\infty\}=\{x\,:\,\cM_{\text{dy}}f(x)>\alpha\lambda\}.
$$
Summing \eqref{eq11.27} with respect to all such sets $Q_\alpha^n$ and using \eqref{eq11.15}, we reach \eqref{eq4.43}. Finally, to remove the condition that $f\ge 0$, it suffices to note that by the triangle inequality, $|f|^{\#}_{\text{dy}}\le 2f^{\#}_{\text{dy}}$. The estimate \eqref{eq4.43} is proved.

In the case that $f \in L_1(\cX,\mu)$ and $\omega \in A_\infty(\mu)$, the proof is the same with a possibly different choice of $\beta$ and $N$ in \eqref{eq11.27} due to $w \in A_\infty$, i.e., there exists $p' \in (1,\infty)$ such that $w \in A_{p'}(\mu)$.
\end{proof}

We are now ready to prove Theorem \ref{thm2}.

\begin{proof}[Proof of Theorem \ref{thm2}]
First we assume that $q=p$.
In this case by Lemma \ref{lem0320_1} $f \in L_1(\mu)$ if $\mu(\cX) < \infty$.
As before, we denote $\omega(dx)=w\mu(dx)$ and $L_p(\omega) := L_p(\cX,w\, d\mu) = L_p(w \, d\mu)$.
Recall the elementary identity:
$$
\|f\|_{L_p(\omega)}^p=p\int_0^\infty \omega\{x\,:\,|f(x)|>\lambda\}\lambda^{p-1}\,d\lambda.
$$
To estimate $\omega\{x\,:\,|f(x)|>\lambda\}$, for $\lambda> \lambda_0$, we use Lemma \ref{lem2.2}. Otherwise, we use the simple upper bound
$$
\omega\{x\,:\,|f(x)|>\lambda\}\le \omega(\operatorname{supp} f).
$$
We then get
\begin{align}
                            \label{eq3.06}
&\|f\|_{L_p(\omega)}^p=p\left(\int_0^{\lambda_0}+\int_{\lambda_0}^\infty\right) \omega\{x\,:\,|f(x)|>\lambda\}\lambda^{p-1}\,d\lambda\nonumber\\
&\le \lambda_0^p \, \omega(\operatorname{supp} f)+Np\int_{\lambda_0}^\infty\left(\int_\cX I_{\cM_{\text{dy}} f(x)>\alpha \lambda}(f_{\text{dy}}^{\#}(x))^\beta\,\omega(dx)\right)\lambda^{p-1-\beta}\,d\lambda.
\end{align}
Clearly, the first term on the right-hand side of \eqref{eq3.06} vanishes, i.e., $\lambda_0 = 0$ when $\mu(\cX)=\infty$.
By Fubini's theorem and H\"older's inequality,
\begin{align}
                            \label{eq3.27}
&\int_{\lambda_0}^\infty\left(\int_\cX I_{\cM_{\text{dy}} f(x)>\alpha \lambda}(f_{\text{dy}}^{\#}(x))^\beta\,\omega(dx)\right)
\lambda^{p-1-\beta}\,d\lambda\nonumber\\
&\le\int_\cX \left(\int_0^{\alpha^{-1}\cM_{\text{dy}} f(x)}\lambda^{p-1-\beta}\,d\lambda\right) (f_{\text{dy}}^{\#}(x))^\beta\,\omega(dx)\nonumber\\
&\le N\int_\cX (\cM_{\text{dy}} f(x))^{p-\beta}(f_{\text{dy}}^{\#}(x))^\beta\,\omega(dx)\nonumber\\
&\le N\|\cM_{\text{dy}} f\|_{L_p(\omega)}^{p-\beta}
\|f_{\text{dy}}^{\#}\|_{L_p(\omega)}^{\beta}.
\end{align}
Combining \eqref{eq3.06}, the definition of $\lambda_0$, and \eqref{eq3.27}, we obtain
\begin{equation}
							\label{eq0322_1}
\|f\|_{L_p(\omega)}^p\le N\|f\|_{L_1(\mu)}^p \left(\mu(\cX)\right)^{-p}\omega(\operatorname{supp} f)+N\|\cM_{\text{dy}} f\|_{L_p(\omega)}^{p-\beta}
\|f_{\text{dy}}^{\#}\|_{L_p(\omega)}^{\beta},
\end{equation}
where the first-term on the right-hand side again vanishes when $\mu(\cX) = \infty$, and $N=N(K_1,K_2,p, K_0)$.
By Theorem \ref{thm1}, \eqref{eq8.52}, and Young's inequality, the second term on the right-hand side of \eqref{eq0322_1} is bounded as
$$
\|\cM_{\text{dy}} f\|_{L_p(\omega)}^{p-\beta}
\|f_{\text{dy}}^{\#}\|_{L_p(\omega)}^{\beta} \le \varepsilon_1 \|f\|_{L_p(\omega)}^p + \varepsilon_1^{-(p-\beta)/\beta} N \|f_{\text{dy}}^{\#}\|_{L_p(\omega)}^p
$$
for any $\varepsilon_1 > 0$.
From this combined with \eqref{eq0322_1}, we see that \eqref{eq1.51} and \eqref{eq1.54} hold true for $\mu(\cX) < \infty$ and $\mu(\cX) =\infty$, respectively.

If we assume that $\operatorname{supp} f \subset B_{r_0}(x_0)$ and $\mu(B_{r_0}(x_0))\le \varepsilon\mu(\cX)$, then by the definition of $\lambda_0$ and Lemma \ref{lem0320_1},
\begin{align*}
&\lambda_0^p \, \omega(\operatorname{supp}f)\le \lambda_0^p \, \omega(B_{r_0}(x_0))\le N\|f\|^p_{L_1(\mu)}\omega(B_{r_0}(x_0))(\mu(\cX))^{-p}\nonumber\\
&\le N\|f\|^p_{L_p(\omega)}\left(\mu(B_{r_0}(x_0))\right)^{p}(\mu(\cX))^{-p}
\le N\varepsilon^p \|f\|^p_{L_p(\omega)},
\end{align*}
which together with \eqref{eq3.06} and \eqref{eq3.27} yields
$$
\|f\|_{L_p(\omega)}^p\le N\varepsilon^p \|f\|^p_{L_p(\omega)} +N\|\cM_{\text{dy}} f\|_{L_p(\omega)}^{p-\beta}
\|f_{\text{dy}}^{\#}\|_{L_p(\omega)}^{\beta}.
$$
Again by  Theorem \ref{thm1} and \eqref{eq8.52}, we get \eqref{eq3.41}.

Next, for general $q \in (1,\infty)$, we only need to consider the case when $q>p$ because otherwise $A_q\subset A_p$ and the result follows from the proof above.
Note that, if $q > p$ and $\mu(\cX) < \infty$, Lemma \ref{lem0320_1} only implies $|f|^{p/q} \in L_1(\mu)$, but still the inequality \eqref{eq1.51} makes sense regardless of whether $f \in L_1(\mu)$ or not.
Observe that by the triangle inequality and H\"older's inequality,
\begin{align*}
&\dashint_{Q_\alpha^n}\big||f(x)|^{p/q}-(|f|^{p/q})_{Q_\alpha^n}\big|\,\mu(dx)\\
&\le \dashint_{Q_\alpha^n}\dashint_{Q_\alpha^n}
\big||f(x)|^{p/q}-|f(y)|^{p/q}\big|\,\mu(dx)\,\mu(dy)\\
&\le \dashint_{Q_\alpha^n}\dashint_{Q_\alpha^n}
\big||f(x)|-|f(y)|\big|^{p/q}\,\mu(dx)\,\mu(dy)\\
&\le \left(\dashint_{Q_\alpha^n}\dashint_{Q_\alpha^n}
|f(x)-f(y)|\,\mu(dx)\,\mu(dy)\right)^{p/q}\\
&\le \left(2\dashint_{Q_\alpha^n}
\big|f(x)-(f)_{Q_\alpha^n}\big|\,\mu(dx)\right)^{p/q},
\end{align*}
which implies that
\begin{equation}
                            \label{eq3.00}
\| (|f|^{p/q})^{\#}_{\text{dy}}\|_{L_q(\omega)}^{q/p}
\le N\|f^{\#}_{\text{dy}}\|_{L_p(\omega)}.
\end{equation}
Using \eqref{eq1.51} and \eqref{eq1.54} with $q$ in place of $p$, \eqref{eq3.00}, and H\"older's inequality, we get
\begin{align*}
&\|f\|_{L_p(\omega)}=\||f|^{p/q}\|_{L_q(\omega)}^{q/p}\\
&\le N\Big(
\| (|f|^{p/q})^{\#}_{\text{dy}}\|_{L_q(\omega)}
+(\mu(\cX))^{-1}\big(\omega(\operatorname{supp} f)\big)^{\frac 1 q}\||f|^{p/q}\|_{L_1(\mu)}\Big)^{q/p}\\
&\le N
\| (|f|^{p/q})^{\#}_{\text{dy}}\|_{L_q(\omega)}^{q/p}
+N(\mu(\cX))^{-q/p}\big(\omega(\operatorname{supp} f)\big)^{\frac 1 p}\||f|^{p/q}\|_{L_1(\mu)}^{q/p}\\
&\le N\|f^{\#}_{\text{dy}}\|_{L_p(\omega)}
+N(\mu(\cX))^{-1}\big(\omega(\operatorname{supp} f)\big)^{\frac 1 p}\|f\|_{L_1(\mu)},
\end{align*}
where, as above, the second term on the right-hand side is to vanish if $\mu(\cX) = \infty$. This gives \eqref{eq1.51}  and \eqref{eq1.54} for $\mu(\cX) < \infty$ and $\mu(\cX) = \infty$, respectively.
The proof of \eqref{eq3.41} is similar.
The theorem is proved.
\end{proof}

\begin{proof}[Proof of Theorem \ref{thm3}]
Without loss of generality, we may assume that $f\ge 0$.
First assume $q = p$ and let $\lambda_0$ be as in \eqref{eq0401_1} using $v$ in place of $f$. Similar to the proof of Theorem \ref{thm2}, for $\lambda>\lambda_0$ we define a stopping time  $\tau(x)=\inf\{n\,:\,v_{|n}(x)>\alpha \lambda\}$. Because $v\ge f$, as \eqref{eq11.15}, we have
\begin{align}
                                    \label{eq4.56}
\omega\{x\,:\,f(x)\ge \lambda\}&=\omega\{x\,:\,f(x)\ge \lambda,v_{|\tau}(x)\le \lambda/2\}\nonumber\\
&\le \omega\{x\,:\,f(x)-v_{|\tau}(x)\ge \lambda/2\}.
\end{align}
Let $Q:=Q_\alpha^n\in \bC_n$ be such that $\tau(x)=n$ for any $x\in Q_\alpha^n$. Since $f\le f^Q\le v$, by the Chebyshev inequality we have
\begin{align*}
&\mu\{x\in Q\,:\,f(x)-v_{|\tau}(x) \ge \lambda/2\}\le \mu\{x\in Q\,:\,f^Q(x)-f^Q_{|\tau}(x) \ge \lambda/2\}\\
&\le 2\lambda^{-1}\int_{Q}(f^Q(x)-(f^Q)_Q)_+\,\mu(dx)\le 2\lambda^{-1}g(y) \mu(Q)
\end{align*}
for any $y\in Q$. Then similar to \eqref{eq11.27}, we have
\begin{equation}
                                    \label{eq4.58}
\omega\{x\in Q\,:\,f(x)-v_{|\tau}(x)\ge \lambda/2\}
\le N\lambda^{-\beta}\int_{Q}g^\beta(x)\,\omega(dx).
\end{equation}
Combining \eqref{eq4.56} and \eqref{eq4.58}, we get
$$
\omega\{x\,:\,f(x)\ge \lambda\}\le N\lambda^{-\beta}\int_\cX I_{\cM_{\text{dy}}v(x)>\alpha\lambda}g^\beta(x)\,\omega(dx).
$$
The remaining proof is the same as that of Theorem \ref{thm2}, and thus omitted.
\end{proof}

\section{Function spaces and notation}
							\label{sec3}

In this section we introduce some function spaces and notation  to be used throughout the rest of the paper.

For $T \in (-\infty,\infty]$, we set $
\bR_T = (-\infty,T)$ and a point in the Euclidean space $\bR_T \times \bR^d$ is denoted by $X = (t,x)$.
The Lebesgue measure for $\bR_T \times \bR^d$ is sometimes denoted by $dX$. We write $x=(x_1,\hat x)$, where $\hat{x} \in \bR^{d-1}$, and set
$$
\bR^d_+=\{x  = (x_1, \hat{x}) \in \bR^d\,:\,x_1>0, \hat{x} \in \bR^{d-1}\}.
$$
In the mixed-norm case, as before we write $x = (x',x'')$, where $x'=(x_1, \ldots, x_k)$ and $x'' = (x_{k+1}, \ldots, x_d)$ for some $k = 0, 1, 2, \ldots,d$.

For $m = 1, 2, \ldots$ fixed depending on the order of the equations/systems under consideration, we denote parabolic cylinders by
$$
Q_r(t,x) = (t-r^{2m}, t) \times B_r(x),
\quad
Q'_r(t,\hat x) = (t-r^{2m}, t) \times B'_r(\hat x),
$$
where
$$
B_r(x) = \{ y \in \bR^d : |x-y| < r \} \subset \bR^d,
\,\,
B'_r(\hat{x}) = \{ \hat y \in \bR^{d-1} : |\hat x-\hat y| < r \} \subset \bR^{d-1}.
$$
As usual, we use, for example, $Q_r$ to indicate $Q_r(0,0)$.
The parabolic distance between $X=(t,x)$ and $Y=(s,y)$ in $\bR^{d+1}$ is defined by $\rho(X,Y)=|x-y|+|t-s|^{\frac 1 {2m}}$.

For a function $g(t,x)$ defined on $\bR^{d+1}$, set
$$
\left[ g(t,\cdot)\right]_{B_r(x)}=
\dashint_{B_r(x)}
\left| g(t, y) -
\dashint_{B_r(x)}
g(t,z) \, dz \right|\, dy,
$$
$$
\left[ g(t,x_1,\cdot)\right]_{B'_r(\hat x)}=
\dashint_{B'_r(\hat x)}\left| g(t, x_1, \hat y) - \dashint_{B'_r(\hat x)} g(t,x_1, \hat z)\, d\hat z \right| \, d\hat y,
$$
$$
\left[ g(\cdot,x_1,\cdot)\right]_{Q'_r(t,\hat x)}=
\dashint_{Q'_r(t,\hat x)}\left| g(t, x_1, \hat y) -
\dashint_{Q'_r(t,\hat x)}g(s,x_1,\hat z) \,ds\,d\hat z \right| \, dt\,d\hat y.
$$
We define mean oscillations of $g$ on parabolic cylinders as follows.
First, we define the mean oscillation of $g$ in $Q_r(s,y)$ with respect to $x$ as
$$
\text{osc}_1
\left(g,Q_r(s,y)\right)
= \frac{1}{r^{2m}} \int_{s-r^{2m}}^{s}
\left[ g(\tau, \cdot)\right]_{B_r(y)} \, d\tau,
$$
and, for $R \in (0,\infty)$, denote
$$
g^{\#,1}_{R}
=\sup_{(s,y) \in \bR^{d+1}}\sup_{r\in (0, R]}\text{osc}_1 \left(g,Q_r(s,y)\right).
$$
Second, we define the mean oscillation of $g$ in $Q_r(s,y)$ with respect to $(t,\hat x)$ as
$$
\text{osc}_2
\left(g,Q_r(s,y)\right)
= \frac{1}{2r} \int_{y_1-r}^{y_1+r}
\left[ g(\cdot,z_1, \cdot)\right]_{Q_r'(s,\hat y)} \, dz_1,
$$
and denote
$$
g^{\#,2}_R
= \sup_{(s,y) \in \bR^{d+1}}\sup_{r\in (0, R]}\text{osc}_2 \left(g,Q_r(s,y)\right).
$$
Third, we define the mean oscillation of $g$ in $Q_r(s,y)$ with respect to $\hat x$ as
$$
\text{osc}
\left(g,Q_r(s,y)\right)
=  \frac{1}{2 r^{2m+1}}\int_{s-r^{2m}}^{s}
\int_{y_1-r}^{y_1+r}
\left[ g(\tau,z_1,\cdot)\right]_{B'_r(\hat y)} \, d z_1 \, d\tau,
$$
and denote
$$
g^{\#}_R =\sup_{(s,y) \in \bR^{d+1}}\sup_{r\in (0, R]}\text{osc}\left(g,Q_r(s,y)\right).
$$
Finally, in the case when $g$ is independent of $t$,
i.e., if $g$ is a function of $x \in \bR^d$, we set
$$
\text{osc} \left(g,B_r(y)\right)
= \frac{1}{2r} \int_{y_1-r}^{y_1+r}
\left[ g(z_1,\cdot) \right]_{B'_r(\hat y)} \, dz_1,
$$
$$
g^{\#}_R=\sup_{y \in \bR^d}\sup_{r\in (0, R]} \text{osc}
\left(g,B_r(y)\right).
$$

Next, we introduce some function spaces to be used when dealing with elliptic and parabolic equations/systems.
The domains are subsets of $\bR^d$ in the elliptic case and those of $\bR_T \times \bR^d$ in the parabolic case.
Since we use the results from Section \ref{sec1} in the later sections, we note that whenever a domain in $\bR^d$ or in $\bR_T \times \bR^d$ is considered, the underlying measure is the Lebesgue measure. The metric is the usual Euclidean distance in the elliptic case and the parabolic distance in the parabolic case.

We use the following weighted Sobolev spaces.
$$
W_{p,w}^k(\Omega) = W_p^k(\Omega, w \, d x) = \{ u : u, Du, \ldots, D^k u \in L_p(\Omega, w \, dx) \}, \quad k = 1, 2, \ldots,
$$
where $\Omega \subset \bR^d$ and $w \in A_p(\Omega, dx)$.
Note that, because of the underlying measure and the metric, the elements of $A_p(\Omega, dx)$ are determined by using open balls in $\Omega$, which are of the form $B_r(x) \cap \Omega$, $x \in \Omega$.
Naturally, we denote
$$
L_{p,w}(\Omega) := L_p(\Omega, w \, dx).
$$
For parabolic systems/equations,
we have
$$
W_{p,w}^{1,k}\left( (S,T) \times \Omega\right) = \{ u : u, Du, \ldots, D^k u, u_t \in L_{p,w}\left((S,T) \times \Omega\right) \}, \quad k = 1, 2, \ldots,
$$
where $-\infty \le S < \infty$, $-\infty < T \le \infty$, $\Omega \subset \bR^d$, and $w \in A_p\left((S,T) \times \Omega, dx \, dt\right)$.
Also note that, when determining elements of $A_p\left((S,T) \times \Omega, dx \, dt\right)$, we use balls in $(S,T) \times \Omega$ with respect to the parabolic distance $|x-y|+|t-s|^{\frac{1}{2m}}$.

For mixed norms,  let
$\Omega_1$ and $\Omega_2$ be open sets in $\bR^{d_1}$ and $\bR^{d_2}$, $d_1+d_2=d$, respectively, and $\Omega \subset \Omega_1 \times \Omega_2$.
Denote $\cX' = \Omega_1$ and $\cX'' = (S,T) \times \Omega_2$.
Set
$$
w(t, x',x'') = w_1(x') w_2(t,x''),
\quad
x' \in \Omega_1,
\,\,
x'' \in \Omega_2,
$$
where
\begin{equation}
							\label{eq0604_01}
w_1 \in A_p(\Omega_1, d x'),
\quad
w_2 \in A_q\left((S,T) \times \Omega_2, dx'' \, dt \right).
\end{equation}
For $k = 1, 2, \ldots$, we define
$$
u \in W^{1,k}_{p,q,w}\left((S,T) \times\Omega\right)
\,\,
\iff
\,\,
u, Du, \ldots, D^k u, u_t \in L_{p,q,w}\left((S,T)\times\Omega\right).
$$
Note that, by the choices of $w_1$ and $w_2$ in \eqref{eq0604_01}, by $f \in L_{p,q,w}\left((S,T)\times\Omega\right)$ we mean
$$
\|f\|_{L_{p,q,w}\left((S,T)\times\Omega\right)} = \left(\int_{(S,T) \times \Omega_2} \left( \int_{\Omega_1} |f|^p I_{\Omega} w_1(x') \, dx' \right)^{q/p} w_2(t,x'') \, dx'' \, dt \right)^{1/q}.
$$
Certainly, one can choose $w_1$ and $w_2$ differently so that the integral with respect to time is included in the inner integral.
In the elliptic case, we just remove (the integral in) the time variable, i.e.,
$$
u \in W_{p,q,w}^k(\Omega)
\,\,
\iff
\,\,
u, Du, \ldots, D^ku \in L_{p,q,w}(\Omega),
$$
where
$$
w(x',x'') = w_1(x') w_2(x''),
\quad
w_1 \in A_p(\Omega_1, dx'),
\quad
w_2 \in A_q (\Omega_2, dx''),
$$
and
$$
\|f\|_{L_{p,q,w}(\Omega)} = \left(\int_{\Omega_2} \left( \int_{\Omega_1} |f|^p I_\Omega w_1(x') \, dx' \right)^{q/p} w_2(x'') \, dx'' \right)^{1/q}.
$$
To deal with divergence type equations/systems, we set
$$
\cH_{p,q, w}^m ((S,T) \times \Omega)
$$
$$
= \{ u: u_t \in \bH_{p,q,w}^{-m} ((S,T) \times \Omega), D^\alpha u \in L_{p,q,w}((S,T) \times \Omega), 0 \le |\alpha| \le m \},
$$
$$
\|u\|_{\cH_{p,q, w}^m ((S,T) \times \Omega)} = \|u_t\|_{\bH_{p,q,w}^{-m}((S,T) \times \Omega)} + \sum_{|\alpha| \le m} \|D^\alpha u\|_{L_{p,q,w}(S,T) \times \Omega)},
$$
where
$$
\bH_{p,q,w}^{-m}((S,T) \times \Omega) = \{ f : f = \sum_{|\alpha| \le m} D^\alpha f_\alpha, \,\, f_\alpha \in L_{p,q,w}((S,T) \times \Omega)\},
$$
$$
\|f\|_{\bH_{p,q,w}^{-m}((S,T) \times \Omega)} = \inf \{ \sum_{|\alpha| \le m} \|f_\alpha\|_{L_{p,q,w} ((S,T) \times \Omega)} : f = \sum_{|\alpha| \le m} D^\alpha f_\alpha \}.
$$
We denote by $\mathring{\cH}_{p,q,w}^m ((S,T) \times \Omega)$ the closure of $C_0^\infty([S,T] \times \Omega)$ in $\cH_{p,q,w}^m((S,T) \times \Omega)$.

\section{Higher order parabolic systems in non-divergence form with BMO coefficients: Mixed norm}
							\label{sec4}

In this section, we consider higher order parabolic systems with leading coefficients merely measurable in the time variable and having small mean oscillations in the spatial variable in small cylinders. We shall generalize some results in \cite{MR2771670} to the case of mixed-norm Lebesgue spaces with $A_p$ weights.

Set
\begin{equation}
							\label{eq0528_01}
Lu = \sum_{|\alpha| \le m, |\beta| \le m} A^{\alpha\beta} D^\alpha D^\beta u,
\end{equation}
where $m$ is a positive integer,
$$
D^\alpha = D_1^{\alpha_1} \ldots D_d^{\alpha_d},
\quad
\alpha = (\alpha_1, \ldots, \alpha_d),
$$
and, for multi-indices $\alpha$ and $\beta$,
the coefficient $A^{\alpha\beta} = [ A_{ij}^{\alpha\beta} ]_{i,j=1}^{\ell}$ is an $\ell \times \ell$ complex matrix-valued function defined on $\bR^{d+1}$.
The involved functions are complex vector-valued functions, that is
$$
u = (u^1, \ldots, u^\ell)^{\text{tr}},
\quad
f = (f^1, \ldots, u^\ell)^{\text{tr}}.
$$
The parabolic and elliptic systems we consider are
$$
u_t + (-1)^m L u = f\quad\text{and}
\quad
Lu = f,
$$
where, for the elliptic case, the coefficient matrices are functions independent of $t \in \bR$, i.e., defined on $\bR^d$.

Throughout the section we assume that coefficients are bounded and satisfy the Legendre-Hadamard ellipticity condition, i.e., there exist $\delta \in (0,1)$ and $K > 0$ such that
$$
\Re \left( \sum_{|\alpha|=|\beta|=m} \theta^{\text{tr}} \xi^\alpha \xi^\beta A^{\alpha\beta}(t,x) \bar{\theta} \right) \ge \delta |\xi|^{2m}|\theta|^2,
$$
for all $\xi \in \bR^d$ and $\theta\in \bC^{\ell}$, and
$$
|A^{\alpha\beta}| \le \left\{
\begin{aligned}
\delta^{-1},&
\quad
|\alpha|=|\beta| = m,
\\
K,&
\quad
\text{otherwise}.
\end{aligned}
\right.
$$
Note that this condition is weaker than the usual strong ellipticity condition.

We impose the following regularity condition on the leading coefficients, where $\gamma$ is a parameter to be specified.

\begin{assumption}[$\gamma$]
							\label{assum0528_2}
Let $\gamma \in (0,1)$.
There exists $R_0 \in (0,\infty)$ such that
$$
\sum_{|\alpha| = |\beta| = m} (A^{\alpha\beta})^{\#,1}_{R_0} \le \gamma.
$$
\end{assumption}

Next we state the main results of this section.
Note that in the theorem below the Euclidean space $\bR^{d+1}$ satisfies the conditions (a) and (b) before Corollary \ref{cor1} as the product space of $\bR^{d_1}$ and $\bR \times \bR^{d_2}$.
So does $\bR \times \bR^d_+ = \{(t,x): t \in \bR, x \in \bR^d_+\}$ in Theorem \ref{thm4.4}, as the product space of $\bR^{d_1}_+$ and $\bR \times \bR^{d_2}$ or $\bR^{d_1}$ and $\bR \times \bR^{d_2}_+$.

\begin{theorem}[The whole space case]
							\label{thm0607_1}
Let $p, q \in (1,\infty)$, $K_0\ge 1$ be constant, $w = w_1(x') w_2(t,x'')$, where
$$
w_1(x') \in A_p(\bR^{d_1}, dx'),
\quad
w_2(t,x'') \in A_q(\bR \times \bR^{d_2}, dx'' \, dt),
$$
$$
d_1 + d_2 = d,\quad
[w_1]_{A_p}\le K_0,\quad [w_2]_{A_q}\le K_0,
$$
and let $L$ be the operator in \eqref{eq0528_01}.
Then there exist
\begin{align*}
\gamma &= \gamma(d, m, \ell, \delta, p, q, d_1, d_2, K_0) \in (0,1),\\
\lambda_0 &= \lambda_0(d, m, \ell, \delta, p, q, d_1, d_2, K_0, K, R_0) \ge 1,
\end{align*}
such that, under Assumption \ref{assum0528_2} ($\gamma$), for $u \in W_{p,q,w}^{1,2m}(\bR^{d+1})$ satisfying
\begin{equation}
							\label{eq0604_06}
u_t + (-1)^m Lu + \lambda u = f
\end{equation}
in $\bR^{d+1}$, where $f \in L_{p,q,w}(\bR^{d+1})$,
we have
\begin{equation*}
\|u_t\|_{L_{p,q,w}}+\sum_{|\alpha| \le 2m} \lambda^{1-\frac{|\alpha|}{2m}} \| D^\alpha u \|_{L_{p,q,w}} \le
N \|f\|_{L_{p,q,w}},
\end{equation*}
provided that $\lambda \ge \lambda_0$,
where $L_{p,q,w} = L_{p,q,w}(\bR^{d+1})$ and
$$
N = N(d,m,\ell,\delta, p, q, d_1, d_2, K_0, K, R_0).
$$
\end{theorem}

\begin{remark}
By setting $d_1 = 0$, Theorem \ref{thm0607_1} is reduced to the case of unmixed $L_p$ spaces with $A_p$ weights.
Moreover, it generalizes the main result of \cite[Chap. 7]{MR2435520} in which $m=1$ and the restriction $q\ge p$ is imposed.
\end{remark}

\begin{theorem}[The half space case]
                                            \label{thm4.4}
The result in Theorem \ref{thm0607_1} still holds if we replace $\bR^{d_1}$ (or $\bR^{d_2}$) by $\bR^{d_1}_+$ (or $\bR^{d_2}_+$, respectively) and impose the Dirichlet boundary condition
\begin{equation}
                                        \label{eq1.36}
u=Du=\cdots=D^{m-1}u=0
\end{equation}
on the lateral boundary of the cylindrical domain.
\end{theorem}

Now we turn our attention to elliptic systems.
A priori estimates for the elliptic system in \eqref{eq0804_01} defined in the whole spaces and on a half space are derived by using the corresponding estimates in Theorems \ref{thm0607_1} and \ref{thm4.4} for the parabolic system and the argument, for instance, in the proof of \cite[Theorem 2.6]{MR2304157}.
The key idea is that one can view an elliptic system as a steady state parabolic system.
Instead of stating all possible results for elliptic equations/systems in this section or in the later sections,
we here present only the elliptic version of Theorem \ref{thm0607_1}.
Recall that the coefficients $A^{\alpha\beta}$ are now independent of $t$.

\begin{theorem}
							\label{thm0807_1}
Let $p, q \in (1,\infty)$, $K_0\ge 1$ be constant, $w = w_1(x') w_2(x'')$, where
$$
w_1(x') \in A_p(\bR^{d_1}, dx'),
\quad
w_2(x'') \in A_q(\bR^{d_2}, dx''),
$$
$$
d_1 + d_2 = d,\quad
[w_1]_{A_p}\le K_0,\quad [w_2]_{A_q}\le K_0,
$$
and let $L$ be the operator in \eqref{eq0528_01}.
Then there exist
\begin{align*}
\gamma &= \gamma(d, m, \ell, \delta, p, q, d_1, d_2, K_0) \in (0,1),\\
\lambda_0 &= \lambda_0(d, m, \ell, \delta, p, q, d_1, d_2, K_0, K, R_0) \ge 1,
\end{align*}
such that, under Assumption \ref{assum0528_2} ($\gamma$), for $u \in W_{p,q,w}^{2m}(\bR^d)$ satisfying
\begin{equation}
							\label{eq0804_01}
Lu + (-1)^m \lambda u = f
\end{equation}
in $\bR^d$, where $f \in L_{p,q,w}(\bR^d)$,
we have
\begin{equation*}
\sum_{|\alpha| \le 2m} \lambda^{1-\frac{|\alpha|}{2m}} \| D^\alpha u \|_{L_{p,q,w}} \le
N \|f\|_{L_{p,q,w}},
\end{equation*}
provided that $\lambda \ge \lambda_0$,
where $L_{p,q,w} = L_{p,q,w}(\bR^d)$ and
$$
N = N(d,m,\ell,\delta, p, q, d_1, d_2, K_0, K, R_0).
$$
\end{theorem}

\begin{proof}
Choose $\zeta \in C_0^\infty(\bR)$ and set $v(t,x) = \zeta(t/n) u(x)$, $n \in \bZ$, which satisfies, in $\bR^{d+1}$,
\begin{equation*}
v_t + (-1)^m Lv + \lambda v = \frac{1}{n} \zeta_t(t/n) u(x) + (-1)^m\zeta(t/n)f.
\end{equation*}
Upon noting that $w_2 \in A_q(\bR \times \bR^{d_2}, dx'' dt)$,
we apply Theorem \ref{thm0607_1} with $w$ to the above system and follow the proof of \cite[Theorem 2.6]{MR2304157} with obvious modifications.
\end{proof}

\subsection{Mixed-norm estimate in the whole space}

This subsection is devoted to the proof of Theorem \ref{thm0607_1}. Throughout the paper, we use the notation
$$
(g)_\cD := \dashint_\cD g(t,x) \, dx \, dt,
$$
where $\cD$ is a subset in $\bR^{d+1}$.

We begin with the following interior estimate.
\begin{lemma}
                        \label{lem4.4}
Let $\lambda \ge 0$, $q \in (1,\infty)$, and $L$ be the operator in \eqref{eq0528_01}.
Suppose that the coefficients $A^{\alpha\beta}$, $|\alpha|=|\beta|=m$, are measurable functions of only $t \in \bR$, i.e., $A^{\alpha\beta} = A^{\alpha\beta}(t)$ and the lower-order coefficients of $L$ are all zero. Then for any $u \in C_{\operatorname{loc}}^\infty(\bR^{d+1})$ satisfying \eqref{eq0604_06} in $Q_4$ with $f=0$, we have
\begin{equation}
							\label{eq11.07}
\sum_{|\alpha| \le 2m} \lambda^{1-\frac{|\alpha|}{2m}} \sup_{Q_1}\big(|D(D^\alpha u)|+|(D^\alpha u)_t|\big)
\le N \sum_{|\alpha| \le 2m} \lambda^{1-\frac{|\alpha|}{2m}}\|D^\alpha u\|_{L_q(Q_4)},
\end{equation}
where $N = N(d,m,\ell,\delta,q)$.
\end{lemma}
\begin{proof}
The special case of the lemma when $q=2$ was proved in \cite[Lemma 3]{MR2771670}, which was derived from Lemma 2 there. The latter still holds with $q$ in place of $2$ thanks to the $W^{1,2m}_q$ estimates established in the same paper. See the proof of Theorem 2 there.
In fact, the estimate in \cite[Lemma 3]{MR2771670} only contains the highest and lowest order terms because we collected only those terms when applying S. Agmon's idea.
By collecting the intermediate order terms as well, we obtain \eqref{eq11.07}.
\end{proof}

\begin{lemma}
							\label{lem0528_01}
Let $\kappa \ge 8$. Under the assumptions of Lemma \ref{lem4.4}, for any $r \in (0,\infty)$, $X_0 \in \bR^{d+1}$, and $u \in W_{q, \operatorname{loc}}^{1,2m}(\bR^{d+1})$ satisfying \eqref{eq0604_06} in $Q_{\kappa r}(X_0)$,
where $f \in L_{q,\operatorname{loc}}(\bR^{d+1})$,
we have
\begin{multline}
							\label{eq0607_01}
\sum_{|\alpha| \le 2m} \lambda^{1-\frac{|\alpha|}{2m}} \left( |D^\alpha u - \left(D^\alpha u\right)_{Q_r(X_0)} |\right)_{Q_r(X_0)}
\\
\le N \kappa^{-1} \sum_{|\alpha| \le 2m} \lambda^{1-\frac{|\alpha|}{2m}}\left( |D^\alpha u|^q \right)_{Q_{\kappa r}(X_0)}^{\frac 1 q} +  N \kappa^{\frac{d+2m}{q}} \left( |f|^q\right)_{Q_{\kappa r}(X_0)}^{\frac 1 q},
\end{multline}
where $N = N(d,m,\ell,\delta,q)$.
\end{lemma}

\begin{proof}
Similar to \cite[Corollary 2]{MR2771670}, the lemma can be derived from Lemma \ref{lem4.4} and the aforementioned $W^{1,2m}_q$ estimates.
In particular, we have \cite[Theorem 10]{MR2771670} with $q$ in place of $2$.
Thus, as described in the proof of \cite[Theorem 2]{MR2771670},
we obtain \eqref{eq0607_01}
by repeating the proof of \cite[Corollary 2]{MR2771670} with $q$ in place of $2$.
While following the arguments in \cite{MR2771670}, due to the standard approximation argument, we may assume that the coefficients are infinitely differentiable. See Remark \ref{rem0713_1}.
\end{proof}

\begin{remark}
							\label{rem0713_1}
When proving mean oscillation estimates as in Lemmas \ref{lem0528_01} and \ref{lem0604_01}, one can always assume that the coefficients are infinitely differentiable.
For example, for the operator in \eqref{eq0528_01}, we set
$$
L_n u  = \sum_{|\alpha| \le m, |\beta| \le m} A^{\alpha\beta}_n D^\alpha D^\beta u,
$$
where $A^{\alpha\beta}_n$ is the mollification of $A^{\alpha\beta}$ in $\bR^{d+1}$, and prove the mean oscillation estimate in \eqref{eq0607_01} with $L_n$ in place of $L$.
Then we let $n \to \infty$ to obtain the desired estimate for $L$ with a constant $N$ independent of the approximation.
This type of argument is used throughout the paper because in the proofs of mean oscillation estimates, we always split $u$ as the sum $u = w + v$, where $v$ is a solution to the given system/equation with right-hand side being zero on a cylinder or a cylinder intersected with a half space.
Under the assumption that the coefficients are infinitely differentiable, using the classical results, we see that $v$ is infinitely differentiable inside the cylinder, so we are able to use Lipschitz estimates or H\"{o}lder estimates as in Lemmas \ref{lem4.4} and \ref{lem4.10}  with $v$ in place of $u$.
\end{remark}

For coefficients $A^{\alpha\beta}(t,x)$ satisfying Assumption  \ref{assum0528_2}, we obtain the following mean oscillation estimate on $Q_r(X_0)$ when $r$ is bounded above.

\begin{lemma}
							\label{lem0604_01}
Let $\lambda \ge 0$, $q \in (1,\infty)$, $\kappa \ge 8$, $\mu, \nu \in (1,\infty)$, $1/\mu+1/\nu = 1$, and $L$ be the operator in \eqref{eq0528_01}.
Suppose that the lower-order coefficients of $L$ are all zero.
Then, under Assumption \ref{assum0528_2} ($\gamma$),
for $r \in (0,R_0/\kappa]$, $X_0 \in \bR^{d+1}$, and $u \in W_{q\mu, \operatorname{loc}}^{1,2m}(\bR^{d+1})$ satisfying \eqref{eq0604_06} in $Q_{\kappa r}(X_0)$, where $f \in L_{q,\operatorname{loc}}(\bR^{d+1})$,
we have
\begin{align*}
&\sum_{|\alpha| \le 2m} \lambda^{1-\frac{|\alpha|}{2m}} \left( |D^\alpha u - \left(D^\alpha u\right)_{Q_r(X_0)} |\right)_{Q_r(X_0)}\\
&\le N \kappa^{-1} \sum_{|\alpha| \le 2m} \lambda^{1-\frac{|\alpha|}{2m}} \left( |D^\alpha u|^q \right)_{Q_{\kappa r}(X_0)}^{\frac 1 q}
+ N \kappa^{\frac{d+2m}{q}} \left( |f|^q\right)_{Q_{\kappa r}(X_0)}^{\frac 1 q}\\
&\quad + N \kappa^{\frac{d+2m}{q}} \gamma^{\frac{1}{q\nu}} \left( |D^{2m}u|^{q\mu}\right)_{Q_{\kappa r}(X_0)}^{\frac{1}{q\mu}},
\end{align*}
where $N = N(d,m,\ell,\delta,q,\mu)$.
\end{lemma}

\begin{proof}
Set
$$
L_0 u = \sum_{|\alpha| \le m, |\beta| \le m} \bar{A}^{\alpha\beta}(t) D^\alpha D^\beta u,
$$
where
$$
\bar{A}^{\alpha\beta}(t) = \dashint_{B_{\kappa r}(x_0)} A^{\alpha\beta}(t,y) \, dy.
$$
Then we see that $u \in W_{q,\text{loc}}^{1,2m}(\bR^{d+1})$ satisfies
$$
u_t + (-1)^m L_0 u + \lambda u = \bar{f}
$$
in $Q_{\kappa r}(X_0)$, where
$$
\bar{f} := f + (-1)^m(L_0 - L)u \in L_{q,\text{loc}}(\bR^{d+1}).
$$
Then by Lemma \ref{lem0528_01}, we obtain \eqref{eq0607_01} with $\bar{f}$ in place of $f$.
Note that
\begin{align*}
\left( |(L_0-L)u|^q\right)_{Q_{\kappa r}(X_0)}^{\frac 1 q}
&\le \left(|\bar{A}^{\alpha\beta} - A^{\alpha\beta}|^{\nu q}\right)_{Q_{\kappa r}(X_0)}^{1/(q\nu)}
\left( |D^{2m}u|^{q\mu} \right)_{Q_{\kappa r}(X_0)}^{\frac{1}{q\mu}}\\
&\le N \gamma^{\frac{1}{q\nu}} \left( |D^{2m}u|^{q\mu} \right)_{Q_{\kappa r}(X_0)}^{\frac{1}{q\mu}},
\end{align*}
where $N = N(\delta,q,\mu)$, the first inequality is due to H\"{o}lder's inequality, and the second inequality is due to the fact that $\kappa r \le R_0$, the boundedness of $A^{\alpha\beta}$, and Assumption \ref{assum0528_2} ($\gamma$).
This together with \eqref{eq0607_01} with $\bar{f}$ proves the desired inequality.
\end{proof}

We use the following filtration of partitions.
$$
\bC_n := \{Q^n = Q^n_{(i_0, i_1, \ldots, i_d)}:  (i_0, i_1, \ldots, i_d) \in \bZ^{d+1}\},
$$
where $n \in \bZ$ and
$$
Q^n_{(i_0, i_1, \ldots, i_d)}
$$
$$
= [i_0 2^{-2m n} , (i_0 + 1)2^{-2m n}) \times [i_1 2^{-n} , (i_1 + 1)2^{-n}) \times \ldots \times [i_d 2^{-n} , (i_d + 1)2^{-n}).
$$
Let $X \in \bR^{d+1}$ and $X \in Q^n \in \bC_n$. Then one can find $X_0 \in \bR^{d+1}$ and the smallest $r > 0$ (in fact, $r =\max\{2^{-n-1}\sqrt d,2^{-n}\}$) such that $Q^n \subset Q_r(X_0)$ and
\begin{equation}
							\label{eq0528_03}
\dashint_{Q^n} |f(Y) - f_{|n}(X)| \, dY \le N \dashint_{Q_r(X_0)} \left| f(Y) - \left( f \right)_{Q_r(X_0)} \right| \, dY,
\end{equation}
where $N$ depends only on $d$ and $m$.

\begin{lemma}
							\label{lem0607_1}
Let $\lambda \ge 0$, $K_0\ge 1$, $p, q \in (1,\infty)$, $t_1 \in \bR$, $w = w_1(x') w_2(t,x'')$, where
$$
w_1(x') \in A_p(\bR^{d_1}, dx'),
\quad
w_2(t,x'') \in A_q(\bR \times \bR^{d_2}, dx'' \, dt),
$$
$$
d_1 + d_2 = d,
\quad [w_1]_{A_p}\le K_0,
\quad [w_2]_{A_q}\le K_0,
$$
and let $L$ be the operator in \eqref{eq0528_01}.
Suppose that the lower-order coefficients of $L$ are all zero.
Then there exist constants
$$
\begin{aligned}
&\gamma = \gamma(d, m, \ell, \delta, p, q, d_1, d_2, K_0) \in (0,1),
\\
&R_1 = R_1(d, m, \ell, \delta, p, q, d_1, d_2, K_0) \in (0,1),
\end{aligned}
$$
such that,
under Assumption \ref{assum0528_2} ($\gamma$), for $u \in W_{p,q,w}^{1,2m}(\bR^{d+1})$ vanishing outside $\left(t_1 - (R_0R_1)^{2m}, t_1\right) \times \bR^d$ and satisfying \eqref{eq0604_06} in $\bR^{d+1}$, where $f \in L_{p,q,w}(\bR^{d+1})$,
we have
\begin{equation}
							\label{eq0604_03}
\sum_{|\alpha| \le 2m} \lambda^{1-\frac{|\alpha|}{2m}} \| D^\alpha u \|_{L_{p, q, w}}
\le N \|f\|_{L_{p, q, w}},
\end{equation}
where $L_{p,q,w} = L_{p,q,w}(\bR^{d+1})$ and
$N = N (d,m,\ell,\delta, p, q, d_1, d_2, K_0)$.
\end{lemma}

\begin{proof}
For the given $w_1 \in A_p(\bR^{d_1}, dx')$ and $w_2 \in A_q(\bR \times \bR^{d_2}, dx'' \, dt)$, using reverse H\"older's inequality \cite[Theorem 3.2]{MS1981} (also see Corollary 7.2.6 and Remark 7.2.3 in \cite{MR3243734}) we choose
$$
\sigma_1 = \sigma_1(d_1, p, K_0),
\quad
\sigma_2 = \sigma_2(d_2, q, K_0)
$$
such that $p-\sigma_1 > 1$, $q - \sigma_2 >1$ and
$$
w_1 \in A_{p-\sigma_1}(\bR^{d_1}, dx'),
\quad
w_2 \in A_{q-\sigma_2}(\bR \times \bR^{d_2}, dx'' \, dt).
$$
Find $q_0, \mu \in (1,\infty)$ satisfying
\begin{equation}
							\label{eq0610_07}
q_0 \mu = \min\left\{ \frac{p}{p-\sigma_1}, \frac{q}{q-\sigma_2} \right\}> 1.
\end{equation}
Note that
\begin{equation}
							\label{eq0605_13}
\begin{aligned}
w_1 &\in A_{p-\sigma_1} \subset A_{p/(q_0\mu)} \subset A_{p/q_0}(\bR^{d_1}, dx'),
\\
w_2 &\in A_{q-\sigma_2} \subset A_{q/(q_0\mu)} \subset A_{q/q_0}(\bR \times \bR^{d_2}, dx'' \, dt).
\end{aligned}
\end{equation}
By Lemma \ref{lem0320_1},
for any $g\in L_{q_0\mu, \text{loc}}(\bR^{d+1})$ and balls $B_1 \subset \bR^{d_1}$ and $B_2 \subset \bR \times \bR^{d_2}$,
where $B_2$ is a ball with respect to the parabolic distance $|x'' - y''|+|t-s|^{\frac{1}{2m}}$,
\begin{align*}
&\frac{1}{|B_1||B_2|}\int_{B_1 \times B_2} |g|^{q_0\mu} \, dx \, dt
= \frac{1}{|B_2|} \int_{B_2} \frac{1}{|B_1|} \int_{B_1} |g|^{q_0 \mu} \, dx' \, dx'' \, dt\\
&\le \frac{1}{|B_2|} \int_{B_2} \left( \frac{[w_1]_{A_{p/(q_0\mu)}}}{\omega_1(B_1)} \int_{B_1} |g|^p \, w_1(x') \, dx'   \right)^{\frac{q_0 \mu}{p}} \, dx'' \, dt\\
&\le \left( \frac{[w_2]_{A_{q/(q_0\mu)}}}{\omega_2(B_2)} \int_{B_2} \left( \frac{[w_1]_{A_{p/(q_0\mu)}}}{\omega_1(B_1)} \int_{B_1} |g|^p \, w_1(x') \, dx'   \right)^{q/p} w_2(t,x'') \, dx'' \, dt                  \right)^{\frac{q_0 \mu}{q}}.
\end{align*}
This shows that
$$
u \in W_{q_0\mu, \text{loc}}^{1,2m}(\bR^{d+1}), \quad
f \in L_{q_0\mu, \text{loc}}(\bR^{d+1}),
\quad
0 \le |\alpha| \le 2m.
$$

Let $\kappa \ge 8$ be a constant to be specified below.
For each $X \in \bR^{d+1}$ and $Q^n \in \bC_n$ such that $X \in Q^n$,  $n \in \bZ$,  find $X_0 \in \bR^{d+1}$ and the smallest $r \in (0,\infty)$ so that $Q^n \subset Q_r(X_0)$ and \eqref{eq0528_03} is satisfied.
If $r > R_0/\kappa$, because $u$ vanishes outside $\left(t_1-(R_0 R_1)^{2m}, t_1\right) \times \bR^d$, for $0 \le |\alpha| \le 2m$,
we have
\begin{multline}
							\label{eq0605_10}
\dashint_{Q^n} | D^{\alpha}u (Y) - \left(D^{\alpha}u\right)_{|n}(X) | \, dY
\le 2 \dashint_{Q^n} | D^{\alpha} u | \, dY
\\
\le 2 \left(\dashint_{Q_r(X_0)} I_{(t_1-(R_0R_1)^{2m},t_1)}(s) \, dY\right)^{1 - \frac{1}{q_0}} \left(\dashint_{Q_r(X_0)} |D^{\alpha}u|^{q_0} \, dY \right)^{\frac{1}{q_0}}
\\
\le N(d,m,q_0) \kappa^{2m(1 - \frac{1}{q_0})} {R_1}^{2m(1 - \frac{1}{q_0})} \left[ \cM \left( |D^{\alpha} u|^{q_0} \right)(X) \right]^{\frac{1}{q_0}},
\end{multline}
where, for the last inequality, we have used the inequality
\begin{equation}
							\label{eq0607_02}
\left( |D^{\alpha}u|^{q_0} \right)^{\frac{1}{q_0}}_{Q_{\kappa r}(X_0)} \le N(d,m) \left[\cM \left( |D^{\alpha}u|^{q_0} \right)(X)\right]^{\frac{1}{q_0}}.
\end{equation}

If $r \in (0, R_0/\kappa]$, by Lemma \ref{lem0604_01} with $q = q_0$ and inequalities as in \eqref{eq0607_02},
\begin{multline}
							\label{eq0605_12}
\sum_{|\alpha|\le 2m} \lambda^{1-\frac{|\alpha|}{2m}} \dashint_{Q^n} | D^{\alpha}u (Y) - \left(D^{\alpha}u\right)_{|n}(X) | \, dY
\\
\le N \kappa^{-1} \sum_{|\alpha| \le 2m} \lambda^{1-\frac{|\alpha|}{2m}} \left[ \cM \left(|D^\alpha u|^{q_0}\right)(X) \right]^{\frac{1}{q_0}}
+ N \kappa^{\frac{d+2m}{q_0}} \left[ \cM\left( |f|^{q_0} \right)(X) \right]^{\frac{1}{q_0}}
\\
+ N \kappa^{\frac{d+2m}{q_0}} \gamma^{\frac{1}{q_0\nu}} \left[ \cM \left( |D^{2m}u|^{q_0\mu} \right)(X) \right]^{\frac{1}{q_0\mu}},
\end{multline}
where $1/\mu + 1/\nu = 1$ and $N=N(d,m,\ell,\delta,q_0,\mu)$.
Note that $q_0$ and $\mu$ depend only on $d_1$, $d_2$, $p$, $q$, and $K_0$. Thus we have
$$
N=N(d,m,\ell,\delta, p, q, d_1, d_2, K_0).
$$
Combining \eqref{eq0605_10} and \eqref{eq0605_12}, and taking the supremum with respect to all $Q^n \ni X$, $n \in \bZ$, we see that
\begin{align*}
&\sum_{|\alpha|\le 2m} \lambda^{1-\frac{|\alpha|}{2m}} \left( D^{\alpha}u \right)_{\text{dy}}^{\#}(X)\\
&\le  N \left(\kappa^{2m(1 - \frac{1}{q_0})} {R_1}^{2m(1 - \frac{1}{q_0})} + \kappa^{-1}\right)
\sum_{|\alpha|\le 2m} \lambda^{1-\frac{|\alpha|}{2m}} \left[ \cM \left( |D^{\alpha} u|^{q_0} \right)(X) \right]^{\frac{1}{q_0}}\\
&\quad + N \kappa^{\frac{d+2m}{q_0}} \left[ \cM \left( |f|^{q_0} \right)(X) \right]^{\frac{1}{q_0}}
+ N \kappa^{\frac{d+2m}{q_0}} \gamma^{\frac{1}{q_0\nu}} \left[ \cM  \left( |D^{2m}u|^{q_0\mu} \right)(X) \right]^{\frac{1}{q_0\mu}}
\end{align*}
for all $X \in \bR^{d+1}$.
Now we take $L_{p,q,w}(\bR^{d+1})$-norms of the both sides of the above inequality,  and use Corollaries \ref{cor1} and \ref{cor2}.
In particular, due to \eqref{eq0605_13} we are able to use Corollary \ref{cor1} to obtain, for instance,
\begin{align*}
&\|\left[ \cM (|D^{\alpha} u |^{q_0\mu} )\right]^{\frac{1}{q_0\mu}}\|_{L_{p,q,w}}
= \| \cM (|D^{\alpha} u |^{q_0\mu} ) \|_{L_{p/(q_0\mu), q/(q_0\mu), w}}^{\frac{1}{q_0\mu}}\\
&\le N\| |D^{\alpha}u|^{q_0\mu}\|_{L_{p/(q_0\mu), q/(q_0\mu), w}}^{\frac{1}{q_0\mu}} = N \| D^{\alpha} u \|_{L_{p, q, w}},
\end{align*}
where $N=N(d, p/(q_0 \mu), q/(q_0\mu), K_0)$; hence
$N=N(d, p, q, d_1, d_2, K_0)$.
Therefore, we get
\begin{multline}
							\label{eq0610_08}
\sum_{|\alpha|\le 2m} \lambda^{1-\frac{|\alpha|}{2m}} \| D^{\alpha} u\|_{L_{p,q,w}}
\\
\le N \left(\kappa^{2m(1-\frac{1}{q_0})} {R_1}^{2m(1-\frac{1}{q_0})} + \kappa^{-1}\right) \sum_{|\alpha|\le 2m} \lambda^{1-\frac{|\alpha|}{2m}} \| D^{\alpha} u\|_{L_{p,q,w}}
\\
+ N \kappa^{\frac{d+2m}{q_0}} \|f\|_{L_{p,q,w}} + N \kappa^{\frac{d+2m}{q_0}} \gamma^{\frac{1}{q_0\nu}} \|D^{2m}u\|_{L_{p,q,w}},
\end{multline}
where $N=N(d,m,\ell,\delta,p,q,d_1, d_2, K_0)$.
Fix $\kappa \ge 8$ such that $N\kappa^{-1} \le 1/6$. Then choose $\gamma \in (0,1)$ and $R_1 \in (0,1)$ so that
$$
N \kappa^{\frac{d+2m}{q_0}} \gamma^{\frac{1}{q_0\nu}} \le 1/6,
\quad
N \kappa^{2m(1-\frac{1}{q_0})} {R_1}^{2m(1-\frac{1}{q_0})} \le 1/6.
$$
Then we arrive at \eqref{eq0604_03}.\end{proof}

We now use the standard partition of unity argument with respect to only one variable (the time variable).

\begin{proposition}
							\label{prop0609_1}
Let $\lambda \ge 0$, $K_0\ge 1$, $p, q \in (1,\infty)$, $w = w_1(x') w_2(t,x'')$, where
$$
w_1(x') \in A_p(\bR^{d_1}, dx'),
\quad
w_2(t,x'') \in A_q(\bR \times \bR^{d_2}, dx'' \, dt),
$$
$$
d_1+d_2=d,\quad
[w_1]_{A_p}\le K_0,\quad [w_2]_{A_q}\le K_0,
$$
and let $L$ be the operator in \eqref{eq0528_01}.
Then there exists
$$
\gamma = \gamma(d, m, \ell, \delta, p, q, d_1, d_2, K_0) \in (0,1)
$$
such that,
under Assumption \ref{assum0528_2} ($\gamma$), for $u \in W_{p,q,w}^{1,2m}(\bR^{d+1})$ satisfying \eqref{eq0604_06} in $\bR^{d+1}$, where $f \in L_{p,q,w}(\bR^{d+1})$,
we have
\begin{equation}
							\label{eq0609_03}
\sum_{|\alpha| \le 2m} \lambda^{1-\frac{|\alpha|}{2m}} \| D^\alpha u \|_{L_{p,q,w}} \le
N_1 \|f\|_{L_{p,q,w}} + N_2 \sum_{|\alpha| \le 2m-1} \| D^\alpha u \|_{L_{p,q,w}},
\end{equation}
where $L_{p,q,w} = L_{p,q,w}(\bR^{d+1})$,
$$
N_1 = N_1(d,m,\ell,\delta, p, q, d_1, d_2, K_0),
$$
$$
N_2 = N_2(d,m,\ell,\delta, p, q, d_1, d_2, K_0, K, R_0).
$$
\end{proposition}

\begin{proof}
Recall that the lower-order coefficients in $L$ are bounded by $K$.
By moving the terms $A^{\alpha\beta}D^{\alpha}D^{\beta}u$, $|\alpha| < m$ or $|\beta| < m$, to the right-hand side of the system, we assume that the lower-order coefficients of $L$ are zero.
Take $\gamma \in (0,1)$ and $R_1 \in (0,1)$ from Lemma \ref{lem0607_1} and fix a non-negative infinitely differentiable function $\zeta(t)$ defined on $\bR$ such that $\zeta(t)$ vanishes outside $(-(R_0R_1)^{2m}, 0)$ and
$$
\int_{\bR} \zeta(t)^q \, dt = 1.
$$
Then $u(t,x) \zeta(t-s)$ satisfies
\begin{multline}
							\label{eq0607_03}
\left( u(t,x) \zeta(t-s) \right)_t + (-1)^m L \left(u(t,x) \zeta(t-s)\right) + \lambda u(t,x) \zeta(t-s)
\\
= \zeta(t-s)f(t,x) + \zeta_t(t-s) u(t,x)
\end{multline}
in $\bR^{d+1}$.
For each $s \in \bR$, since $u(t,x)\zeta(t-s)$ vanishes outside $(s-(R_0R_1)^{2m}, s) \times \bR^d$, by Lemma \ref{lem0607_1} applied to \eqref{eq0607_03}, we get
\begin{equation}
							\label{eq0607_04}
\sum_{|\alpha| \le 2m} \lambda^{1-\frac{|\alpha|}{2m}} \| D^{\alpha} \left(u \zeta(\cdot -s)\right) \|_{L_{p, q, w}}
\le N \|f \zeta(\cdot-s)\|_{L_{p, q, w}}
+ N \|u \zeta_t(\cdot-s)\|_{L_{p, q, w}},
\end{equation}
where $L_{p,q,w} = L_{p,q,w}(\bR^{d+1})$ and
$N = N (d,m,\ell,\delta, p, q, d_1, d_2, K_0)$.
Note that
\begin{align*}
&\|D^\alpha u(t,\cdot,x'')\|_{L_{p,w_1}(\bR^{d_1})}^q = \int_{\bR} \|D^\alpha u(t,\cdot,x'')\|_{L_{p,w_1}(\bR^{d_1})}^q \zeta(t-s)^q \, ds,\\
&= \int_{\bR} \|D^\alpha u(t,\cdot,x'')\zeta(t-s)\|_{L_{p,w_1}(\bR^{d_1})}^q \, ds.
\end{align*}
Thus, by integrating with respect to $t$ and $x''$,
$$
\| D^{\alpha} u \|_{L_{p,q,w}}^q = \int_{\bR} \|D^{\alpha}\left( u \zeta(\cdot -s) \right)\|_{L_{p,q,w}}^q \, ds.
$$
From this and \eqref{eq0607_04} it follows that
$$
\sum_{|\alpha| \le 2m} \lambda^{1-\frac{|\alpha|}{2m}} \| D^\alpha u \|_{L_{p,q,w}}
\le N_1 \|f\|_{L_{p,q,w}} + N_2 \|u\|_{L_{p,q,w}},
$$
where $N_1 = N_1(d,m,\ell,\delta,p,q,d_1, d_2, K_0)$, and $N_2$ depends on $R_0R_1$ and the same parameters as $N_1$ does.
The proposition is proved.
\end{proof}

\begin{proof}[Proof of Theorem \ref{thm0607_1}]
In Proposition \ref{prop0609_1} we choose $\lambda_0 \ge 1$, depending only on $N_2$, so that
$$
\frac{1}{2}\sum_{|\alpha|\le 2m-1} \lambda^{1-\frac{|\alpha|}{2m}} \le \sum_{|\alpha|\le 2m-1} \left(\lambda^{1-\frac{|\alpha|}{2m}} - N_2\right)
$$
for any $\lambda \ge \lambda_0$.
Then in the right-hand side of \eqref{eq0609_03} the terms involving $D^\alpha u$, $0 \le |\alpha| \le 2m-1$ can be absorbed to the left-hand side so that we obtain
$$
\sum_{|\alpha|\le 2m} \lambda^{1-\frac{|\alpha|}{2m}} \|D^\alpha u\|_{L_{p,q,w}} \le N \|f\|_{L_{p,q,w}}.
$$
Finally, we estimate $u_t$ using this estimate and the system $u_t = f - (-1)^m Lu - \lambda u$. The theorem is proved.
\end{proof}

\subsection{Mixed-norm estimate in a half space}

This section is devoted to the proof of Theorem \ref{thm4.4}.
We only consider the case when $d_1\ge 1$ and thus the cylindrical domain under consideration is $\bR\times \bR^d_+  =: \bR^{d+1}_+$.
That is, regarding the product of weights and the mixed norm, we view the domain as the product space of $\bR^{d_1}_+$, $d_1 \ge 1$, and $\bR  \times \bR^{d_2}$.
The other case can be treated in a similar way.

We mainly follow the arguments in the previous subsection and Sections 7--9 of \cite{MR2771670}. For a function $g$ defined on a subset $\cD$ in $\bR^{d+1}$,
we set
$$
[g]_{\cC^{\nu}(\cD)}
= \sup_{\substack{(t,x), (s,y) \in \cD \\ (t,x) \ne (s,y)}}
\frac{|g(t,x)-g(s,y)|}{|t-s|^{\frac \nu {2m}} + |x-y|^\nu},
$$
where $0 < \nu \le 1$.
In what follows, we write $Q_r^+(X_0)$ to denote $Q_r(X_0) \cap \bR^{d+1}_+$.

First we have the following boundary H\"older estimate, which corresponds to Lemma \ref{lem4.4}.

\begin{lemma}
                                    \label{lem4.10}
Let $\lambda \ge 0$, $q \in (1,\infty)$, and $L$ be the operator in \eqref{eq0528_01}.
Suppose that the coefficients $A^{\alpha\beta}$, $|\alpha|=|\beta|=m$, are measurable functions of only $t \in \bR$, i.e., $A^{\alpha\beta} = A^{\alpha\beta}(t)$ and the lower-order coefficients of $L$ are all zero. Then for any $u \in C_{\operatorname{loc}}^\infty(\overline{\bR^{d+1}_+})$ satisfying \eqref{eq0604_06} in $Q_4^+$ with $f=0$ and \eqref{eq1.36} on $Q_4 \cap \{x_1=0\}$, we have
\begin{equation*}
\sum_{|\alpha| \le 2m,\alpha_1<2m} \lambda^{1-\frac{|\alpha|}{2m}} [D^\alpha u]_{\cC^{1/2}(Q_1^+)}
\le N \sum_{|\alpha| \le 2m} \lambda^{1-\frac{|\alpha|}{2m}}\|D^\alpha u\|_{L_q(Q_4^+)},
\end{equation*}
where $N = N(d,m,\ell,\delta,q)$.
\end{lemma}
\begin{proof}
We only consider the case when $\lambda=0$, i.e.,
\begin{equation}
                                \label{eq9.27}
[D_{\hat{x}}D^{2m-1} u]_{\cC^{1/2}(Q_1^+)}
\le N \sum_{|\alpha| = 2m} \|D^\alpha u\|_{L_q(Q_4^+)}.
\end{equation}
The general case then follows by using S. Agmon's idea.

By the $W^{1,2m}_q$ estimates in the half space (\cite[Theorem 4]{MR2771670}), the Sobolev imbedding theorem, and the standard bootstrap argument, we have
$$
\|u\|_{W^{1,2m}_p(Q_1^+)}\le N\|u\|_{L_q(Q_4^+)}
$$
for any $p\in (1,\infty)$, where $N=N(d,\ell,m,\delta,p,q)$. By taking $p>2(2m+d)$ and using the parabolic Sobolev imbedding theorem (see, for instance, \cite[Sec. 18.12]{MR521808}),
\begin{equation}
                    \label{eq9.34}
[D^{2m-1}u]_{\cC^{1/2}(Q_1^+)}\le N\|u\|_{L_q(Q_4^+)}\le N\| D_1^m u\|_{L_q(Q_4^+)},
\end{equation}
where the last inequality is due to the boundary Poincar\'e inequality.
Let $Q(x)$ be a vector-valued polynomial of order at most $m-1$ such that, for $P(x):=x_1^m Q(x)$,
\begin{equation}
                                            \label{eq9.48}
\left(D^k D_1^m P(x)\right)_{Q^+_4}= \left(D^k D_1^m u(t,x) \right)_{Q^+_4},
\end{equation}
where $0\le k\le m-1$.
Notice that $P$ is a vector-valued polynomial of order at most $2m-1$. Let $v=u-P(x)$, which satisfies the same system with the same Dirichlet boundary condition on $\{x_1=0\}$ as $u$ does. By \eqref{eq9.34} with $v$ in place of $u$ and applying a Poincar\'e type inequality \cite[Lemma 6.1]{MR3385152} with $q$ in place of $2$ (the same proof applies), we get
\begin{multline}
                                \label{eq9.47}
[D^{2m-1}u]_{\cC^{1/2}(Q_1^+)}=[D^{2m-1}v]_{\cC^{1/2}(Q_1^+)}
\le N\| D_1^m v\|_{L_q(Q_4^+)}\\
\le N\|D^{m-1}D_1^m v\|_{L_q(Q_4^+)}\le N\|D^{m-1}D_1^m u\|_{L_q(Q_4^+)},
\end{multline}
where we used \eqref{eq9.48} with $k=m-1$ in the last inequality. Since $D_{\hat{x}}u$ satisfies the same system as $u$ with the same boundary condition on $\{x_1=0\}$, we finally obtain \eqref{eq9.27} from \eqref{eq9.47}. The lemma is proved.
\end{proof}

Similar to Lemma \ref{lem0528_01} and \cite[Corollary 5]{MR2771670}, from Lemma \ref{lem4.10} we obtain the following boundary mean oscillation estimate for all $D^\alpha u$, $0 \le |\alpha| \le 2m$, except $D_1^{2m}u$.

\begin{lemma}
							\label{lem4.11}
Let $\kappa \ge 64$. Under the assumptions of Lemma \ref{lem4.10}, for any $r \in (0,\infty)$, $X_0 \in \overline{\bR^{d+1}_+}$, and $u \in W_{q, \operatorname{loc}}^{1,2m}(\overline{\bR^{d+1}_+})$ satisfying \eqref{eq0604_06} in $Q_{\kappa r}^+(X_0)$ and \eqref{eq1.36} on $Q_{\kappa r}(X_0)\cap \{x_1=0\}$,
where $f \in L_{q,\operatorname{loc}}(\overline{\bR^{d+1}_+})$,
we have
\begin{multline*}
\sum_{|\alpha| \le 2m, \alpha_1<2m} \lambda^{1-\frac{|\alpha|}{2m}} \left( | D^\alpha u - \left(D^\alpha u\right)_{Q^+_r(X_0)} |\right)_{Q^+_r(X_0)}
\\
\le N \kappa^{-\frac 1 2} \sum_{|\alpha| \le 2m} \lambda^{1-\frac{|\alpha|}{2m}}\left( |D^\alpha u|^q \right)_{Q^+_{\kappa r}(X_0)}^{\frac 1 q} +  N \kappa^{\frac{d+2m}{q}} \left(
|f|^q\right)_{Q^+_{\kappa r}(X_0)}^{\frac 1 q},
\end{multline*}
where $N = N(d,m,\ell,\delta,q)$.
\end{lemma}

In order to estimate $D_1^{2m}u$, as in \cite[Sec. 8]{MR2771670}, we consider the following system with the special coefficients
$$
\fL_0 u = A(t)D^{2m}_1u + \sum_{j=2}^d D_j^{2m} u,
$$
where $A(t) = A^{\check{\alpha}\check{\alpha}}(t)$, $\check{\alpha}=(m,0, \cdots,0)$.
Note that if a sufficiently smooth $u$ satisfies
\begin{equation}
                            \label{eq2.12}
u_t+(-1)^m \fL_0 u+\lambda u=0
\end{equation}
with the Dirichlet boundary condition on $\{x_1=0\}$, then $D_1^{2m}u$ satisfies the same system with the Dirichlet boundary condition because
$$
D_1^{2m}u=(A(t))^{-1}(-1)^{m+1}(u_t+\lambda u)-(A(t))^{-1}\sum_{j=2}^d D_j^{2m} u.
$$
Thus, we have boundary H\"older and mean oscillation estimates of the $2m$-th order normal derivative of $u$.

\begin{lemma}
                                    \label{lem4.12}
Let $\lambda \ge 0$ and $q \in (1,\infty)$.
For any $u \in C_{\operatorname{loc}}^\infty(\overline{\bR^{d+1}_+})$ satisfying \eqref{eq2.12} in $Q_4^+$ and \eqref{eq1.36} on $Q_4\cap \{x_1=0\}$, we have
\begin{equation*}
[D_{1}^{2m} u]_{\cC^{1}(Q_1^+)}
\le N  \sum_{|\alpha| \le 2m} \lambda^{1-\frac{|\alpha|}{2m}}\|D^\alpha u\|_{L_q(Q_4^+)},
\end{equation*}
where $N = N(d,m,\ell,\delta,q)$.
\end{lemma}
\begin{proof}
See \cite[Corollary 6]{MR2771670}.
\end{proof}

\begin{lemma}
							\label{lem4.13}
Let $\lambda \ge 0$, $q \in (1,\infty)$, and $\kappa \ge 64$.
For any $r \in (0,\infty)$, $X_0 \in \overline{\bR^{d+1}_+}$, and $u \in W_{q, \operatorname{loc}}^{1,2m}(\overline{\bR^{d+1}_+})$ satisfying $u_t+(-1)^m \fL_0 u+\lambda u=f$ in $Q_{\kappa r}^+(X_0)$ and \eqref{eq1.36} on $Q_{\kappa r}(X_0)\cap \{x_1=0\}$,
where $f \in L_{q,\operatorname{loc}}(\overline{\bR^{d+1}_+})$,
we have
\begin{multline*}
\left( |D_{1}^{2m} u - \left(D_{1}^{2m} u\right)_{Q^+_r(X_0)} |\right)_{Q_r(X_0)}
\\
\le N \kappa^{-1} \sum_{|\alpha| \le 2m} \lambda^{1-\frac{|\alpha|}{2m}}\left( |D^\alpha u|^q \right)_{Q^+_{\kappa r}(X_0)}^{\frac 1 q} +  N \kappa^{\frac{d+2m}{q}} \left(
|f|^q\right)_{Q^+_{\kappa r}(X_0)}^{\frac 1 q},
\end{multline*}
where $N = N(d,m,\ell,\delta,q)$.
\end{lemma}
\begin{proof}
Similar to Lemmas \ref{lem0528_01} and \ref{lem4.11}, the lemma is derived from Lemma \ref{lem4.12} and the $W^{1,2m}_q$ estimate in the half space.
\end{proof}

We are now ready to give the proof of Theorem \ref{thm4.4}.

\begin{proof}[Proof of Theorem \ref{thm4.4}]
Due to the arguments in the proofs of Proposition \ref{prop0609_1} and Theorem \ref{thm0607_1},
it suffices to prove
\begin{equation}
							\label{eq0604_03b}
\sum_{|\alpha| \le 2m} \lambda^{1-\frac{|\alpha|}{2m}} \| D^\alpha u \|_{L_{p, q, w}}
\le N \|f\|_{L_{p, q, w}},
\end{equation}
provided that the lower order coefficients are all zero, and $u$ vanishes outside $(t_1-(R_0R_1)^{2m}, t_1)\times \bR^d_+$
for some $t_1\in \bR$ and a small $R_1=R_1(d,m,\ell,\delta,p,q,d_1,d_2,K_0)\in (0,1)$.
As in the proof of Lemma \ref{lem0607_1}, from Lemma \ref{lem4.11},
for any $\kappa_1\ge 64$, we have
\begin{multline}
                                    \label{eq2.40}
\sum_{|\alpha| \le 2m, \alpha_1 < 2m} \lambda^{1-\frac{|\alpha|}{2m}}
\|D^\alpha u\|_{L_{p,q,w}}
\\
\le N \left(\kappa_1^{2m(1-\frac 1 {q_0})}R_1^{2m(1-\frac 1 {q_0})}+\kappa_1^{-\frac 1 2}\right) \sum_{|\alpha| \le 2m}
\lambda^{1-\frac{|\alpha|}{2m}}\|D^\alpha u \|_{L_{p,q,w}}\\
+  N \kappa_1^{\frac{d+2m}{q_0}} \|f\|_{L_{p,q,w}}
+N \kappa_1^{\frac{d+2m}{q_0}}\gamma^{\frac 1 {q_0\nu}} \|D^{2m}u\|_{L_{p,q,w}},
\end{multline}
where $L_{p,q,w}=L_{p,q,w}(\bR^{d+1}_+)$.
Next we move all the spatial derivatives except $A^{\check\alpha\check\alpha}D_1^{2m} u$ to the right-hand side of the system, and add $(-1)^m\sum_{j=2}^d D^{2m}_j u$ to both sides.
Similarly to \eqref{eq2.40}, from Lemma \ref{lem4.13}, we derive for any $\kappa_2\ge 64$,
\begin{multline}
                                    \label{eq2.43}
\|D_{1}^{2m} u\|_{L_{p,q,w}}
\le N \left(\kappa_2^{2m(1-\frac 1 {q_0})}R_1^{2m(1-\frac 1 {q_0})}+\kappa_2^{-1}\right) \sum_{|\alpha| \le 2m}
\lambda^{1-\frac{|\alpha|}{2m}}\|D^\alpha u\|_{L_{p,q,w}}\\
+  N \kappa_2^{\frac{d+2m}{q_0}}\Big( \|f\|_{L_{p,q,w}}
+\sum_{|\alpha|=2m,\alpha_1<2m}
\|D^\alpha u\|_{L_{p,q,w}}
+\gamma^{\frac 1 {q_0\nu}} \|D_1^{2m}u\|_{L_{p,q,w}}\Big).
\end{multline}
Combining \eqref{eq2.40} and \eqref{eq2.43}, we reach
\begin{align}
                                    \label{eq2.46}
&\sum_{|\alpha| \le 2m} \lambda^{1-\frac{|\alpha|}{2m}}
\|D^\alpha u\|_{L_{p,q,w}}\nonumber\\
&\le N \left(\kappa_2^{2m(1-\frac 1 {q_0})}R_1^{2m(1-\frac 1 {q_0})}+\kappa_2^{-1}\right)\sum_{|\alpha| \le 2m}
\lambda^{1-\frac{|\alpha|}{2m}}\|D^\alpha u\|_{L_{p,q,w}}\nonumber\\
&\quad +N\kappa_2^{\frac{d+2m}{q_0}}\left(\kappa_1^{2m(1-\frac 1 {q_0})}R_1^{2m(1-\frac 1 {q_0})}+\kappa_1^{-\frac 1 2}\right)\sum_{|\alpha| \le 2m}
\lambda^{1-\frac{|\alpha|}{2m}}\|D^\alpha u\|_{L_{p,q,w}}\nonumber\\
&+  N \kappa_2^{\frac{d+2m}{q_0}}(1+\kappa_1^{\frac{d+2m}{q_0}}) \|f\|_{L_{p,q,w}}
+N \kappa_2^{\frac{d+2m}{q_0}}(1+ \kappa_1^{\frac{d+2m}{q_0}})\gamma^{\frac 1 {q_0\nu}} \| D^{2m}u\|_{L_{p,q,w}}.
\end{align}
To see \eqref{eq0604_03b}, it suffices to first take $\kappa_2$ sufficiently large, then $\kappa_1$ sufficiently large, and finally $R_1$ and $\gamma$ sufficiently small in \eqref{eq2.46}, such that
\begin{multline*}
N \left(\kappa_2^{2m(1-\frac 1 q_0)}R_1^{2m(1-\frac 1 q_0)}+\kappa_2^{-1}\right)
+N\kappa_2^{\frac{d+2m}{q_0}}\left(\kappa_1^{2m(1-\frac 1 q_0)}R_1^{2m(1-\frac 1 q_0)}+\kappa_1^{-\frac 1 2}\right)\\
+N \kappa_2^{\frac{d+2m}{q_0}}(1+ \kappa_1^{\frac{d+2m}{q_0}})\gamma^{\frac 1 {q_0\nu}}\le 1/2.
\end{multline*}
The theorem is proved.
\end{proof}

\section{Second order parabolic equations in non-divergence form with measurable coefficients}
							\label{sec06}

In this section, we consider second order equations with partially BMO coefficients. Thus, all the involved coefficients and functions are real scalar-valued functions.
Throughout the section we use the notation in Section \ref{sec3} by setting $m=1$.
In particular,
$$
Q_r(t,x) = (t-r^2, t) \times B_r(x),
\quad
Q'_r(t,\hat x) = (t-r^2, t) \times B_r'(\hat x).
$$
Note that, for a multi-index $\alpha = (\alpha_1, \ldots, \alpha_d)$, if we write
$$
D^\alpha u, \quad 0\le|\alpha| \le 2, \quad \alpha_1 \le 1,
$$
it means one of the elements or the whole elements of the set
$$
\{u, D_1u, \ldots, D_du, D_{ij}u : (i,j) \ne (1,1)\}.
$$
Set
\begin{equation}
							\label{eq0609_01}
Lu = \sum_{i,j=1}^d a^{ij} D_{ij} u + b^i D_i u + c u.
\end{equation}

Throughout the section we impose the following assumptions on the coefficients.
\begin{enumerate}

\item[(i)] There exists $\delta \in (0,1)$ such that
$$
a^{ij} \xi_i \xi_j \ge \delta |\xi|^2,
\quad
|a^{ij}| \le \delta^{-1}
$$
for all $\xi \in \bR^d$.

\item[(ii)] $b^i$ and $c$ are measurable and bounded. In particular, there exists $K \in (0, \infty)$ such that
$$
|b^i| \le K,
\quad
|c| \le K.
$$

\end{enumerate}

We also assume one of the following regularity assumptions on the leading coefficients $a^{ij}$.

\begin{assumption}[$\gamma$]
							\label{assum0522_1}
Let $\gamma \in (0,1)$. There exists $R_0 \in (0,\infty)$ such that
$$
(a^{11})^{\#,1}_{R_0} + \sum_{(i,j) \ne (1,1)} (a^{ij})^{\#}_{R_0} \le \gamma.
$$
\end{assumption}					

\begin{assumption}[$\gamma$]
							\label{assum0522_2}
Let $\gamma \in (0,1)$. There exists $R_0 \in (0,\infty)$ such that
$$
(a^{11})^{\#,2}_{R_0} + \sum_{(i,j) \ne (1,1)} (a^{ij})^{\#}_{R_0} \le \gamma.
$$
\end{assumption}		

The following theorems are the main results of this section.
As in Section \ref{sec4}, note that in Theorems \ref{thm0522_1} and \ref{thm5.4} below the domains $\bR^{d+1}$ and $\bR \times \bR^d_+$, as product spaces, satisfy the conditions (a) and (b) before Corollary \ref{cor1}.

\begin{theorem}[The whole space case]
							\label{thm0522_1}
Let $p,q \in (1,\infty)$ and $K_0 \ge 1$ be constants, $w = w_1(x') w_2(t,x'')$, where
$$
w_1(x') \in A_p(\bR^{d_1}, dx'),
\quad
w_2(t,x'') \in A_q(\bR \times \bR^{d_2}, dx'' \, dt),
$$
$$
d_1 + d_2 = d,\quad
[w_1]_{A_p}\le K_0,\quad [w_2]_{A_q}\le K_0,
$$
and let $L$ be the operator in \eqref{eq0609_01}.
Then there exist
$$
\gamma = \gamma(d,\delta,p,q,d_1,d_2, K_0) \in (0,1),
$$
$$
\lambda_0 = \lambda_0(d,\delta,p,q,d_1,d_2,K_0, K, R_0) \ge 1,
$$
such that, under Assumption \ref{assum0522_1} ($\gamma$) or Assumption \ref{assum0522_2} ($\gamma$), the following holds true.
For $u \in W_{p,q,w}^{1,2}(\bR^{d+1})$ satisfying
\begin{equation}
							\label{eq0525_01}
- u_t + L u - \lambda u = f
\end{equation}
in $\bR^{d+1}$, where $f \in L_{p,q,w}(\bR^{d+1})$, we have
$$
\|u_t\|_{L_{p,w}}+ \lambda \|u\|_{L_{p,q,w}} + \sqrt{\lambda}\|Du\|_{L_{p,q,w}} + \|D^2 u\|_{L_{p,q,w}} \le N \|f\|_{L_{p,q,w}},
$$
provided that $\lambda \ge \lambda_0$, where $L_{p,w} = L_{p,w}(\bR^{d+1})$ and
$$
N = N(d,\delta,p,q,d_1,d_2, K_0, K, R_0).
$$
\end{theorem}

\begin{theorem}[The half space case]
                    \label{thm5.4}
The result in Theorem \ref{thm0522_1} still holds if we replace $\bR^{d}$ by $\bR^{d}_+$ (i.e., replace $\bR^{d_1}$ by $\bR^{d_1}_+$ when $d_1 \ge 1$ or replace $\bR^{d_2}$ by $\bR^{d_2}_+$ when $d_1 = 0$) and impose the Dirichlet boundary condition $u=0$ or the Neumann boundary condition $D_1u=0$
on the lateral boundary of the cylindrical domain.
\end{theorem}

\begin{remark}
The domain in Theorem \ref{thm5.4} is fixed as $\bR \times \bR^d_+ = \{(t,x): x_1 > 0\}$, whereas the domain  in Theorem \ref{thm4.4} can be a half space $\{(t,x) \in \bR^{d+1}: x_k > 0\}$ for any $k = 1, \ldots, d$.
This is because the coefficients in Theorem \ref{thm4.4} have regularity in all the spatial variables, but the coefficients $a^{ij}$ in Theorem \ref{thm5.4} are only measurable (no regularity assumptions) with respect to $x_1$.
\end{remark}

\begin{remark}
							\label{rem0811}
It is worth noting that Theorem \ref{thm5.4} generalizes a recent result in \cite{DKZ14} on weighted $L_p$ estimates for second order parabolic equations in the half space $\bR \times \bR^d_+$ with the Neumann boundary condition, where the leading coefficients are assumed to satisfy the condition imposed in Section \ref{sec4} (i.e., they have small mean oscillations in all the spatial variables) and the weight $w(t,x)=x_1^{\theta-d}$ for some $\theta\in (d-1,d-1+p)$. It is easily seen that $w$ is an $A_p$ weight in the half space. Therefore, one can apply  Theorem \ref{thm5.4} to get the same result for equations with more general coefficients.
\end{remark}

To prove Theorem \ref{thm0522_1}, we start with the following mean oscillation estimate for equations with leading coefficients depending only on $x_1$.

\begin{lemma}
							\label{lem0610_1}
Let $\lambda \ge 0$, $q \in (1,\infty)$, $\kappa \ge 8$, and $L$ be the operator in \eqref{eq0609_01}.
Suppose that
$$
a^{ij} = a^{ij}(x_1),
\quad b^i = c = 0.
$$
Then, for any $r \in (0,\infty)$, $X_0 \in \bR^{d+1}$, and $u \in W_{q,\operatorname{loc}}^{1,2}(\bR^{d+1})$ satisfying \eqref{eq0525_01} in $Q_{\kappa r}(X_0)$, where $f \in L_{q,\operatorname{loc}}(\bR^{d+1})$, we have
\begin{multline*}
\left( |u_t - \left(u_t\right)_{Q_r(X_0)} |\right)_{Q_r(X_0)}
+ \sum_{|\alpha| \le 2, \alpha_1 \le 1} \lambda^{1-\frac{|\alpha|}{2}} \left( |D^\alpha u - \left(D^\alpha u\right)_{Q_r(X_0)} |\right)_{Q_r(X_0)}
\\
\le N \kappa^{-1} \sum_{|\alpha| \le 2} \lambda^{1-\frac{|\alpha|}{2}}\left( |D^\alpha u|^q \right)_{Q_{\kappa r}(X_0)}^{\frac 1 q} +  N \kappa^{\frac{d+2}{q}} \left( |f|^q\right)_{Q_{\kappa r}(X_0)}^{\frac 1 q},
\end{multline*}
where $N = N(d,\delta,q)$.
\end{lemma}

\begin{proof}
As in the proof of Lemma \ref{lem0528_01},
we only need to prove the following interior H\"older estimate
\begin{equation}
                            \label{eq9.58}
[u_t]_{\cC^1(Q_1)}+[DD_{\hat{x}} u]_{\cC^1(Q_1)}\le N \|D^2 u\|_{L_q(Q_4)},
\end{equation}
when a sufficiently smooth $u$ satisfies \eqref{eq0525_01} in $Q_4$ with $\lambda=0$ and $f=0$. To this end, we first recall the $W^{1,2}_q$-estimates established in \cite{MR2300337, MR2833589} for parabolic equations with partially BMO coefficients, which, in particular, can be applied to the equation considered in this lemma. Using the $W_q^{1,2}$-estimates together with a localization and bootstrap argument yields
\begin{equation}
                                \label{eq10.12}
\|u\|_{W^{1,2}_p(Q_1)}\le N\|u\|_{L_q(Q_2)}
\end{equation}
for any $p\in (1,\infty)$. By the parabolic Sobolev imbedding theorem, by choosing a sufficiently large $p>d+2$, we see that
\begin{equation}
                        \label{eq10.13}
\|u\|_{L_\infty(Q_1)}+\|Du\|_{L_\infty(Q_1)}\le N\|u\|_{L_q(Q_2)}
\end{equation}
Since $u_t$ and $D_{\hat{x}}u$ satisfy the same equation as $u$, by \eqref{eq10.13} with $u_t$ and $D_{\hat{x}}u$ in place of $u$ and using \eqref{eq10.12} with a scaling, we reach
\begin{equation}
                        \label{eq10.14}
\|u_t\|_{L_\infty(Q_1)}+\|DD_{\hat{x}}u\|_{L_\infty(Q_1)}\le N\|u_t\|_{L_q(Q_2)}+N\|D_{\hat{x}}u\|_{L_q(Q_2)}\le N\|u\|_{L_q(Q_4)}.
\end{equation}
It follows from the equation that $a^{11}D_{11}u=u_t-\sum_{(i,j)\neq (1,1)}a^{ij}D_{ij}u$. Therefore, from \eqref{eq10.14} as well as \eqref{eq10.13},
\begin{equation}
                        \label{eq10.17}
\|u_t\|_{L_\infty(Q_1)}+\|Du\|_{L_\infty(Q_1)}+\|D^2u\|_{L_\infty(Q_1)}\le N\|u\|_{L_q(Q_4)}.
\end{equation}
Again with $u_t$ and $D_{\hat{x}}u$ in place of $u$, it follows easily from \eqref{eq10.17} that
\begin{equation}
                            \label{eq9.58b}
[u_t]_{\cC^1(Q_1)}+[DD_{\hat{x}} u]_{\cC^1(Q_1)}\le \|u\|_{L_q(Q_4)}.
\end{equation}
To conclude \eqref{eq9.58}, it suffices to replace $u$ by $u-(u)_{Q_4}-x_i(D_iu)_{Q_4}$ in \eqref{eq9.58b} and then apply a parabolic Poincar\'e type inequality \cite[Lemma 4.2.2]{MR2435520}. The lemma is proved.
\end{proof}

The following assumption reads that the coefficients are merely measurable in $x_1$ and have small mean oscillations with respect to $(t,x')$ in small cylinders.

\begin{assumption}[$\gamma$]
							\label{assum0528_1}
Let $\gamma \in (0,1)$. There exists $R_0 \in (0,\infty)$ such that
$$
\sum_{i,j=1}^n (a^{ij})^{\#,2}_{R_0} \le \gamma.
$$
\end{assumption}				

\begin{lemma}
							\label{lem0610_3}
Let $\lambda \ge 0$, $q \in (1,\infty)$, $\kappa \ge 8$, $\mu, \nu \in (1,\infty)$, $1/\mu+1/\nu = 1$, and $L$ be the operator in \eqref{eq0609_01}.
Suppose that $b^i = c = 0$.
Then, under Assumption \ref{assum0528_1} ($\gamma$), for any $r \in (0,R_0/\kappa]$,  $X_0 \in \bR^{d+1}$, and $u \in W_{{q\mu},\operatorname{loc}}^{1,2}(\bR^{d+1})$ satisfying \eqref{eq0525_01} in $Q_{\kappa r}(X_0)$, where $f \in L_{q,\operatorname{loc}}(\bR^{d+1})$, we have
\begin{multline*}
\left( |u_t - \left(u_t\right)_{Q_r(X_0)} |\right)_{Q_r(X_0)}
+ \sum_{|\alpha| \le 2, \alpha_1 \le 1} \lambda^{1-\frac{|\alpha|}{2}} \left( |D^\alpha u - \left(D^\alpha u\right)_{Q_r(X_0)} |\right)_{Q_r(X_0)}
\\
\le N \kappa^{-1} \sum_{|\alpha| \le 2} \lambda^{1-\frac{|\alpha|}{2}}\left( |D^\alpha u|^q \right)_{Q_{\kappa r}(X_0)}^{\frac 1 q} +  N \kappa^{\frac{d+2}{q}} \left( |f|^q\right)_{Q_{\kappa r}(X_0)}^{\frac 1 q}
\\
+ N \kappa^{\frac{d+2}{q}} \gamma^{\frac{1}{q\nu}} \left( |D^2 u|^{q\mu}\right)_{Q_{\kappa r}(X_0)}^{\frac{1}{q\mu}},
\end{multline*}
where $N = N(d,\delta,q,\mu)$.
\end{lemma}

\begin{proof}
We use Lemma \ref{lem0610_1} and follow the same steps as we derive Lemma \ref{lem0604_01} from Lemma \ref{lem0528_01}.
\end{proof}

\begin{lemma}
							\label{lem0610_4}
Let $\lambda \ge 0$, $K_0 \ge 1$, $p, q \in (1,\infty)$, $t_1 \in \bR$, $w = w_1(x') w_2(t,x'')$, where
$$
w_1(x') \in A_p(\bR^{d_1}, dx'),
\quad
w_2(t,x'') \in A_q(\bR \times \bR^{d_2}, dx'' \, dt),
$$
$$
d_1 + d_2 = d,\quad [w_1]_{A_p}\le K_0,\quad [w_2]_{A_q}\le K_0,
$$
and let $L$ be the operator in \eqref{eq0609_01}.
Suppose that $b^i = c = 0$.
Then there exist constants
$$
\begin{aligned}
&\gamma = \gamma(d, \delta, p, q, d_1, d_2, K_0) \in (0,1),
\\
&R_1 = R_1(d, \delta, p, q, d_1, d_2, K_0) \in (0,1),
\end{aligned}
$$
such that, under Assumption \ref{assum0528_1} ($\gamma$),
for $u \in W_{p,q,w}^{1,2}(\bR^{d+1})$ vanishing outside $\left(t_1 - (R_0R_1)^2, t_1\right) \times \bR^d$ and satisfying \eqref{eq0525_01} in $\bR^{d+1}$, where $f \in L_{p,q,w}(\bR^{d+1})$,
we have
\begin{equation}
							\label{eq0610_12}
\sum_{|\alpha| \le 2} \lambda^{1-\frac{|\alpha|}{2}} \| D^\alpha u \|_{L_{p, q, w}}
\le N \|f\|_{L_{p, q, w}},
\end{equation}
where $L_{p,q,w} = L_{p,q,w}(\bR^{d+1})$ and
$N = N (d,\delta, p, q, d_1, d_2, K_0)$.
\end{lemma}

\begin{proof}
Let $\kappa \ge 8$.
Choose $q_0, \mu \in (1,\infty)$ depending only on $p$, $q$, $d_1$, $d_2$, and $K_0$ and  satisfying \eqref{eq0610_07} as in the proof of Lemma \ref{lem0607_1}.
Then from Lemma \ref{lem0610_3},  by repeating the steps for obtaining \eqref{eq0610_08} in the proof of Lemma \ref{lem0607_1}, we get
\begin{multline}
							\label{eq0610_06}
\|u_t\|_{L_{p,q,w}} + \sum_{|\alpha|\le 2, \alpha_1 \le 1} \lambda^{1-\frac{|\alpha|}{2}} \|D^\alpha u\|_{L_{p,q,w}}
\\
\le N \left( \kappa^{2(1-\frac{1}{q_0})} R_1^{2(1-\frac{1}{q_0})}  + \kappa^{-1}\right) \sum_{|\alpha|\le 2} \lambda^{1-\frac{|\alpha|}{2}} \|D^\alpha u\|_{L_{p,q,w}}
\\
+ N \kappa^{\frac{d+2}{q_0}} \|f\|_{L_{p,q,w}} + N \kappa^{\frac{d+2}{q_0}} \gamma^{\frac{1}{q_0\nu}} \|D^2u\|_{L_{p,q,w}},
\end{multline}
where $N = N(d,\delta,p,q,d_1, d_2, K_0)$ and $1/\mu+1/\nu=1$.
On the other hand, the equation \eqref{eq0525_01} along the fact that $1/a^{11} \le \delta^{-1}$ shows that
$$
\|D_1^2 u\|_{L_{p,q,w}} \le N \|u_t\|_{L_{p,q,w}} + N \sum_{|\alpha| \le 2, \alpha_1 \le 1} \|D^\alpha u\|_{L_{p,q,w}} + N \|f\|_{L_{p,q,w}},
$$
where $N=N(\delta)$.
Upon combining this and \eqref{eq0610_06}, we see that the left-hand side of \eqref{eq0610_12} is bounded by the right-hand side of \eqref{eq0610_06}.
Fix $\kappa \ge 8$ so that $N \kappa^{-1} \le 1/6$. Then choose $\gamma \in (0,1)$ and $R_1 \in (0,1)$ so that
$$
N \kappa^{2(1-\frac{1}{q_0})} R_1^{2(1-\frac{1}{q_0})} \le 1/6,
\quad
N \kappa^{\frac{d+2}{q_0}} \gamma^{\frac{1}{q_0\nu}} \le 1/6.
$$
Then the inequality \eqref{eq0610_12} follows.
\end{proof}

Next we consider equations with leading coefficients merely measurable in $(t,x_1)$ except for $a^{11}$, which is a measurable function either in $t$ or in $x_1$.

\begin{lemma}
							\label{lem0610_2}
Let $\lambda \ge 0$, $q \in (1,\infty)$, $\kappa \ge 8$, and $L$ be the operator in \eqref{eq0609_01}.
Suppose that
$$
a^{11} = a^{11}(t) \quad
\text{or}
\quad
a^{11} = a^{11}(x_1),
$$
$$
a^{ij} = a^{ij}(t,x_1),
\quad
(i,j) \ne (1,1),
$$
and $b^i = c = 0$.
Then, for any $r \in (0,\infty)$, $X_0 \in \bR^{d+1}$, and $u \in W_{q,\operatorname{loc}}^{1,2}(\bR^{d+1})$ satisfying \eqref{eq0525_01} in $Q_{\kappa r}(X_0)$, where $f \in L_{q,\operatorname{loc}}(\bR^{d+1})$, we have
\begin{multline*}
\sum_{|\alpha| \le 2, \alpha_1 \le 1} \lambda^{1-\frac{|\alpha|}{2}} \left( |D^\alpha u - \left(D^\alpha u\right)_{Q_r(X_0)} |\right)_{Q_r(X_0)}
\\
\le N \kappa^{-\frac{1}{2}} \sum_{|\alpha| \le 2} \lambda^{1-\frac{|\alpha|}{2}}\left( |D^\alpha u|^q \right)_{Q_{\kappa r}(X_0)}^{\frac 1 q} +  N \kappa^{\frac{d+2}{q}} \left( |f|^q\right)_{Q_{\kappa r}(X_0)}^{\frac 1 q},
\end{multline*}
where $N = N(d,\delta,q)$.
\end{lemma}

\begin{proof}
As before, it suffices to prove the following interior H\"older estimate
\begin{equation}
                            \label{eq9.58c}
[DD_{\hat{x}} u]_{\cC^{1/2}(Q_1)}\le N\|D^2 u\|_{L_q(Q_4)},
\end{equation}
when a sufficiently smooth $u$ satisfies \eqref{eq0525_01} in $Q_4$ with $\lambda=0$ and $f=0$.
By \eqref{eq10.12} with $p=2(d+2)$, which is also applicable to the equation considered in this lemma (cf. \cite{MR2833589}), we get
\begin{equation}
                            \label{eq10.47}
[Du]_{\cC^{1/2}(Q_1)}\le N\|u\|_{L_q(Q_2)}.
\end{equation}
Noting that $D_{\hat{x}}u$ satisfies the same equation as $u$, from \eqref{eq10.47} and \eqref{eq10.12}, we have
\begin{equation*}
[DD_{\hat{x}}u]_{\cC^{1/2}(Q_1)}\le N\|D_{\hat{x}}u\|_{L_q(Q_2)}\le N\|u\|_{L_q(Q_4)}.
\end{equation*}
To conclude \eqref{eq9.58c}, it suffices to replace $u$ by $u-(u)_{Q_4}-x_i(D_iu)_{Q_4}$ as in the proof of Lemma \ref{lem0610_1}.
\end{proof}

Now we deal with the case with either Assumption \ref{assum0522_1} or \ref{assum0522_2}.
In doing so, we use the results (Lemmas \ref{lem0607_1} and \ref{lem0610_4}) under stronger assumptions (Assumptions \ref{assum0528_2} and \ref{assum0528_1}).

\begin{lemma}
							\label{lem0610_5}
Let $\lambda \ge 0$, $K_0 \ge 1$, $p, q \in (1,\infty)$, $t_1 \in \bR$,
$w = w_1(x') w_2(t,x'')$, where
$$
w_1(x') \in A_p(\bR^{d_1}, dx'),
\quad
w_2(t,x'') \in A_q(\bR \times \bR^{d_2}, dx'' \, dt),
$$
$$
d_1+d_2=d,\quad [w_1]_{A_p}\le K_0,\quad [w_2]_{A_q}\le K_0,
$$
and let $L$ be the operator in \eqref{eq0609_01}.
Suppose that $b^i = c = 0$.
Then there exist constants
$$
\begin{aligned}
&\gamma = \gamma(d, \delta, p, q, d_1, d_2, K_0) \in (0,1),
\\
&R_1 = R_1(d, \delta, p, q, d_1, d_2, K_0) \in (0,1),
\end{aligned}
$$
such that, under Assumption \ref{assum0522_1} ($\gamma$) or Assumption \ref{assum0522_2} ($\gamma$),
for $u \in W_{p,q,w}^{1,2}(\bR^{d+1})$ vanishing outside $\left(t_1 - (R_0R_1)^2, t_1\right) \times \bR^d$ and satisfying \eqref{eq0525_01} in $\bR^{d+1}$, where $f \in L_{p,q,w}(\bR^{d+1})$,
we have
\begin{equation}
							\label{eq0610_13}
\sum_{|\alpha| \le 2} \lambda^{1-\frac{|\alpha|}{2}} \| D^\alpha u \|_{L_{p, q, w}}
\le N \|f\|_{L_{p, q, w}},
\end{equation}
where $L_{p,q,w} = L_{p,q,w}(\bR^{d+1})$ and
$N = N (d,\delta, p, q, d_1, d_2,K_0)$.
\end{lemma}

\begin{proof}
Let $\kappa \ge 8$.
As in the proof of Lemma \ref{lem0610_4}, we choose $q_0, \mu \in (1,\infty)$ and derive from Lemma \ref{lem0610_2} the following inequality
\begin{multline}
							\label{eq0610_09}
\sum_{|\alpha|\le 2, \alpha_1 \le 1} \lambda^{1-\frac{|\alpha|}{2}} \|D^\alpha u\|_{L_{p,q,w}}
\\
\le N \left( \kappa^{2(1-\frac{1}{q_0})} R_1^{2(1-\frac{1}{q_0})}  + \kappa^{-\frac{1}{2}}\right) \sum_{|\alpha|\le 2} \lambda^{1-\frac{|\alpha|}{2}} \|D^\alpha u\|_{L_{p,q,w}}
\\
+ N \kappa^{\frac{d+2}{q_0}} \|f\|_{L_{p,q,w}} + N \kappa^{\frac{d+2}{q_0}} \gamma^{\frac{1}{q_0\nu}} \|D^2u\|_{L_{p,q,w}},
\end{multline}
where $N = N(d,\delta,p,q,d_1, d_2, K_0)$ and $1/\mu+1/\nu=1$.

Now we write the equation \eqref{eq0525_01} as
$$
-u_t + a^{11} D_1^2 u + \Delta_{d-1}u = \bar{f},
$$
where
$$
\bar{f} = f + \Delta_{d-1} u - \sum_{i,j, (i,j) \ne (1,1)}^d a^{ij} D_{ij} u.
$$
The coefficients of the operator $a^{11} D_1^2  + \Delta_{d-1}$, in particular, the coefficient $a^{11}$ satisfies either Assumption \ref{assum0528_2} or Assumption \ref{assum0528_1}.
Then as long as $\gamma$ and $R_1$ are smaller than those in Lemma \ref{lem0607_1} with $m=\ell=1$ or those in Lemma \ref{lem0610_4}, we have
$$
\sum_{|\alpha|\le 2} \lambda^{1-\frac{|\alpha|}{2}} \|D^\alpha u\|_{L_{p,q,w}} \le N \|\bar{f}\|_{L_{p,q,w}}
\le N \|f\|_{L_{p,q,w}}  + N \sum_{|\alpha|=2, \alpha_1 \le 1} \|D^\alpha u\|_{L_{p,q,w}},
$$
where $N = N(d,\delta,p,q,d_1,d_2, K_0)$.
This combined with \eqref{eq0610_09} shows that
$$
\sum_{|\alpha|\le 2} \lambda^{1-\frac{|\alpha|}{2}} \|D^\alpha u\|_{L_{p,q,w}} \le N \|f\|_{L_{p,q,w}}
$$
$$
+ N \left( \kappa^{2(1-\frac{1}{q_0})} R_1^{2(1-\frac{1}{q_0})}  + \kappa^{-\frac{1}{2}}\right) \sum_{|\alpha|\le 2} \lambda^{1-\frac{|\alpha|}{2}} \|D^\alpha u\|_{L_{p,q,w}}
$$
$$
+ N \kappa^{\frac{d+2}{q_0}} \|f\|_{L_{p,q,w}} + N \kappa^{\frac{d+2}{q_0}} \gamma^{\frac{1}{q_0\nu}} \|D^2u\|_{L_{p,q,w}},
$$
where $N = N(d,\delta,p,q,d_1,d_2,K_0)$.
Now we fix $\kappa \ge 8$ so that $N \kappa^{-\frac{1}{2}} \le 1/6$.
Then we choose $\gamma$ and $R_1$ so that they are less than those in Lemma \ref{lem0607_1} with $m=\ell=1$ or in Lemma \ref{lem0610_4}, and satisfy
$$
N \kappa^{2(1-\frac{1}{q_0})} R_1^{2(1-\frac{1}{q_0})} \le 1/6,
\quad
N \kappa^{\frac{d+2}{q_0}} \gamma^{\frac{1}{q_0\nu}} \le 1/6.
$$
Then the inequality \eqref{eq0610_13} follows.
\end{proof}

Now we are ready prove the main theorems of this section.

\begin{proof}[Proof of Theorem \ref{thm0522_1}]
We use the partition of unity argument in the proof of Proposition \ref{prop0609_1} and proceed as in the proof of Theorem \ref{thm0607_1}.
\end{proof}

Note that the extension argument in the proof below is possible because the coefficients are allowed to have no regularity assumptions with respect to one spatial variable. Thus the argument is not applicable if coefficients are continuous or have vanishing (or small) mean oscillations in all the spatial variables.

\begin{proof}[Proof of Theorem \ref{thm5.4}]
We take the even extensions of $w_1$ (or $w_2$ if $d_1=0$), $a^{11}$, $a^{ij}$, $b^i$, $c$ for $i,j\ge 2$  with respect to $x_1=0$, and the odd extensions of $a^{1j}$, $a^{j1}$ for $j\ge 2$ and $b^1$ with respect to $x_1=0$. It is easily seen that $w_1$ (or $w_2$) is an $A_p$ (or $A_q$) weight in the whole space, and $a^{ij}$ satisfy Assumption \ref{assum0522_1} or \ref{assum0522_2} in the whole space.

In the case of the Dirichlet boundary condition, we take the odd extensions of $u$ and $f$ with respect to $x_1=0$, while in the case of the Neumann boundary condition, we take the even extensions of $u$ and $f$. Then $u$ satisfies \eqref{eq0525_01} in $\bR^{d+1}$. By applying Theorem \ref{thm0522_1} and noting that the norms in the half space are comparable to those in the whole space, we immediately get Theorem \ref{thm5.4}.
\end{proof}

\section{Higher order systems in divergence form in Reifenberg domains with partially BMO coefficients}
							\label{sec07}

Set
\begin{equation}
							\label{eq0625_01}
\cL u = \sum_{|\alpha| \le m, |\beta| \le m} D^\alpha \left( A^{\alpha\beta} D^\beta u \right),
\end{equation}
where the $\ell \times \ell$ matrices $A^{\alpha\beta}$ are complex valued functions on $\bR^{d+1}$ and $u$ is a complex vector-valued function.
The coefficients $A^{\alpha\beta}$ are bounded and satisfy the strong ellipticity condition as follows.

\begin{enumerate}

\item[(i)] There exists a constant $\delta \in (0,1)$ such that $|A^{\alpha\beta}| \le \delta^{-1}$, $|\alpha|=|\beta|=m$, and
$$
\sum_{|\alpha|=|\beta|=m} \Re \left( A^{\alpha\beta} \xi_\beta, \xi_\alpha \right) \ge \delta |\xi|^2
$$
for any $\xi = (\xi_\alpha)_{|\alpha|=m}$, $\xi_\alpha \in \bC^\ell$.

\item[(ii)] There exists a constant $K > 0$ such that $|A^{\alpha\beta}| \le K$ if $|\alpha| < m$ or $|\beta| < m$.

\end{enumerate}

In this section we consider a domain of the form $\bR \times \Omega$, where $\Omega \subset \bR^d$ is a Reifenberg flat domain.
We impose the following regularity assumption on the coefficients $A^{\alpha\beta}$, $|\alpha| = |\beta| = m$, and the boundary of the domain $\Omega$.

\begin{assumption}[$\gamma$]
							\label{assum0625_1}
Let $\gamma \in (0,1/4)$.
There exists $R_0 \in (0,1]$ satisfying the following.

\begin{enumerate}

\item[(i)] For any $X = (t ,x) \in \bR \times \Omega$ and $r \in (0, \min\{ R_0, \text{dist}(x, \partial \Omega)/2 \}]$ (so that $B_r(x) \subset \Omega$), there is a spatial coordinate system depending on $x$ and $r$ such that in this new coordinate system, we have
\begin{equation}
							\label{eq0625_04}
\dashint_{Q_r(t,x)} \left| A^{\alpha\beta}(s,y_1,\hat y) - \dashint_{Q'_r(t,\hat x)} A^{\alpha\beta}(\tau,y_1,\hat z) \, d\hat z \, d\tau \right| \, dy \, ds \le \gamma.
\end{equation}

\item[(ii)] For any $x \in \partial \Omega$, $t \in \bR$, and any $r \in (0,R_0]$, there is a spatial coordinate system depending on $X = (t,x)$ and $r$ such that in this new coordinate system, we have \eqref{eq0625_04} and
$$
\{ (y_1, \hat y) : x_1 + \gamma r < y_1 \} \cap B_r(x) \subset \Omega \cap B_r(x)
\subset \{(y_1,\hat y) : x_1 - \gamma r < y_1\} \cap B_r(x).
$$

\end{enumerate}					
\end{assumption}

For the mixed-norm case, we view a Reifenberg flat domain $\Omega$ as a subset of $\Omega_1 \times \Omega_2$, where $\Omega_1 \subset \bR^{d_1}$, $\Omega_2 \subset \bR^{d_2}$, and $d_1 + d_2 = d$.
We assume that $\Omega_1$ and $\bR \times \Omega_2$ are spaces of homogeneous type with the usual Lebesgue measures. The metrics are the Euclidean distance in $\Omega_1$ and the parabolic distance in $\bR \times \Omega_2$; thus the constant $K_1$ in \eqref{eq0709_1} is $1$.
For example, $\Omega_1$ can be $\bR^{d_1}$.

\begin{theorem}
							\label{thm0625_1}
Let $p, q \in (1,\infty)$, $K_0\ge 1$, $\Omega \subset \Omega_1 \times \Omega_2$, $\Omega_1 \subset \bR^{d_1}$, $\Omega_2 \subset \bR^{d_2}$, $d_1+d_2 = d$,
and $\bR \times \Omega$, as a subset of $\Omega_1 \times ( \bR \times \Omega_2 )$, satisfy the conditions (a) and (b) before Corollary \ref{cor1} with a doubling constant $K_2$ (see \eqref{eq0709_2}).
Also let $w = w_1(x') w_2(t,x'')$, where
$$
w_1(x') \in A_p(\Omega_1, dx'),
\quad
w_2(t,x'') \in A_q(\bR \times \Omega_2, dx'' \, dt),
$$
$$
[w_1]_{A_p}\le K_0,\quad [w_2]_{A_q}\le K_0,
$$
and $\cL$ be the operator in \eqref{eq0625_01}.
Then there exist
$$
\gamma = \gamma(d, m, \ell, \delta, p, q, d_1, d_2, K_0,K_2) \in (0,1/4),
$$
$$
\lambda_0 = \lambda_0(d, m, \ell, \delta, p, q, d_1, d_2, K_0,K_2, K, R_0) \ge 1,
$$
such that,
under Assumption \ref{assum0625_1} ($\gamma$),
for $u \in \mathring{\cH}_{p,q,w}^m(\bR \times \Omega)$ satisfying
\begin{equation}
							\label{eq0625_02}
u_t + (-1)^m \cL u + \lambda u = \sum_{|\alpha| \le m} D^\alpha f_\alpha
\end{equation}
in $\bR \times \Omega$, where $f_\alpha \in L_{p,q,w}(\bR \times \Omega)$,
we have
\begin{equation}
							\label{eq0625_03}
\sum_{|\alpha| \le m} \lambda^{1-\frac{|\alpha|}{2m}} \| D^\alpha u \|_{L_{p,q,w}} \le
N \sum_{|\alpha| \le m} \lambda^{\frac{|\alpha|}{2m}} \|f_\alpha\|_{L_{p,q,w}},
\end{equation}
provided that $\lambda \ge \lambda_0$,
where $L_{p,q,w} = L_{p,q,w}(\bR \times \Omega)$ and
$$
N = N(d,m,\ell,\delta, p, q, d_1, d_2, K_0, ,K_2, K, R_0).
$$
\end{theorem}
\begin{remark}
One can check that the conditions (a) and (b) before Corollary \ref{cor1} are satisfied if, for instance, $\Omega$ is a bounded Reifenberg flat domain in $\bR^d$.
The doubling constant $K_2$ is then determined only by $d$, $m$, $R_0$, and $|\Omega|$ as long as $\gamma$ is sufficiently small, for instance, $\gamma \in (0,1/4)$.
Indeed, the doubling inequality \eqref{eq0709_2} follows from
\begin{equation}
							\label{eq0709_3}
N_1 |Q_r(X_0)| \le \left| \cB_r(X_0) \right| \le 2|Q_r(X_0)|
\end{equation}
for $r/4 \in (0, R_0]$, and
\begin{equation}
							\label{eq0709_4}
N_2 r^{2m}|\Omega| \le \left| \cB_r(X_0) \right| \le 2 r^{2m}|\Omega|
\end{equation}
for $r/4 \in (R_0,\infty)$, where
$$
X_0 = (t_0,x_0) \in \bR \times \Omega,
\quad N_1 = N_1(d,m),\quad
N_2 = N_2(d, m, R_0, |\Omega|),
$$
and $\cB_r(X_0)$ is a ball in $\bR \times \Omega$ using the parabolic distance, i.e.,
$$
\cB_r(X_0) = \{ (t,x) \in \bR \times \Omega: |x-x_0| + |t-t_0|^{\frac{1}{2m}} < r \}.
$$
To verify the above inequalities (we show only the lower bounds), note that
$$
Q_{r/2}(X_0) \cap (\bR \times \Omega) \subset \cB_r(X_0),
$$
$$
Q_{r/2}(X_0) \cap (\bR \times \Omega) \supset
\left\{
\begin{aligned}
&Q_{r/4}(X_0)
\,\quad\quad\quad\quad\quad\quad \text{if} \quad \text{dist}(x_0, \partial\Omega) > r/4,
\\
&Q_{r/4}(t_0,\tilde{x}) \cap (\bR \times \Omega)
\quad \text{if} \quad \text{dist}(x_0, \partial\Omega) \le r/4,
\end{aligned}
\right.
$$
where $\tilde{x} \in \partial \Omega$ such that $|x_0 - \tilde{x}| = \text{dist}(x_0, \partial\Omega)$.

We assume that $\text{dist}(x_0, \partial\Omega) \le r/4$ as the other case is simpler.
If $r/4 \le R_0$, then, since $\gamma < 1/4$, by the property of Reifenberg flat domains, we have
$$
|Q_{r/4}(X_0)| \ge |Q_{r/4}(t_0,\tilde{x}) \cap (\bR \times \Omega)| \ge N(d,m)|Q_r(X_0)|.
$$
Hence the first equality in \eqref{eq0709_3} follows.
If $r/4 > R_0$, then again by the property of Reifenberg flat domains
$$
|Q_{r/4}(X_0)| \ge |Q_{r/4}(t_0,\tilde{x}) \cap (\bR \times \Omega)| \ge (r/4)^{2m} |B_{R_0}(\tilde{x}) \cap \Omega|
$$
$$
\ge N(d,m) r^{2m} |B_{R_0}| = N(d, m, R_0, |\Omega|) r^{2m} |\Omega|.
$$
This verifies the first equality in \eqref{eq0709_4}.
\end{remark}

For given constant $\lambda \ge 0$ and functions  $u$ and $f_\alpha$, $|\alpha| \le m$,
we write
\begin{equation}
							\label{eq0625_05}
\begin{aligned}
U &= \big(\lambda^{\frac{1}{2}-\frac{|\alpha|}{2m}} D^\alpha u\big)_{|\alpha| \le m}, \quad U'=\big(\lambda^{\frac{1}{2}-\frac{|\alpha|}{2m}} D^\alpha u\big)_{|\alpha| \le m, \alpha_1 < m},
\\
\Theta &= \sum_{|\beta|=m} A^{\check{\alpha}\beta} D^\beta u,
\quad\quad\quad\,\,
F=\big(\lambda^{\frac{|\alpha|}{2m}-\frac{1}{2}} f_\alpha \big)_{|\alpha| \le m},
\end{aligned}
\end{equation}
where $\check{\alpha} = m e_1 = (m, 0, \ldots, 0)$. In what follows, we assume that $f_\alpha \equiv 0$ for $|\alpha|<m$ whenever $\lambda = 0$.
By using the strong ellipticity condition, we have
$$
N^{-1} |U| \le |U'| + |\Theta| \le N |U|,
$$
where $N = N(d,m,\ell,\delta)$.

We start with the following interior and boundary H\"older estimates.

\begin{lemma}
                                    \label{lem6.4}
Let $\lambda \ge 0$, $q \in (1,\infty)$, and $\cL$ be the operator in \eqref{eq0625_01}.
Suppose that the coefficients $A^{\alpha\beta}$, $|\alpha|=|\beta|=m$, are measurable functions of only $x_1 \in \bR$, i.e., $A^{\alpha\beta} = A^{\alpha\beta}(x_1)$ and the lower-order coefficients of $\cL$ are all zero.

\begin{enumerate}
\item[(i)] For any $u \in C_{\operatorname{loc}}^\infty(\bR^{d+1})$ satisfying \eqref{eq0625_02} in $Q_2$ with $f_\alpha=0$,
we have
$$
\|U'\|_{C^{1-1/q}(Q_1)}+\|\Theta\|_{C^{1-1/q}(Q_1)}
\le N \|U\|_{L_q(Q_2)},
$$
where $N = N(d,m,\ell,\delta,q)$.

\item[(ii)] For any $u \in C_{\operatorname{loc}}^\infty(\overline{\bR^{d+1}_+})$ satisfying \eqref{eq0625_02} in $Q_2^+$ with $f_\alpha=0$ and $u=D_{\hat x}u=\cdots=D_{\hat x}^{m-1}u=0$ on $Q_2\cap\{x_1=0\}$,
we have
$$
\|U'\|_{C^{1-1/q}(Q_1^+)}+\|\Theta\|_{C^{1-1/q}(Q_1^+)}
\le N \|U\|_{L_q(Q_2^+)}
$$
where $N = N(d,m,\ell,\delta,q)$.
\end{enumerate}
\end{lemma}

\begin{proof}
The proof follows those of Corollary 4.2, Lemma 4.6, Corollary 7.6, and Lemma 7.7 of \cite{MR2835999} by using the interior and boundary $\cH^m_q$ estimates, i.e., Theorems 2.2 and 2.4 in the same paper, instead of the $\cH^m_2$ estimates.
\end{proof}

Next we derive the following mean oscillation estimate.
Denote
$$
\cC_r(X_0) = (\bR \times \Omega) \cap Q_r(X_0).
$$

\begin{lemma}
                                        \label{lem6.5}
Let $\lambda \ge 0$, $q \in (1,\infty)$, $\kappa \ge 64$, $\mu, \nu \in (1,\infty)$, $1/\mu + 1/\nu = 1$, and $\cL$ be the operator in \eqref{eq0625_01}.
Suppose that the lower-order coefficients of $\cL$ are all zero.
Then, under Assumption \ref{assum0625_1} ($\gamma$) with $\gamma < 1/(4\kappa)$,
for $r \in (0, R_0/\kappa]$, $X_0 = (t_0,x_0) \in \bR^{d+1}$ with $x_0 \in \overline{\Omega}$, and
$u \in \mathring{\cH}_{q\mu, \operatorname{loc}}^m(\bR \times \Omega)$ satisfying \eqref{eq0625_02} in $\cC_{\kappa r}(X_0)$, where $f_\alpha \in L_{q, \operatorname{loc}}(\cC_{\kappa r}(X_0))$,
there exists $U^{\cC}$ on $\cC := \cC_{\kappa r}(X_0)$ such that, on $\cC$,
$$
N^{-1} |U| \le |U^{\cC}| \le N |U|
$$
and
\begin{align}
&\left( |U^{\cC} - (U^{\cC})_{\cC_r(X_0)}|\right)_{\cC_r(X_0)}
\le N (\kappa^{\frac 1 q-1} + \kappa \gamma) \left( |U|^q \right)_{\cC_{\kappa r}(X_0)}^{\frac{1}{q}}\nonumber\\
                                    \label{eq3.33}
&\,\,+ N \kappa^{\frac{d+2m}{q}} \left(|F|^q\right)_{\cC_{\kappa r}(X_0)}^{\frac{1}{q}}
+ N \kappa^{\frac{d+2m}{q}} \gamma^{\frac{1}{q\nu}} \left( |U|^{q\mu} \right)_{\cC_{\kappa r}(X_0)}^{\frac{1}{q\mu}},
\end{align}
where $N=N(d,m,\ell, \delta, q, \mu)$.
\end{lemma}

\begin{proof}
We follow the lines in the proof of Proposition 7.10 in \cite{MR2835999}, where $\Omega=\bR^d_+$ and $q=2$.
As mentioned in Remark \ref{rem0713_1}, we assume that the coefficients are infinitely differentiable.
We further assume $\lambda > 0$. Otherwise, we add the term $\varepsilon u$, $\varepsilon > 0$, to both sides of \eqref{eq0625_02} and obtain \eqref{eq3.33} for the modified system.
Then we let $\varepsilon \searrow 0$.

Let $\tilde x\in \partial\Omega$ be such that $|x_0 -\tilde x|=\rho:=\text{dist}(x_0,\partial\Omega)$.
We consider two cases.

{\em Case 1: $\rho\ge \kappa r/16$.} In this case, we have
$$
\cC_r(X_0)=Q_r(X_0)\subset Q_{\kappa r/16}(X_0)\subset \bR\times\Omega.
$$
Since $\kappa/16\ge 4$, \eqref{eq3.33} follows from Lemma \ref{lem6.4} (i) by using a scaling and rotation of coordinates. See, for instance, the proof of Theorem 6.1 in \cite{MR2835999}.

{\em Case 2: $\rho< \kappa r/16$.} Without loss of generality, one may assume $(t_0,\tilde x)=(0,0)$. Note that
\begin{equation}
                                \label{eq22.15}
\cC_r(X_0)\subset \cC_{\kappa r/8} \subset \cC_{\kappa r/2}
\subset \cC_{\kappa r}(X_0).
\end{equation}
Denote $R=\kappa r/2(< R_0)$. Due to Assumption \ref{assum0625_1}, by taking an orthogonal transformation if necessary, we have
$$
 \{(x_1, \hat{x}):\gamma R< x_1\}\cap B_R
 \subset\Omega\cap B_{R}
 \subset \{(x_1, \hat{x}):-\gamma R<x_1\}\cap B_{R},
$$
and
\begin{equation}
                                \label{eq17_50}
\sup_{|\alpha|=|\beta|=m}\dashint_{Q_R} \left| A^{\alpha\beta}(t,x_1,\hat x) - \bar A^{\alpha\beta}(x_1)\right| \, dx \, dt \le \gamma,
\end{equation}
where
\begin{equation}
							\label{eq0715_01}
\bar A^{\alpha\beta}(x_1)=\dashint_{Q'_R} A^{\alpha\beta}(\tau,x_1,\hat y)\,d\hat y\,d\tau.
\end{equation}
Take a smooth function $\chi$ on $\bR$ such that
$$
\chi(x_1)\equiv 0\quad\text{for}\,\,x_1\le \gamma R,
\quad \chi(x_1)\equiv 1\quad\text{for}\,\,x_1\ge  2\gamma R,
$$
$$
|D^k \chi|\le N(\gamma R)^{-k}\quad\text{for}\,\,k=1,2,...,m.
$$
Let $\hat u=\chi u$, which vanishes on $Q_R\cap \{x_1\le \gamma R\}$. As in \cite[Lemma 7.11]{MR2835999}, $\hat u$ satisfies in $Q_{R}^{\gamma+}:=Q_{R}\cap \{x \in \bR^d : x_1> \gamma R\}$,
$$
\hat u_t+(-1)^m \cL_0 \hat u+\lambda \hat u=(-1)^m \sum_{|\alpha|=|\beta|=m}D^\alpha\left(
(\bar A^{\alpha\beta}- A^{\alpha\beta})D^\beta  u\right)
$$
\begin{equation}
                                    \label{eq17.23b}
+ \sum_{|\alpha|\le m}\chi D^\alpha  f_\alpha+(-1)^m g+(-1)^m h,
\end{equation}
where $\cL_0$ is the differential operator with the coefficients $\bar A^{\alpha\beta}$ from \eqref{eq17_50}, and
\begin{align*}
 g&=\sum_{|\alpha|=|\beta|=m}
D^\alpha\big(\bar A^{\alpha\beta} D^{\beta}((\chi-1) u)\big),\\
 h&=(1-\chi)\sum_{|\alpha|=|\beta|=m}
D^\alpha( A^{\alpha\beta} D^{\beta} u).
\end{align*}
Now let $\hat w$ be the unique $\mathring\cH^{m}_q(\bR\times \{ x: x_1 >\gamma R\})$ solution of
\begin{multline}
							\label{eq0714_03}
\hat w_t+(-1)^m \cL_0 \hat w+\lambda \hat w
=(-1)^m \sum_{|\alpha|=|\beta|=m}D^\alpha\left(
 \varphi(\bar A^{\alpha\beta}- A^{\alpha\beta})D^\beta  u\right)
\\
+ \sum_{|\alpha|\le m}\chi D^\alpha (\varphi f_\alpha) +(-1)^m\hat  g + (-1)^m \hat h
\end{multline}
in $\bR\times \{ x:x_1>\gamma R\}$, where
$\varphi := I_{Q_R}$ and
\begin{align*}
\hat  g&=\sum_{|\alpha|=|\beta|=m}
D^\alpha\big(\bar A^{\alpha\beta} \varphi D^{\beta}((\chi-1) u)\big),\\
\hat h&=(1-\chi)\sum_{|\alpha|=|\beta|=m}
D^\alpha( A^{\alpha\beta} \varphi D^{\beta} u).
\end{align*}
By using Hardy's inequality, the $\cH^m_q$ estimate, and a duality argument (see Lemma \ref{lem0714_1} in Appendix \ref{appendix01} for details), we have
\begin{equation}
            \label{eq21.52h}
\sum_{k=0}^m\lambda^{\frac 1 2-\frac k {2m}}(I_{Q_{R}^{\gamma +}}|D^k \hat w|^q)_{\cC_{R}}^{\frac 1 q}\le N\gamma^{\frac 1 {q\nu}} (|U|^{q\mu})_{\cC_R}^{\frac 1 {q\mu}}+
N(|F|^q)_{\cC_{R}}^{\frac 1 q}.
\end{equation}
We extend $\hat w$ to be zero in $\cC_R\setminus Q_R^{\gamma +}$, so that $\hat w\in \cH^{m}_2(\cC_R)$, and let $w=\hat w+(1-\chi)u$.
Similar to (7.20) of \cite{MR2835999}, we deduce from \eqref{eq21.52h} that
\begin{equation}
            \label{eq18.34h}
(|W|^q)_{\cC_{R}}^{\frac 1 q}
\le N\gamma^{\frac 1 {q\nu}} (|U|^{q\mu})_{\cC_{R}}^{\frac 1  {q\mu}}+
N(|F|^q)_{\cC_{R}}^{\frac 1 q},
\end{equation}
where $W$ is defined in the same way as $U$ with $w$ in place of $u$.
Noting that $\cC_r(X_0) \subset \cC_R$ and $|\cC_R|/|\cC_r(X_0)|
\le N(d) \kappa^{d+2m}$, from \eqref{eq18.34h} we obtain
\begin{equation}
                                \label{eq28_01}
(|W|^q)_{\cC_r(X_0)}^{\frac 1 q}
\le N\kappa^{\frac {d+2m} q}\gamma^{\frac 1 {q\nu}} (|U|^{q\mu})_{\cC_{R}}^{\frac 1 {q\mu}}+N \kappa^{\frac {d+2m} q}
(|F|^q)_{\cC_{R}}^{\frac 1 q}.
\end{equation}

Next we define $v= u- w$ $(=\chi u-\hat w)$ in $\cC_R$.
From \eqref{eq17.23b} and \eqref{eq0714_03}, it is easily seen that $ v=0$ in $\cC_R\setminus Q_R^{\gamma +}$ and $ v$ satisfies
\begin{equation}
                                \label{eq0902}
 v_t+(-1)^m \cL_0  v+\lambda  v=0
\end{equation}
in $Q_{R/2} \cap \{x: x_1> \gamma R\}$ and vanishes along with its derivatives up to $(m-1)$-th order on $Q_R \cap \{ x: x_1=\gamma R\}$.
Note that since the coefficients of $\cL_0$ are infinitely differentiable and $v$ satisfies \eqref{eq0902} in $Q_{R/2} \cap \{x : x_1 > \gamma R\}$ with the Dirichlet boundary condition, by the classical results, $v$ is infinitely differentiable in $Q_{\tau} \cap \{x: x_1 > \gamma R\}$ for any $\tau < R/2$.
Denote
$$
\cD_1=\cC_{r}(X_0)\cap \{x_1<\gamma R\},
\quad
\cD_2=\cC_{r}(X_0)\setminus \cD_1,
\quad
\cD_3=Q_{R/4}\cap\{x_1 > \gamma R\}.
$$
Because of \eqref{eq22.15},
$|\cD_1|\le N\kappa\gamma|\cC_{r}(X_0)|$.
As in \eqref{eq0625_05}, we set
$$
V' = ( \lambda^{\frac{1}{2} - \frac{|\alpha|}{2m}} D^\alpha v )_{|\alpha|\le m, \alpha_1 < m},
\quad
\hat{\Theta} = \sum_{|\beta|=m} \bar{A}^{\check{\alpha}\beta} (x_1) D^\beta v,
$$
where $\bar{A}^{\check{\alpha}\beta}(x_1)$ are from \eqref{eq0715_01}.
Then applying Lemma \ref{lem6.4} (ii) with a scaling argument ($Q_r^+$, $r = 1, 2$, in Lemma \ref{lem6.4} (ii) can be replaced by  $Q_r^+(X)$ if, for instance, $X = (0,x_1, 0)$ and $-1/2<x_1<1/2$),
we get
\begin{align*}
&\big(|V'-(V')_{\cC_r(X_0)}|\big)_{\cC_r(X_0)}
+\big(|\hat \Theta-(\hat \Theta)_{\cC_r(X_0)}|\big)_{\cC_r(X_0)}\\
&\le N r^{1-\frac 1 q}\big([V']_{C^{1-\frac 1 q}(\cD_2)}
+[\hat \Theta]_{C^{1-\frac 1 q}( \cD_2)}\big)
+N\kappa\gamma \|V\|_{L_\infty(\cD_2)}\\
&\le N r^{1-\frac 1 q}\big([V']_{C^{1-\frac 1 q}(\cD_3)}
+[\hat \Theta]_{C^{1-\frac 1 q}( \cD_3)}\big)
+N\kappa\gamma \|V\|_{L_\infty(\cD_3)}\\
&\le N(\kappa^{-1+\frac 1 q}+\kappa \gamma)(|V|^q)_{\cC_{R/2}}^{1/q},
\end{align*}
which together with \eqref{eq18.34h} and \eqref{eq28_01} yields \eqref{eq3.33}.
Indeed,
we set $U^{\cC}=(U', \Theta)$, where
$$
\Theta := \hat{\Theta} + \sum_{|\beta|=m} \bar{A}^{\check{\alpha}\beta}(x_1) D^\beta w.
$$
Then
$$
\big(|U^\cC-(U^\cC)_{\cC_r(X_0)}|\big)_{\cC_r(X_0)}
\le N\big(|V'-(V')_{\cC_r(X_0)}|\big)_{\cC_r(X_0)}
$$
$$
+ N \big(|\hat \Theta-(\hat \Theta)_{\cC_r(X_0)}|\big)_{\cC_r(X_0)}
+ N \big(|W|\big)_{\cC_r(X_0)}.
$$
Here the last term is estimated by \eqref{eq28_01}  and, as shown above, the first two terms on the right-hand side are estimated by
$N(\kappa^{-1+1/q}+\kappa \gamma)(|V|^q)_{\cC_{R/2}}^{1/q}$, which is taken care of by \eqref{eq18.34h} and the fact that $ u =  v+ w$ in $\cC_R$.
This completes the proof of Lemma \ref{lem6.5}.
\end{proof}

Finally, we complete the proof of Theorem \ref{thm0625_1}.
\begin{proof}[Proof of Theorem \ref{thm0625_1}]
With Lemma \ref{lem6.5} in hand, by using Corollaries \ref{cor1} and \ref{cor3} we prove Theorem \ref{thm0625_1} in the same way as Theorems \ref{thm0607_1}, \ref{thm4.4}, and \ref{thm0522_1} were proved.

In particular, to apply Corollaries \ref{cor1} and \ref{cor3}, we first find a filtration of partitions by using
Theorem \ref{thm0326}.
By the properties of the partitions, for each $Q^n$ in the partitions,
there exist $r \in (0,\infty)$ and $X_0 \in \bR \times \Omega$ such that
$$
Q^n \subset \cC_r(X_0),
\quad |\cC_r(X_0)| \le N |Q^n|,
$$
where $N$ depends on $K_2$ (recall that $K_1 = 1$).
Then we use $U^{\cC}$, $\cC = \cC_{\kappa r}(X_0)$, in place of $f^Q$ in Corollary \ref{cor3}, where $\kappa$ is to be chosen appropriately.
For the right-hand side of the inequality in  Corollary \ref{cor3}, we set
$$
g(Y) = (\kappa^{\frac{1}{q_0} -1} + \kappa \gamma) \left[ \cM(|U|^{q_0})(Y)\right]^{\frac{1}{q_0}}
+ \kappa^{\frac{d+2m}{q_0}} \left[ \cM(|F|^{q_0})(Y)\right]^{\frac{1}{q_0}}
$$
$$
+ \kappa^{\frac{d+2m}{q_0}} \gamma^{\frac{1}{q_0\nu}} \left[ \cM (|U|^{q_0\mu})(Y) \right]^{\frac{1}{q_0\mu}},
$$
where $q_0$ and $\mu$ are chosen as in the proof of Lemma \ref{lem0607_1}.
The rest of the details are omitted.
\end{proof}

\section{Existence of solutions}
                                    \label{sec8}

The a priori estimates proved in the previous sections can be used to derive the existence of solutions to the corresponding equations/systems. In this section, as an example we show the solvability of \eqref{eq0625_02} in the mixed-norm weighted Sobolev spaces.

Throughout this section, we assume that the conditions in Theorem \ref{thm0625_1} are satisfied. Then by reverse H\"older's inequality and the doubling property of $A_p$ weights, we can find a sufficiently large constant $p_1$ and small constants $\varepsilon_1,\varepsilon_2\in (0,1)$ depending only on $d$, $p$, $q$, $K_2$, $[w_1]_{A_p}$, and $[w_2]_{A_q}$ such that
$$
1-\frac p {p_1}=\frac 1 {1+\varepsilon_1},\quad 1-\frac q {p_1}=\frac 1 {1+\varepsilon_2},
$$
and both $w_1^{1+\varepsilon_1}$ and $w_2^{1+\varepsilon_1}$ are locally integrable and satisfy the doubling property. That is, for any $r>0$, $x'_0\in \Omega_1$ and $(t_0,x''_0)\in \bR\times \Omega_2$, we have
\begin{align}
                                \label{eq2.41}
\int_{B_{2r}(x_0')\cap \Omega_1}w_1^{1+\varepsilon_1}\,dx'&\le N_0\int_{B_{r}(x_0')\cap \Omega_1}w_1^{1+\varepsilon_1}\,dx',\\
                                    \label{eq2.42}
\int_{\tilde Q_{2r}(t_0,x''_0)\cap (\bR\times \Omega_2)}w_2^{1+\varepsilon_2}\,dx''\,dt&\le N_0\int_{\tilde Q_{r}(t_0,x''_0)\cap (\bR\times \Omega_2)}w_2^{1+\varepsilon_2}\,dx''\,dt,
\end{align}
where $N_0$ is independent of $r$, $x_0'$, and $(t_0,x_0'')$,
$$
\tilde Q_{r}(t_0,x''_0)=(t_0-r^{2m},t_0+r^{2m})\times B_r(x_0'')
$$
is a double parabolic cylinder in $\bR\times \bR^{d_2}$, and
$B_{r}(x_0')$ (and $B_r(x_0'')$) is a ball in $\bR^{d_1}$ (and $\bR^{d_2}$, respectively).
By H\"older's inequality, it is easily seen that any function $f\in L_{p_1}(\bR\times\Omega)$ is locally in $L_{p,q,w}(\bR\times \Omega)$, and for any $r>0$,
\begin{equation}
                                \label{eq3.39}
\|f\|_{L_{p,q,w}(\cQ_r \cap (\bR\times \Omega))}\le N\|f\|_{L_{p_1}(\cQ_r\cap (\bR\times \Omega))},
\end{equation}
where $\cQ_r := (-r^{2m}, r^{2m}) \times B_r \subset \bR \times \bR^d$ and $N$ also depends on $r$.

Now let $f_\alpha \in L_{p,q,w}(\bR \times \Omega)$ for $|\alpha|\le m$. By the denseness of $C_0^\infty(\bR\times \overline{\Omega})$ in $L_{p,q,w}(\bR\times \Omega)$, for each multi-index $\alpha$, we can find a sequence of smooth functions $\{f_{\alpha,k}\}$ with bounded supports such that
\begin{equation}
                                \label{eq3.16}
f_{\alpha,k}\to f_\alpha\quad\text{in}\,\,L_{p,q,w}(\bR\times \Omega)
\quad\text{as}\,\,k\to \infty.
\end{equation}
Since for each $k=0,1,\ldots$, $f_{\alpha,k}\in L_{p_1}(\bR\times\Omega)$, by the solvability in unweighted Sobolev spaces (see \cite[Theorem 8.2]{MR2835999}), there is a unique solution $u_k\in \mathring{\cH}_{p_1}^m(\bR \times \Omega)$ to
$$
(u_k)_t + (-1)^m \cL u_k + \lambda u_k = \sum_{|\alpha| \le m} D^\alpha f_{\alpha,k}
$$
in $\bR\times \Omega$ provided that
$$
\gamma\le \gamma_1(d,n,m,p_1,\delta)\quad\text{and}\quad \lambda\ge\lambda_1(d,n,m,p_1,\delta,R_0).
$$

We claim that if $\gamma$ is taken to be smaller than $\gamma_1$ as well as the constant $\gamma$ in Theorem \ref{thm0625_1}, and $\lambda\ge \max\{\lambda_0,\lambda_1\}$, then $u_k\in \mathring{\cH}_{p,q,w}^m(\bR \times \Omega)$. Assume for the moment that the claim is proved. Then it follows from the a priori estimate \eqref{eq0625_03} and \eqref{eq3.16} that $\{u_k\}$ is a Cauchy sequence in $\mathring{\cH}_{p,q,w}^m(\bR \times \Omega)$. Let $u$ be its limit. Then by taking the limit of the weak formulation for the equation of $u_k$, it is easily seen that $u$ is a solution to \eqref{eq0625_02}.

It remains to prove the claim. We fix a $k\in \mathbb N$ and assume that $f_\alpha,|\alpha|\le m,$ are supported in $ \cQ_R \cap (\bR\times \Omega)$ for some $R\ge 1$.
By \eqref{eq3.39}, we have
\begin{equation}
                                    \label{eq4.27}
\|D^\alpha u_k\|_{L_{p,q,w}(\cQ_{2R}\cap (\bR\times \Omega))}< \infty, \quad 0 \le |\alpha| \le m.
\end{equation}
For $j\ge 0$, we take a sequence of smooth functions $\eta_j$ such that $\eta_j\equiv 0$ in $\cQ_{2^jR}$, $\eta_j\equiv 1$ outside $\cQ_{2^{j+1}R}$, and
$$
|D^\alpha \eta_j|\le N 2^{-j|\alpha|},\quad |\alpha|\le m,
\quad |(\eta_j)_t|\le N2^{-2mj}.
$$
A simple calculation reveals that $u_k\eta_j\in\mathring{\cH}_{p_1}^m(\bR \times \Omega)$ satisfies
\begin{align*}
&(u_k\eta_j)_t + (-1)^m \cL (u_k\eta_j) + \lambda u_k\eta_j \\
&= u_k (\eta_j)_t+
\sum_{|\alpha| \le m,|\beta|\le m}\sum_{1\le |\tilde \beta|\le |\beta|} D^\alpha \big(A^{\alpha\beta}c_{\beta,\tilde \beta}D^{\tilde \beta}\eta_j D^{\beta-\tilde\beta}u_k\big),
\end{align*}
where $c_{\beta,\tilde \beta}$ are combinatorial numbers. Applying the a priori estimate in \cite[Theorem 8.2]{MR2835999} with $p_1$ instead of $p$ to $u_k\eta_j$, we get
\begin{align*}
&\sum_{|\alpha| \le m} \lambda^{1-\frac{|\alpha|}{2m}} \| D^\alpha (u_k\eta_j) \|_{L_{p_1}(\bR \times \Omega)} \le N\|u_k(\eta_j)_t\|_{L_{p_1}(\bR \times \Omega)} \\
&+\quad N \sum_{|\alpha| \le m, |\beta| \le m}\sum_{1\le |\tilde \beta| \le |\beta|} \lambda^{\frac{|\alpha|}{2m}}
\|D^{\tilde \beta}\eta_j D^{\beta-\tilde\beta}u_k\|_{L_{p_1}(\bR \times \Omega)},
\end{align*}
which implies that
\begin{align*}
&\sum_{|\alpha| \le m} \lambda^{1-\frac{|\alpha|}{2m}} \| D^\alpha u_k \|_{L_{p_1}(\bR \times \Omega \setminus \cQ_{2^{j+1}R})} \le N2^{-j}\|u_k\|_{L_{p_1}((\cQ_{2^{j+1}R}\setminus \cQ_{2^{j}R})\cap (\bR\times \Omega))} \\
&+\quad N 2^{-j}\sum_{|\alpha| \le m, |\beta| \le m-1} \lambda^{\frac{|\alpha|}{2m}}
\|D^{\beta}u_k\|_{L_{p_1}((\cQ_{2^{j+1}R}\setminus \cQ_{2^{j}R})\cap (\bR\times \Omega))}.
\end{align*}
Thus we have
\begin{align*}
&\sum_{|\alpha| \le m} \lambda^{1-\frac{|\alpha|}{2m}} \| D^\alpha u_k \|_{L_{p_1}((\cQ_{2^{j+2}R}\setminus \cQ_{2^{j+1}R})\cap (\bR\times \Omega))} \\
&\le N2^{-j}
\sum_{|\alpha| \le m} \lambda^{1-\frac{|\alpha|}{2m}} \| D^\alpha u_k \|_{L_{p_1}((\cQ_{2^{j+1}R}\setminus \cQ_{2^{j}R})\cap (\bR\times \Omega))},
\end{align*}
where we also used the fact that $\lambda \ge 1$.
By induction, we obtain for each $j\ge 1$,
\begin{align}
                                    \label{eq4.19}
&\sum_{|\alpha| \le m} \lambda^{1-\frac{|\alpha|}{2m}} \| D^\alpha u_k \|_{L_{p_1}((\cQ_{2^{j+1}R}\setminus \cQ_{2^{j}R})
\cap (\bR\times \Omega))} \nonumber\\
&\le N2^{-\frac{j(j-1)} 2}
\sum_{|\alpha| \le m} \lambda^{1-\frac{|\alpha|}{2m}} \| D^\alpha u_k \|_{L_{p_1}(\cQ_{2R} \cap (\bR\times\Omega))}.
\end{align}
Finally, by H\"older's inequality, \eqref{eq2.41}, \eqref{eq2.42}, and \eqref{eq4.19}, we get for each $j\ge 1$,
\begin{align*}
&\sum_{|\alpha| \le m} \lambda^{1-\frac{|\alpha|}{2m}} \| D^\alpha u_k \|_{L_{p,q,w}((\cQ_{2^{j+1}R}\setminus \cQ_{2^{j}R})
\cap (\bR\times \Omega))} \\
&\le \|w_1\|_{L_{1+\varepsilon_1}(B_{2^{j+1}R} \cap \Omega_1)}^{\frac 1 p}\|w_2\|_{L_{1+\varepsilon_2}(\tilde Q_{2^{j+1}R} \cap (\bR\times \Omega_2))}^{\frac 1 q}\\
&\quad \cdot\sum_{|\alpha| \le m} \lambda^{1-\frac{|\alpha|}{2m}} \| D^\alpha u_k \|_{L_{p_1}(( \cQ_{2^{j+1}R}\setminus \cQ_{2^{j}R})
\cap (\bR\times \Omega))} \\
&\le NN_0^{j(\frac 1 p+\frac 1 q)}2^{-\frac{j(j-1)} 2}
\sum_{|\alpha| \le m} \lambda^{1-\frac{|\alpha|}{2m}} \| D^\alpha u_k \|_{L_{p_1}(\cQ_{2R} \cap (\bR\times\Omega))},
\end{align*}
where $B_{2^{j+1}R}$ is a ball in $\bR^{d_1}$ as in \eqref{eq2.41}.
The above inequality together with \eqref{eq4.27} implies that
\begin{equation}
							\label{eq1220_01}
D^\alpha u_k \in L_{p,q,w}(\bR \times \Omega), \quad 0 \le |\alpha| \le m.
\end{equation}
Then the claim,  i.e., $u_k \in \mathring{\cH}_{p,q,w}^m(\bR \times \Omega)$, follows from \eqref{eq1220_01}, the equation for $u_k$, and  the fact that $u_k \in \mathring{\cH}_{p_1}^m(\bR \times \Omega)$.

\appendix
\section{}							\label{appendix01}

\begin{proof}[Proof of Theorem \ref{thm0623}]
We follow the proof of Theorem 1.4 in \cite{MR2797562} in the setting of spaces of homogeneous type.

For the given $w \in A_p$, define
$$
\cR h(x) = \sum_{k=0}^\infty \frac{\cM^k h(x)}{2^k N_1^k}, \quad
h \in L_p(w \, d\mu),
$$
where, for $k \ge 1$, $\cM^k = \cM \circ \cdots \circ \cM$ is the $k$-th iteration of the maximal operator, $\cM^0 h = |h|$, and $N_1$ is the constant $N$ from Theorem \ref{thm1} determined by $K_1$, $K_2$, $p$, and $[w]_{A_p}$.
The operator $\cR$ has the following properties:
\begin{enumerate}
\item for all $x \in \cX$, $|h(x)| \le \cR h (x)$;

\item $\| \cR h\|_{L_p(w \, d\mu)} \le 2 \| h \|_{L_p(w \, d\mu)}$;

\item $\cR h \in A_1$, i.e., there exists a constant $C$ such that $\cM \left(\cR h\right)(x) \le C \cR h(x)$ for almost every $x \in \cX$.
The infimum of all such $C$, denoted by $[\cR h]_{A_1}$, satisfies
$$
[\cR h]_{A_1} \le 2 N_1.
$$
\end{enumerate}

Now we define the operator $\cM' h = \cM( h w)/w$.
Since $w^{1-p'} \in A_{p'}$, where $1/p + 1/p' = 1$, $\cM$ is bounded on $L_{p'}(w^{1-p'} \, d\mu)$,
thus $\cM'$ is bounded on $L_{p'}(w \, d\mu)$.
Note that
$$
[w^{1-p'}]_{A_{p'}} = \left( [w]_{A_p} \right)^{1/(p-1)},
$$
and, for $h \in L_{p'}(w \, d\mu)$,
\begin{align*}
&\|\cM' h\|_{L_{p'}( w \, d\mu)}
= \| \cM (hw) \|_{L_{p'}( w^{1-p'} \, d\mu)}\\
&\le N_2 \| h w \|_{L_{p'}( w^{1-p'} \, d\mu)}
= N_2 \|h\|_{L_{p'}( w \, d\mu)},
\end{align*}
where $N_2$ is the constant $N$ from Theorem \ref{thm1} determined by $K_1$, $K_2$, $p'$, and $[w^{1-p'}]_{A_{p'}}$.
Define
$$
\cR' h(x) = \sum_{k=0}^\infty \frac{(\cM')^k h(x)}{2^k N_2^k},
\quad
h \in L_{p'}(w \, d\mu),
$$
where, again, $(\cM')^0 h = |h|$.
As above, the operator $\cR'$ has the following properties:
\begin{enumerate}
\item for all $x \in \cX$, $|h(x)| \le \cR' h (x)$;

\item $\| \cR' h\|_{L_{p'}(w \, d\mu)} \le 2 \| h \|_{L_{p'}(w \, d\mu)}$;

\item $\cM' (\cR'h)(x) \le 2 N_2 \cR' h(x)$.
Thus $(\cR' h) w \in A_1$ with $[(\cR' h) w]_{A_1} \le 2 N_2$.
\end{enumerate}

Set
$$
\Lambda_0 = 2^{p_0} N_1^{p_0-1}N_2.
$$
Using this $\Lambda_0$ and the assumption that \eqref{eq0623_01} holds for all $\tilde{w} \in A_{p_0}$ satisfying $[\tilde{w}]_{A_{p_0}} \le \Lambda_0$,
we derive \eqref{eq0623_02}.
Assume that $\|g\|_{L_p(w \, d\mu)} < \infty$.
Otherwise, there is nothing to prove.
Let $h \in L_{p'}(w \, d\mu)$.
Since $g \in L_p(w \, d\mu)$ and $h \in L_{p'}(w \, d\mu)$, by the properties of $\cR$ and $\cR'$ above, we have
$$
\cR g \in A_1, \quad [\cR g]_{A_1} \le 2 N_1,
\quad
(\cR' h) w \in A_1, \quad [(\cR' h) w]_{A_1} \le 2 N_2.
$$
Denote
$$
\tilde{w} = (\cR g)^{1-p_0} (\cR' h)w.
$$
Then, for instance, by Proposition 1.2 in \cite{MR0740173}, we have $\tilde{w} \in A_{p_0}$ and
$$
[\tilde{w}]_{A_{p_0}} \le [\cR g]_{A_1}^{p_0-1} [(\cR' h) w]_{A_1} \le 2^{p_0} N_1^{p_0-1} N_2 = \Lambda_0.
$$
Note that, from the properties of $\cR$ and $\cR'$,
\begin{align}
							\label{eq0624_01}
\int_\cX (\cR g) (\cR' h) w \, d\mu
&\le \left(\int_\cX |\cR g|^p w \, d\mu\right)^{1/p} \left( \int_\cX |\cR'h|^{p'} w \, d \mu \right)^{1/p'}\nonumber
\\
&\le 4 \|g\|_{L_p(w \, d\mu)} \|h\|_{L_{p'}(w \, d\mu)},
\end{align}
and by \eqref{eq0623_01}
\begin{align*}
&\left(\int_\cX |f|^{p_0} (\cR g)^{1-p_0} (\cR'h) w \, d\mu \right)^{1/p_0}
= \left(\int_\cX |f|^{p_0} \tilde{w} \, d\mu \right)^{1/p_0}\\
&\le N_0 \left(\int_\cX |g|^{p_0} \tilde{w} \, d\mu \right)^{1/p_0}
= N_0 \left(\int_\cX |g|^{p_0} (\cR g)^{1-p_0} (\cR'h) w \, d\mu \right)^{1/p_0}.
\end{align*}
From this, the inequalities $|h| \le \cR' h$, $|g| \le \cR g$, H\"older's inequality, and \eqref{eq0624_01}, we have
\begin{align*}
&\int_\cX |f| h w \, d \mu
\le \int_\cX |f| (\cR g)^{-1/p_0'} (\cR g)^{1/p_0'} (\cR' h) w \, d \mu\\
&\le \left( \int_\cX |f|^{p_0} (\cR g)^{1-p_0} (\cR'h) w \, d\mu \right)^{1/p_0}
\left(\int_\cX (\cR g) (\cR' h) w \, d\mu \right)^{1/p_0'}\\
&\le N_0 \left(\int_\cX |g|^{p_0} (\cR g)^{1-p_0} (\cR'h) w \, d\mu \right)^{1/p_0} \left(\int_\cX (\cR g) (\cR' h) w \, d\mu \right)^{1/p_0'}\\
&\le N_0 \left(\int_\cX (\cR g) (\cR'h) w \, d\mu \right)^{1/p_0} \left(\int_\cX (\cR g) (\cR' h) w \, d\mu \right)^{1/p_0'}\\
&= N_0 \int_\cX (\cR g) (\cR'h) w \, d\mu
\le 4 N_0 \|g\|_{L_p(w \, d\mu)} \|h\|_{L_{p'}(w \, d\mu)}.
\end{align*}
Since $h \in L_{p'}(w \, d\mu)$ is arbitrary and, by the definition of $A_p$ weights, $w \, d\mu$ is a $\sigma$-finite measure of $\cX$, from the above inequality we obtain $\|f\|_{L_p(w \, d\mu)} < \infty$ and the inequality \eqref{eq0623_02}.
The second statement is clear because the constants $N_1$ and $N_2$ from Theorem \ref{thm1} can be chosen depending only on the upper bound of $[w]_{A_p}$.
\end{proof}

Since the proof of Lemma 6.5 follows that of \cite[Proposition 7.10]{MR2835999}, the following lemma is proved by combining $L_q$-versions of the claims in Lemmas 7.12 and 7.13 of \cite{MR2835999}.
However, the integration by parts argument in the proof of  \cite[Lemma 7.12]{MR2835999} for $q=2$ does not work for general $q \in (1,\infty)$, so instead we use a duality argument (also used in \cite{MR3403998}), the details of which are given below.

\begin{lemma}
							\label{lem0714_1}
Let $\hat w$ be the unique $\mathring\cH^{m}_q(\bR\times \{x \in \bR^d : x_1>\gamma R\})$ solution of \eqref{eq0714_03} in $\bR\times \{x:x_1>\gamma R\}$,
Then we have \eqref{eq21.52h}.
\end{lemma}

\begin{proof}
Set $\tilde\Omega = \{x \in \bR^d: x_1 > \gamma R\}$ and
$$
\cL_1 v = \sum_{|\alpha|=|\beta|=m} D^\beta \left(\bar{A}^{\alpha\beta} D^\alpha v\right),
$$
where $\bar{A}^{\alpha\beta}$ are from \eqref{eq0715_01}.
By using the results in \cite{MR2835999}, in particular, by replacing $t$ by $-t$ in Theorem 2.4 (iii) there,
for $h_\alpha \in C_0^\infty(\bR \times\tilde\Omega)$, $0 \le |\alpha| \le m$, find a unique solution $v \in \mathring{\cH}_{q'} ( \bR \times \tilde\Omega )$, $1/q + 1/q' = 1$, to the system
\begin{equation}
							\label{eq0715_02}
- v_t + (-1)^m \cL_1 v + \lambda v = (-1)^{|\alpha|} D^\alpha h_\alpha
\end{equation}
in $\bR \times \tilde\Omega$ satisfying
\begin{equation}
							\label{eq0715_04}
\sum_{|\alpha| \le m}\lambda^{1-\frac{|\alpha|}{2m}} \| D^\alpha v\|_{L_{q'}(\bR \times \tilde\Omega)} \le N \sum_{|\alpha| \le m} \lambda^{\frac{|\alpha|}{2m}} \|h_\alpha\|_{L_{q'}(\bR \times \tilde\Omega)}.
\end{equation}
Since an approximation is available, if it is convenient, we may assume that $v$ is infinitely differentiable and has a compact support.
From \eqref{eq0715_02} and \eqref{eq0714_03} (apply $v$ as a test function to \eqref{eq0714_03}),
we see that
\begin{align*}
&\sum_{|\alpha| \le m} \int_{\bR \times \tilde\Omega} h_\alpha D^\alpha \hat{w} \, dx \, dt =
\int_{\bR \times \tilde\Omega} \left(- \hat{w} v_t + \bar{A}^{\alpha\beta} D^\beta \hat{w} D^\alpha v + \lambda \hat{w} v\right) \, dx \, dt\\
&= \sum_{|\alpha|=|\beta|=m}\int_{\bR \times \tilde\Omega}
\varphi ( \bar{A}^{\alpha\beta} - A^{\alpha\beta}) D^\beta u D^\alpha v \, dx \, dt\\
&\quad + \sum_{|\alpha| \le m} (-1)^{|\alpha|}\int_{\bR \times \tilde\Omega} \varphi f_\alpha D^\alpha(\chi v) \, dx \, dt\\
&\quad + \sum_{|\alpha|=|\beta|=m} \int_{\bR \times \tilde\Omega}
 \bar{A}^{\alpha\beta} \varphi D^\beta\left((\chi-1)u\right) D^\alpha v \, dx \, dt\\
&\quad + \sum_{|\alpha|=|\beta|=m}\int_{\bR \times \tilde\Omega} I_{x_1<2\gamma R}A^{\alpha\beta} \varphi D^\beta u D^\alpha \left( (1-\chi) v \right) \, dx \, dt,
\end{align*}
where in the last integral we used the fact that $1-\chi = 0$ if $x_1 \ge 2\gamma R$.
The above equalities with H\"{o}lder's inequality show that
\begin{multline}
							\label{eq0715_03}
\sum_{|\alpha| \le m} \int_{\bR \times \tilde\Omega} h_\alpha D^\alpha \hat{w} \, dx \, dt \le N \| \varphi(\bar{A}^{\alpha\beta} - A^{\alpha\beta}) D^m u\|_{L_q} \|D^m v\|_{L_{q'}}
\\
+ \sum_{|\alpha| \le m} \| \varphi f_\alpha  \|_{L_q} \| D^\alpha (\chi v) \|_{L_{q'}}
+ N \| \varphi D^m \left( (\chi-1) u \right) \|_{L_q} \|D^m v\|_{L_{q'}}
\\
+ N \|I_{x_1<2\gamma R} \, \varphi D^m u\|_{L_q} \| D^m \left( (1 - \chi) v \right) \|_{L_{q'}},
\end{multline}
where $L_q = L_q(\bR \times \tilde\Omega)$ and $L_{q'} = L_{q'}(\bR \times \tilde\Omega)$.
Notice that
$$
|D^k \chi| \le N (x_1 - \gamma R)^{-k}
$$
for $k = 1, 2, \ldots, m$.
Thus by Lemma 7.9 in \cite{MR2835999} we have
$$
\|D^\alpha (\chi v) \|_{L_{q'}} \le N \| D^\alpha v\|_{L_{q'}},
\quad
|\alpha| \le m.
$$
Since $h_\alpha \in C_0^\infty (\bR \times\tilde\Omega)$ are arbitrary, this inequality together with \eqref{eq0715_03} and \eqref{eq0715_04} implies that
\begin{multline}
							\label{eq0715_05}
\sum_{|\alpha| \le m}\lambda^{\frac{1}{2} - \frac{|\alpha|}{2m}} \| D^\alpha \hat{w}\|_{L_q}
\le N \| \varphi(\bar{A}^{\alpha\beta} - A^{\alpha\beta}) D^m u \|_{L_q}
+ N \sum_{|\alpha| \le m} \lambda^{\frac{|\alpha|}{2m} - \frac{1}{2}} \| \varphi f_\alpha \|_{L_q}
\\
+ N \|\varphi D^m\left( (\chi - 1) u \right) \|_{L_q}
+ N \|I_{x_1<2\gamma R} \, \varphi D^m u\|_{L_q},
\end{multline}
where $N=N(d,m,\ell,\delta,q)$.

Now we observe that by H\"{o}lder's inequality and Assumption \ref{assum0625_1} (recall that $\varphi = I_{Q_R}$)
\begin{align*}
\| \varphi(\bar{A}^{\alpha\beta} - A^{\alpha\beta}) D^m u \|_{L_q}
&\le \left(\int_{\cC_R} |\bar{A}^{\alpha\beta} - A^{\alpha\beta}|^q |D^m u|^q \, dx \, dt\right)^{1/q}\\
&\le N R^{\frac{d+2m}{q}} \gamma^{\frac{1}{q\nu}} \left( |U|^{q\nu} \right)_{\cC_R}^{\frac{1}{q\mu}}.
\end{align*}
By Lemma 7.9 in \cite{MR2835999} and H\"{o}lder's inequality (see the proof of \cite[Lemma 7.13]{MR2835999})
\begin{align*}
&\|\varphi D^m ((\chi-1) u) \|_{L_q}\\
&\le \left( \int_{\cC_R} I_{-\gamma R < x_1 < 2 \gamma R} |D^m u|^q \, dx \, dt \right)^{1/q}
\le N  R^{\frac{d+2m}{q}} \gamma^{\frac{1}{q\nu}} \left( |U|^{q\mu} \right)_{\cC_R}^{\frac{1}{q\mu}}.
\end{align*}
Upon treating the other terms in the right-hand side of \eqref{eq0715_05} similarly and bounding the right-hand side of \eqref{eq21.52h} by the left-hand side of \eqref{eq0715_05}, we finally obtain \eqref{eq21.52h}.
\end{proof}

\section*{Acknowledgments}
The authors would like to thank the referees for their careful review as well as many valuable comments and suggestions, and for bringing our attention to the references \cite{MR2867756, MR2901199, MR2561181, MR3231530}.

\bibliographystyle{plain}

\def\cprime{$'$}

\end{document}